\documentclass[11pt,reqno]{article}

\usepackage[leqno]{amsmath}
\usepackage{amssymb,xfrac}
\usepackage{amsthm}
\usepackage{ifthen} 
\usepackage[pagebackref,breaklinks,unicode]{hyperref} 
    \renewcommand*{\backrefalt}[4]{\ifcase #1 (Not cited).\or (Cited p.~#2).\else (Cited pp.~#2).\fi} 
\usepackage[capitalise,noabbrev]{cleveref}
\usepackage{enumitem}
\usepackage{graphicx}
\usepackage{tikz-cd, tikz}
\usepackage{float}
\usepackage[justification=centering]{caption}
\usepackage{subcaption}
\usepackage{wasysym}

\numberwithin{equation}{section}

\newtheorem{theorem}{Theorem}[section]
\newtheorem{proposition}[theorem]{Proposition}
\newtheorem{lemma}[theorem]{Lemma}
\newtheorem{corollary}[theorem]{Corollary}

\newtheorem*{theorem*}{Theorem}
\newtheorem*{proposition*}{Proposition}
\newtheorem*{lemma*}{Lemma}
\newtheorem*{corollary*}{Corollary}

\newtheorem{mthm}{Theorem} 
\newtheorem{mcor}[mthm]{Corollary}

\theoremstyle{definition}
\newtheorem{definition}[theorem]{Definition}
\newtheorem{observation}[theorem]{Observation}
\newtheorem{remark}[theorem]{Remark}
\newtheorem{example}[theorem]{Example}

\newtheorem*{definition*}{Definition}
\newtheorem*{observation*}{Observation}
\newtheorem*{remark*}{Remark}
\newtheorem*{example*}{Example}
\newtheorem*{question*}{Question}
\newtheorem*{exercise*}{Exercise}
\newtheorem*{fact*}{Fact}
\newtheorem*{notation*}{Notation}

\newcommand{\bbR}{\mathbb{R}}

\newcommand{\Z}{\mathbb{Z}}

\newcommand{\HH}{\mathcal{H}}

\DeclareMathOperator{\Hull}{Hull}

\newcounter{claimref}
\newcounter{claimcount}

\newenvironment{claim*}{\par\medskip\noindent\textbf{Claim:}\hspace{0.5mm}}{}

\newenvironment{claim*proof}{\medskip\noindent\emph{Proof of Claim.}\hspace{0.5mm}}
    {\leavevmode\unskip\penalty9999\hbox{}\nobreak\hfill\quad\hbox{$\diamondsuit$}\medskip}

\newcounter{subclaimcount}

\newenvironment{subclaim*}{\par\medskip\textbf{Subclaim:}\hspace{0.5mm}}{}
\newenvironment{subclaimproof*}{\medskip\noindent\emph{Proof of Subclaim.}\hspace{0.5mm}}
    {\leavevmode\unskip\penalty9999\hbox{}\nobreak\hfill\quad\hbox{$\fullmoon$}\medskip}

\crefname{cond}{condition}{conditions}
\creflabelformat{cond}{#2#1\@#3}
\crefname{obs}{observation}{observations}
\creflabelformat{obs}{#2#1\@#3}

\renewcommand{\cal}{\mathcal}
\newcommand*{\eps}{\varepsilon}

\newcommand*{\R}{\mathbb{R}}

\newcommand*{\dist}{d}

\newcommand*{\sgen}[1]{\langle#1\rangle}

\renewcommand*{\subset}{\subseteq}

\newcounter{shcount}
\newcounter{thmcount}

\swapnumbers
 
\newcommand*{\numberedtheorem}[3]{\theoremstyle{plain}\newtheorem*{makethm\thethmcount}{#1}
    \ifthenelse{\equal{#2}{}}{\begin{makethm\thethmcount}#3\end{makethm\thethmcount}\stepcounter{thmcount}}
    {\begin{makethm\thethmcount}[#2]#3\end{makethm\thethmcount}\stepcounter{thmcount}}} 

\makeatletter\newcommand{\linkdest}[1]{\Hy@raisedlink{\hypertarget{#1}{}}}\makeatother 

\usepackage{lmodern,anyfontsize} 
\usepackage{setspace} 
\usepackage{tocloft} 
\usepackage[nottoc]{tocbibind}

\usepackage{amsmath,amssymb,xfrac}
    \usepackage{enumitem} \setlist{nosep}
\usepackage{ifthen} 
\usepackage[pagebackref,breaklinks,unicode]{hyperref} 
    \renewcommand*{\backrefalt}[4]{\ifcase #1 (Not cited).\or (Cited p.~#2).\else (Cited pp.~#2).\fi} 
\usepackage{graphicx} 
\usepackage{mathrsfs}
\usepackage[T1]{fontenc}
\usepackage[margin=3cm]{geometry}
\usepackage{natbib}

\newcommand*{\ext}{\mathrm{ext}}
\newcommand*{\N}{\mathbb{N}}

\makeatletter
\newtheorem*{rep@theorem}{\rep@title}
\newcommand{\newreptheorem}[2]{%
\newenvironment{rep#1}[1]{%
 \def\rep@title{#2 \ref{##1}}%
 \begin{rep@theorem}}%
 {\end{rep@theorem}}}
\makeatother

\newreptheorem{theorem}{Theorem}
\newreptheorem{example}{Example}
\newreptheorem{fact}{Fact}

\newcommand*{\ssm}{\smallsetminus}

\definecolor{harrycomment}{rgb}{0.6,0,0.4}

\definecolor{shakedcomment}{rgb}{0, 0, 255}

\definecolor{oussamacomment}{rgb}{0,0.5,0}

\title{From branching quasiflats to flats in CAT(0) cube complexes}
\hypersetup{pdftitle = From branching quasiflats to flats}
\date{}
\author{Shaked Bader, Oussama Bensaid, and Harry Petyt}

\newcommand{\Addresses}{{\bigskip\footnotesize\par
\textsc{Mathematical Institute, University of Oxford, UK}\par\nopagebreak\textit{E-mail address}: 
\texttt{shaked.bader@sjc.ox.ac.uk}
\par\medskip\par
\textsc{Institut de Recherche en Mathématique et en Physique, Université Catholique de Louvain, Belgium} \par\nopagebreak\textit{E-mail address}: 
\texttt{oussama.bensaid@uclouvain.be}
\par\medskip\par
\textsc{Mathematics Institute, University of Warwick, UK}\par\nopagebreak\textit{E-mail address}: 
\texttt{harrypetyt@gmail.com}
}}

\begin{document}

\maketitle

\begin{abstract}
We study quasiisometric embeddings between finite-dimensional CAT(0) cube complexes. More specifically, we introduce geometric branching conditions under which flats in the domain, not necessarily of top rank, are mapped within finite Hausdorff distance of flats. As a consequence, one obtains embeddings between natural graphs associated with the Tits boundaries of those cube complexes. 

These results form a key step in understanding quasiisometric embeddings between right-angled Artin groups.
In an appendix, we also explain how the same methods recover previously established rigidity results for quasiisometric embeddings of symmetric spaces and Euclidean buildings of the same spherical type.
\end{abstract}

{\hypersetup{hidelinks}\setcounter{tocdepth}{1}\tableofcontents\setcounter{tocdepth}{2}}

\section{Introduction} 

Quasiisometric rigidity is one of the central themes of geometric group theory and coarse geometry. Broadly speaking, it asks to what extent groups or spaces can be determined by their large-scale geometry.
This guiding question has been very influential, and there is a rich body of work showing that in many important settings, quasiisometric groups or spaces turn out to be essentially the same. 

As a natural extension of this, one can ask when one group or space can be quasiisometrically embedded in another. This is especially relevant because in most cases, results on quasiisometric rigidity rely in some way on understanding quasiisometrically embedded flats. The goal of this paper is to provide geometric conditions under which a quasiisometric embedding between CAT(0) cube complexes sends flats within finite Hausdorff distance of flats. This is more difficult than in the setting of quasiisometric rigidity, because we do not have access to a quasiinverse.

\subsection{History and motivation}

The study of groups from the large-scale geometric perspective goes back at least as far as work of Milnor and Wolf on growth in groups \cite{milnor:note,wolf:growth,milnor:growth}, and really began to come into focus with important developments such as Stallings' theorem on ends of groups \cite{stallings:ontorsionfree}, Mostow's rigidity theorem \cite{mostow:quasiconformal,marden:geometry,prasad:strong}, and Gromov's polynomial growth theorem \cite{gromov:groups}. 

As a consequence of Gromov's theorem and the Bass--Guivarc'h formula \cite{bass:degree,guivarch:croissance}, any group quasiisometric to a finitely generated abelian group is itself virtually abelian. Subsequently, Gromov proposed the program of trying to classify finitely generated groups up to quasiisometry \cite{gromov:asymptotic}. The body of work in this direction is too large for us to encompass here, so let us give just a sample of some of it.

Generalising the case of Euclidean spaces, two symmetric spaces of non-positive curvature are quasiisometric if and only if (after permuting and rescaling the irreducible de Rham factors) they are isometric \cite{kleinerleeb:rigidity}; see also \cite{tukia:homeomorphic,pansu:metriques,chow:groups}, and an analogous rigidity statement also holds for higher-rank Euclidean buildings. For mapping class groups, one likewise has strong quasiisometric rigidity: except in some low-complexity cases, every self-quasiisometry of a mapping class group is close to an isometry, and consequently any finitely generated group quasiisometric to a mapping class group is commensurable with it \cite{behrstockkleinerminskymosher:geometry}. Further examples include lattices in semisimple Lie groups other than $\mathrm{Isom}(\mathbb H^2)$, and solvable Baumslag--Solitar groups $\mathrm{BS}(1,n)$ \cite{schwartz:quasiisometry,eskin:quasiisometric,farbmosher:quasiisometric}.

In the cubical setting, which is the setting of the present paper, Bestvina--Kleiner--Sageev \cite{bestvinakleinersageev:asymptotic} and Huang \cite{huang:quasiisometric:1,huang:quasiisometric:2} proved quasiisometric rigidity within natural classes of right-angled Artin groups. For example, any two quasiisometric right-angled Artin groups with finite outer automorphism group are isomorphic \cite{huang:quasiisometric:1}. Even more, Huang found a class of right-angled Artin groups that are strongly quasiisometrically rigid \cite{huang:commensurability}. It should be noted, though, that there are right-angled Artin groups for which these results fail \cite{behrstockneumann:quasiisometric}.

In some particularly nice situations, rigidity phenomena of this kind persist when one replaces quasiisometries by quasiisometric embeddings. In higher-rank symmetric spaces, Fisher--Whyte proved that, under natural compatibility assumptions, quasiisometric embeddings remain rigid, forcing maximal flats to be mapped within finite Hausdorff distance of maximal flats \cite{fisherwhyte:quasiisometric}. They also constructed exotic embeddings when these assumptions fail. See also \cite{nguyen:quasiisometric,bensaidnguyen:embedding} for further developments, including Euclidean buildings. Analogously, Bowditch proved that, for compact orientable surfaces of the same complexity at least $4$, any quasiisometric embedding between their mapping class groups is at bounded distance from an isometry \cite{bowditch:large:mapping}. In particular, those surfaces are homeomorphic and the embedding is a quasiisometry. In another direction, Nairne proved that $\mathrm{BS}(1,m)$ quasiisometrically embeds in $\mathrm{BS}(1,n)$ if and only if the two are quasiisometric \cite{nairne:embeddings}. As yet there has been no work in this direction for right-angled Artin groups.

A common feature of many of these rigidity theorems is that they rely, in one form or another, on a higher-rank analogue of the Morse lemma. Recall that the Morse lemma states that every quasigeodesic in a $\delta$--hyperbolic metric space lies at finite Hausdorff distance from a geodesic. This fails in higher-rank spaces: consider a log-spiral in the Euclidean plane, for instance. Nevertheless, some form of rigidity can often be recovered in higher rank by replacing quasigeodesics with top-rank quasiflats. 

In symmetric spaces and Euclidean buildings, foundational results of Kleiner--Leeb \cite{kleinerleeb:rigidity} and later Eskin--Farb \cite{eskinfarb:quasiflats} show that top-rank quasiflats are Hausdorff-close to finite unions of Weyl sectors. In CAT(0) cube complexes, analogous theorems for top-dimensional quasiflats, with Weyl sectors replaced by cubical orthants, were obtained by Bestvina--Kleiner--Sageev \cite{bestvinakleinersageev:quasiflats} in the cocompact 2--dimensional case, and by Huang \cite{huang:top} in general. A finer statement for top-rank biLipschitz flats in median metric spaces was also obtained by Bowditch \cite{bowditch:large:mapping}. Similar quasiflat theorems with coarser objects in the conclusion were later obtained by Behrstock--Hagen--Sisto and Bowditch in greater generality \cite{behrstockhagensisto:quasiflats,bowditch:quasiflats}.

Our goal in this paper is to find geometric conditions under which we can upgrade these quasiflats theorems for CAT(0) cube complexes. We wish to upgrade them in three ways. First of all, we want to obtain a genuine flat in the codomain, not just a union of orthants. Secondly, we would like to consider CAT(0) cube complexes whose dimension is strictly greater than the maximal dimension of a flat, for improved stability properties under quasiisometries. And finally, we want to obtain results for flats that are not necessarily top-rank. These refinements form the basis for results on quasiisometric embeddings between right-angled Artin groups in \cite{baderbensaidpetyt:quasiisometric:rigidity,baderbensaidpetyt:quasiisometric:flexibility,baderbensaidpetyt:quasiisometric:products}. In particular, having control over lower-rank flats is important for obtaining \emph{obstructions} to the existence of quasiisometric embeddings.

In the setting of quasiisometries, the first of these refinements can often be deduced from the existence of a quasiinverse. Indeed, Huang proved that if $f\colon X\to Y$ is a quasiisometry between universal covers of compact weakly special cube complexes of dimension $n$, then every $n$--flat in $X$ is mapped within uniformly bounded Hausdorff distance of an $n$--flat in $Y$ \cite[Thms~1.3,~5.4]{huang:top}. His argument, however, relies in an essential way on the existence of a quasiinverse, and thus does not generalise to the setting of quasiisometric embeddings. The necessity of additional assumptions for proving a quasiflats theorem for lower-dimensional flats can already be seen in the log-spiral example.


Our main input is a collection of branching conditions on geodesics and flats, designed to capture the presence of enough transverse flat structure to force rigidity. While our results are formulated for general CAT(0) cube complexes, the guiding examples come from universal covers of Salvetti complexes of right-angled Artin groups, where these branching conditions can be read directly from the defining graph, making them especially concrete and useful in practice.

The idea that branching can help control quasiflats already appears in \cite{kleinerleeb:rigidity}, where the fact that in a thick Euclidean building each maximal flat has ``rank-plus-one directions'' of branching is the key point in concluding that quasiisometries are close to rescalings of isometries. It is also a vital part of \cite{behrstockkleinerminskymosher:geometry}. Indeed, the authors show that the strong branching exhibited by Dehn-twist flats characterises them (topologically, on the level of the asymptotic cone) among quasiflats, which, together with the quasiinverse, allows them to induce an automorphism of the curve graph from the given quasiisometry and invoke Ivanov's theorem \cite{ivanov:automorphisms}. This idea was extended into the hierarchically hyperbolic setting in an axiomatic way in \cite{behrstockhagensisto:quasiflats,mangioni:random} to obtain quasiisometric rigidity for larger classes of groups. Additionally, Bowditch used the branching of Dehn-twist flats to obtain his result on rigidity of quasiisometric embeddings between mapping class groups of the same complexity \cite{bowditch:large:mapping}. Part of the novelty in this article is that the branching conditions we require are significantly weaker than in previous works; see the discussion in \cite[Rem.~3.4]{baderbensaidpetyt:quasiisometric:rigidity}, for example.

\subsection{Main results}

As described above, our aim is to provide geometric branching conditions under which a quasiisometric embedding of CAT(0) cube complexes sends flats (not necessarily of top rank) within finite Hausdorff distance of flats. 

In the more general setting, where the domain is a complete CAT(0) space, we can control the images of sufficiently branching top-rank flats.
The following is a version of \cref{thm:fully_branching_general} that fulfils this brief. We shall give the relevant definitions bellow.

\begin{mthm} \label{mthm:sketch}
Let $X$ be a complete CAT(0) space, $Y$ be a finite-dimensional CAT(0) cube complex, and $f\colon X\to Y$ be a quasiisometric embedding. If $X$ and $Y$ have the same asymptotic rank, $n$, then $f(F)$ lies at finite Hausdorff distance from an $n$--flat of $Y$ for every directionally branch-complemented $n$--flat $F\subset X$.
\end{mthm}

\emph{Asymptotic rank}, introduced by Wenger \cite{wenger:asymptotic}, is a large-scale notion of dimension that is preserved by quasiisometries. The asymptotic rank of a CAT(0) cube complex $X$ is the largest $n$ such that one can find arbitrarily large $n$--dimensional boxes inside $X$; see \cref{lem:rank_cone_CCC}. For example, every hyperbolic CAT(0) cube complex has asymptotic rank at most one, even though its dimension could be arbitrarily large. In this case, \cref{mthm:sketch} is just the Morse lemma, but it serves to illustrate the significant increase in generality compared to taking $X$ and $Y$ to have dimension $n$.

The key notion in \cref{mthm:sketch} is that of a \emph{directionally branch-complemented} flat, which we introduce for general CAT(0) spaces in \cref{def:fully_branching_flats}. 

\begin{definition} \label{def:intro_branching}
Let $X$ be a CAT(0) space of asymptotic rank $n$. A flat $H\subset X$ is called \emph{branching} if it is an intersection of finitely many $n$--flats. 

A geodesic $\gamma\subset X$ is \emph{branch-complemented} if there is an $n$--flat $F\supset\gamma$ and an $(n-1)$--flat $H\subset F$ transverse to $\gamma$ such that $F$ is coarsely covered by branching parallels of $\gamma$ and also coarsely covered by branching parallels of $H$.

A flat $E\subset X$ is \emph{directionally branch-complemented} if it is spanned by geodesics such that the branch-complemented parallels of each one coarsely cover $E$.
\end{definition}

In the case of standard flats in universal covers of Salvetti complexes, these notions can be read directly from the defining graph $\Gamma$. More precisely, a standard geodesic corresponding to a vertex $v\in\Gamma$ is branching precisely when $v$ is an intersection of top-dimensional cliques. It is branch-complemented if moreover $v$ belongs to a top-dimensional clique $K$ such that the codimension-one face $K \ssm \{v\}$ is also an intersection of top-dimensional cliques. A standard flat is directionally branch-complemented precisely when each vertex of the corresponding clique defines a branch-complemented standard geodesic. For instance, if $\Gamma$ is triangle-free, then the standard geodesic corresponding to $v$ is branching as long as $v$ is not a leaf, and it is branch-complemented if $v$ also has a non-leaf neighbour. The special case of right-angled Artin groups is developed further in \cite{baderbensaidpetyt:quasiisometric:rigidity} and \cite{baderbensaidpetyt:quasiisometric:products}.

\begin{figure}[ht]
    \centering
    \begin{tikzpicture}[baseline=-.1cm, scale=1.3]
\fill[fill=red!30]
(-.4,-.35) to (.4,-.35) to (0,.35);
\draw[thick]
(-.4,-.35) to (.4,-.35) to (0,.35) to (-.4,-.35);
\draw[thick]
(0,.35) to (-.2,.7) to (.2,.7) to (0,.35)
(-.2,.7) to (.2,.7) to (0,1.05) to (-.2,.7)
(-.4,-.35) to (-.8,-.35) to (-.6,-.7) to (-.4,-.35)
(-.8,-.35) to (-.6,-.7) to (-1,-.7) to (-.8,-.35)
(.4,-.35) to (.8,-.35) to (.6,-.7) to (.4,-.35)
(.8,-.35) to (.6,-.7) to (1,-.7) to (.8,-.35);
\fill[red] (-.4,-.35) circle(.08);
\fill[red] (.4,-.35) circle(.08);
\fill[red] (0,.35) circle(.08);
\fill (-.2,.7) circle(.08);
\fill (.2,.7) circle(.08);
\fill (0,1.05) circle(.08);
\fill (-.8,-.35) circle(.08);
\fill (-.6,-.7) circle(.08);
\fill (-1,-.7) circle(.08);
\fill (.8,-.35) circle(.08);
\fill (.6,-.7) circle(.08);
\fill (1,-.7) circle(.08);
\end{tikzpicture}
\qquad
\begin{tikzpicture}[baseline=-.1cm, scale=1.3]
\fill[fill=red!30]
(-.4,-.35) to (0,.35) to (.4,-.35);
\draw[thick]
(-.8,.35) -- (.8,.35)
(-1.2,-.35) -- (1.2,-.35);
\draw[thick]
(-1.2,-.35) to (-.8,.35) to (-.4,-.35) to (0,.35) to (.4,-.35) to
(.8,.35) to (1.2,-.35);
\fill (-1.2,-.35) circle(.08);
\fill (-.8,.35) circle(.08);
\fill[red] (-.4,-.35) circle(.08);
\fill[red] (0,.35) circle(.08);
\fill[red] (.4,-.35) circle(.08);
\fill (.8,.35) circle(.08);
\fill (1.2,-.35) circle(.08);
\end{tikzpicture}
\qquad
\begin{tikzpicture}[baseline=-.1cm, scale=1.25]
\fill[red!30]
(-.4,.35) -- (0,-.35) -- (.4,.35) -- cycle;

\draw[thick]
(-.8,-.35) -- (.8,-.35)
(-.4,.35) -- (.4,.35);

\draw[thick]
(-.8,-.35) -- (-.4,.35) -- (0,-.35) -- (.4,.35) -- (.8,-.35)
(-.4,.35) -- (0,1.05) -- (.4,.35);

\fill (-.8,-.35) circle(.08);
\fill[red] (-.4,.35) circle(.08);
\fill[red] (0,-.35) circle(.08);
\fill[red] (.4,.35) circle(.08);
\fill (.8,-.35) circle(.08);
\fill (0,1.05) circle(.08);
\end{tikzpicture}
\caption{Examples in rank $3$: the red vertices define branch-complemented standard geodesics, and the red cliques define directionally branch-complemented standard flats.}
    \label{fig:Intro_exples_FB_geodesics_and_cliques}
\end{figure}
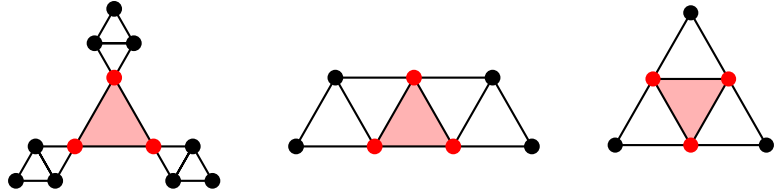
 
Another basic example is given by products of $3$--regular trees. In that case, every geodesic contained in a factor is branch-complemented, and every $k$--flat contained in a product of $k$ factors is directionally branch-complemented. In fact, this example has a much stronger intersection pattern than the one required in \Cref{def:intro_branching}: every such directionally branch-complemented flat is itself branching, as in the example on the right of \Cref{fig:Intro_exples_FB_geodesics_and_cliques}. By contrast, the two other examples show that a $3$--flat may be directionally branch-complemented even though none of its $2$--dimensional subflats is branching. Thus, our notion only imposes branching in dimension $1$ and codimension $1$, and not a full higher-dimensional intersection pattern.

In the case the domain is also a finite-dimensional CAT(0) cube complex, we can add to \cref{mthm:sketch} and get control over branching subflats of top-rank flats, as seen in the following statement, which is a combination of Theorems~\ref{thm:fully_branching_general}, \ref{thm:subflat_branching_in_fully_branching}, and \ref{thm:subflat_branching_asymptotic_rank}.

\begin{mthm} \label{mthm:subflats}
Let $X$ and $Y$ be finite-dimensional CAT(0) cube complexes of asymptotic rank $n$. For each $q$ there exists $D=D(q,Y)$ such that the following holds for every $q$--quasiisometric embedding $f\colon X\to Y$.

For each directionally branch-complemented $n$--flat $F\subset X$, the image $f(F)$ lies at Hausdorff distance at most $D$ from an $n$--flat of $Y$. Moreover, if $H$ is a branching subflat of $F$, then $f(H)$ is at finite Hausdorff distance from a semisingular flat $H'\subset Y$.

If $\dim X=\dim Y=n$, then $H'$ is singular and $D$ depends only on $q$.
\end{mthm}

This gives more precise information about the image of $F$ than \cref{mthm:sketch}. The control it gives on the images of top-rank flats and branching subflats is rather strong. Not only are the images close to genuine flats, those flats are \emph{semisingular}, or even \emph{singular}. Let us define these terms; see also \cref{sec:singular}.

CAT(0) cube complexes arise naturally in various settings, and classically have two natural metric structures: the CAT(0) metric and the median, or $\ell^1$ metric. In this paper we are working with the CAT(0) metric, so the ``flats'' mentioned above are isometrically embedded copies of $(\R^d,\ell^2)$ with respect to the CAT(0) metric. However, a particularly nice family of flats are those that are also isometric embeddings of $(\R^d,\ell^1)$ with respect to the median metric. These are precisely the \emph{singular} flats.

In an $n$--dimensional CAT(0) cube complex $X$, an $n$--flat is automatically singular (\cref{lem:top_dim_singular}), but in the generic case where the dimension of $X$ is greater than its asymptotic rank $n$, it can happen that $X$ has no singular $n$--flats, even when $n=1$: see \cref{eg:corner_to_corner}. We therefore introduce the more flexible notion of \emph{semisingularity}: a flat is semisingular if it becomes singular after passing to an asymptotic cone. See \cref{def:semisingular}.

Thus, \cref{mthm:subflats} does not merely produce flats, it produces flats that match well with the cubical structure on $Y$. We also highlight here that there is no properness assumption in \cref{mthm:subflats}. This contrasts with the general CAT(0) setting, where we do not know whether a complete CAT(0) space that is quasiisometric to a line necessarily contains a geodesic line. For CAT(0) cube complexes this statement is provided by \cref{prop:lines}. 



\medskip

All of the flats that can be understood using \cref{mthm:subflats} are either top-rank or are intersections of top-rank flats. Since the behaviour of lower-rank flats is in general much wilder than that of top-rank flats, if one wants to control maximal flats that are not top-rank, then one needs a stronger branching hypothesis. This leads to the following definition, requiring branch-complemented geodesics to arise as intersections of directionally branch-complemented top-rank flats; see \cref{def:strong_fully_branching}.

\begin{definition} \label{def:intro_strong_branching}
Let $X$ be a CAT(0) space of asymptotic rank $n$. A geodesic $\gamma\subset X$ is \emph{strongly branch-complemented} if it is an intersection of directionally branch-complemented $n$--flats.

A flat $H\subset X$ is \emph{directionally strongly branch-complemented} if it is spanned by geodesics such that the strongly branch-complemented parallels of each one coarsely cover $E$.
\end{definition}


As before, for standard geodesics in universal covers of Salvetti complexes these properties can be read directly from the defining graph. It is worth noting that directionally strongly branch-complemented flats of dimension greater than one need not be branching, and in fact they need not even be contained in any $n$--flat. See, for example, the standard flats associated to the three middle edges in Figure~\ref{fig:Intro_fully_branching_subgraphs_3D}.

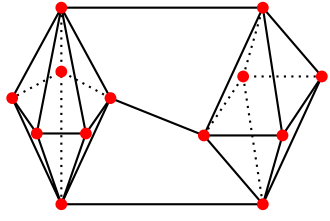
\begin{figure}[ht]
\centering
\begin{tikzpicture}[baseline=-.1cm, scale=1.3]

\coordinate (tL) at (-.85,1);
\coordinate (bL) at (-.85,-1);

\coordinate (v1) at (-1.35,.08);
\coordinate (v2) at (-1.10,-.28);
\coordinate (v3) at (-.60,-.28);
\coordinate (v4) at (-.35,.08);
\coordinate (v5) at (-.85,.35);

\coordinate (a1) at (.6,-.3);
\coordinate (a2) at (1.4,-.3);
\coordinate (a3) at (1.8,.3);
\coordinate (a4) at (1,.3);
\coordinate (tR) at (1.2,1);
\coordinate (bR) at (1.2,-1);

\draw[thick]
(v1) -- (v2) -- (v3) -- (v4);
\draw[thick,dotted]
(v4) -- (v5) -- (v1);

\draw[thick]
(tL) -- (v1) (tL) -- (v2) (tL) -- (v3) (tL) -- (v4)
(bL) -- (v1) (bL) -- (v2) (bL) -- (v3) (bL) -- (v4);
\draw[thick,dotted]
(tL) -- (v5)
(bL) -- (v5);

\draw[thick]
(a1) -- (a2) -- (a3);
\draw[thick,dotted]
(a1) -- (a4) -- (a3);

\draw[thick]
(a1) -- (tR)
(a2) -- (tR)
(a3) -- (tR)
(a1) -- (bR)
(a2) -- (bR)
(a3) -- (bR);
\draw[thick,dotted]
(a4) -- (tR)
(a4) -- (bR);

\draw[thick]
(tL) -- (tR)
(bL) -- (bR)

(v4) -- (a1);

\fill[red] (v1) circle(.06);
\fill[red] (v2) circle(.06);
\fill[red] (v3) circle(.06);
\fill[red] (v4) circle(.06);
\fill[red] (v5) circle(.06);
\fill[red] (tL) circle(.06);
\fill[red] (bL) circle(.06);

\fill[red] (a1) circle(.06);
\fill[red] (a2) circle(.06);
\fill[red] (a3) circle(.06);
\fill[red] (a4) circle(.06);
\fill[red] (tR) circle(.06);
\fill[red] (bR) circle(.06);
\end{tikzpicture}
\caption{In the universal cover of the Salvetti complex associated with this graph, all standard geodesics are strongly branch-complemented, and hence all standard flats are directionally strongly branch-complemented. Nevertheless, the standard $2$-flats associated to the three middle edges are not contained in any 3--flat.}
    \label{fig:Intro_fully_branching_subgraphs_3D}
\end{figure}

Note that even though \cref{def:intro_strong_branching} contains more information than \cref{def:intro_branching}, it is still the case that a flat being directionally strongly branch-complemented depends only on its one-dimensional subspaces. Such one-dimensional information is not enough for us to consider quasiisometries with the domain having lower asymptotic rank than dimension (see \cref{rem:no_asrk}), but we can still allow that in the codomain - see \cref{thm:strong_branching_asymptotic_rank}. Our results are strongest in the 
%
%
case where the dimensions of $X$ and $Y$ agree with their asymptotic rank, and there we can even control Hausdorff distances. For this reason, and to simplify the discussion below, we now restrict to that setting. The following is a combination of \cref{thm:strong_fully_branching} and \cref{cor:embedding_orthant_strong_fullybranching}.

Recall that an \emph{orthant} is a subspace isometric to $[0,\infty)^k$. It should be noted that it is not automatic for a quasiisometry that sends flats to flats to also send orthants to orthants. See \cref{rem:bad_orthants}, for instance.

\begin{mthm} \label{mthm:strong_ndim}
Let $X$ and $Y$ be $n$--dimensional CAT(0) cube complexes of asymptotic rank $n$, and let $f\colon X\to Y$ be a $q$--quasiisometric embedding. If $H\subset X$ is a directionally strongly branch-complemented flat, then there is a constant $D$, depending only on $n$, $q$, and the coarse density in \cref{def:intro_strong_branching}, such that the following holds. 

There is a singular flat of $Y$ at Hausdorff distance at most $D$ from $f(H)$. Moreover, if $O\subset H$ is a singular orthant, then $f(O)$ lies at Hausdorff distance at most $D$ from a singular orthant.
\end{mthm}

This seems to be the first general result giving control on the images of lower-rank flats.


The fact that \cref{mthm:strong_ndim} gives us control over the images of orthants implies that the quasiisometric embedding induces an embedding between distinguished subsets of the Tits boundaries of $X$ and $Y$. The following is a rephrasing of part of \cref{def:singular_boundary_graph}, together with part of \cref{def:fully_branching_boundary}. We say that a geodesic ray is strongly branch-complemented if it is contained in a strongly branch-complemented geodesic.

\begin{definition} \label{def:boundary_intro}
Let $X$ be a CAT(0) cube complex. The \emph{singular boundary graph} of $X$, denoted $\partial_{\mathrm{sing}}X$, has a vertex for each point of $\partial_TX$ represented by a singular geodesic ray, and two vertices are joined by an edge whenever they are at angle $\frac\pi2$.

The \emph{strongly branch-complemented boundary graph} of $X$, denoted $\partial_{\mathrm{sbc}}X$, has a vertex for each point of $\partial_TX$ represented by a strongly branch-complemented geodesic, and two such vertices $\xi,\eta$ are joined by an edge whenever there exists a directionally strongly branch-complemented 2--flat $H$ with $\xi,\eta\in\partial_TH$.
\end{definition}

A basic example is obtained when $X$ is a product of $n$ trees. Since the singular geodesics are precisely those contained in a single factor, $\partial_{\mathrm{sing}}X$ is just the 1--skeleton of the join of the boundaries of the factors: it is a complete $n$--partite graph. If each factor is a regular tree of degree at least three, then $\partial_{\mathrm{sbc}}X = \partial_{\mathrm{sing}}X$. 

Another important example is for $X$ the universal cover of the Salvetti complex of a right-angled Artin group $A_\Gamma$. In this setting, the adjacency relations between the endpoints in $\partial_T X$ of standard geodesics are encoded by the \emph{extension graph} $\Gamma^\ext$ of $\Gamma$, introduced by Kim--Koberda \cite{kimkoberda:embedability}. In particular, if every standard geodesic is branch-complemented, which, as mentioned, can be read from $\Gamma$, then all standard flats of $X$ are directionally strongly branch-complemented, and hence $\Gamma^\ext$ is a subgraph of $\partial_{\mathrm{sbc}}X$. In particular, $\Gamma\subset\partial_{\mathrm{sbc}}X$. We refer to \cref{eg:singular_boundary} for more examples and a comparison with other notions of boundary.

Interpreting \cref{mthm:strong_ndim} on the level of boundaries, we obtain the following corollary.

\begin{mcor} \label{mcor:boundary}
Let $X$ and $Y$ be $n$--dimensional CAT(0) cube complexes of asymptotic rank $n$. Every quasiisometric embedding $X\to Y$ induces a graph embedding $\partial_{\mathrm{sbc}}X\to\partial_{\mathrm{sing}}Y$.
\end{mcor}

Again, a version of this also holds in more generality; see \cref{cor:boundary_map_asymptotic_rank}. In fact, in this $n$--dimensional setting, we are able to establish a more general version of \cref{mcor:boundary} involving a larger subgraph of $\partial_{\mathrm{sing}}X$ whose vertices need not be represented by strongly branch-complemented rays. Although the strong orthant-rigidity statement of \Cref{mthm:strong_ndim} can fail in this greater generality, one still has enough control on certain singular $2$--orthants to induce a similar boundary embedding; see \Cref{thm:fully_branching_2_flat,cor:embedding_fully_branching_boundary}.

Since asymptotic rank is a quasiisometry invariant, the fact that some of the above results can be stated without relying on the dimensions of the CAT(0) cube complexes gives them additional stability properties. For example, one can extend them to allow $X$ and $Y$ to be products of CAT(0) cube complexes and hyperbolic spaces of finite asymptotic dimension. Indeed, hyperbolic spaces of finite asymptotic dimension are quasiisometric to finite-dimensional CAT(0) cube complexes of asymptotic rank one \cite{haglundwise:combination,bonkschramm:embeddings,petyt:mapping}, and the various branching conditions are preserved by such a quasiisometry. If we consider only products of hyperbolic spaces, then our results can be compared with \cite{bowditch:quasiisometric}. We refer to \Cref{rem:cor_for_hyperbolic_spaces} for more discussion.

As an illustration of this, we record the following simple consequence of our methods. Recall that $A_{C_n}$ denotes the right-angled Artin group defined on the cycle graph with $n$ vertices. See \cref{eg:c5}.

\begin{mthm} \label{mthm:c5}
If $Y_1$ and $Y_2$ are hyperbolic spaces and $n>1$ is odd, then there is no quasiisometric embedding $A_{C_n}\to Y_1\times Y_2$.
\end{mthm}

This result complements a theorem of Rull \cite{rull:embedding}, which states that if $\Gamma$ is an $n$--colourable graph, then $A_\Gamma$ can be quasiisometrically embedded in a product of $n$ trivalent trees. In particular, if $n$ is even then $A_{C_n}$ can be quasiisometrically embedded in a product of two free groups.

\cref{mthm:c5} will be greatly expanded upon in \cite{baderbensaidpetyt:quasiisometric:products}.



\medskip

Because \cref{def:intro_branching,def:intro_strong_branching} are phrased for general CAT(0) spaces, the methods of this paper can be used more widely. We include \Cref{sec:appendix}, in which we explain how our results can be used to recover rigidity results for quasiisometric embeddings of symmetric spaces and Euclidean buildings of the same spherical type. 

The case of type $A_1^n$ follows rather directly, since symmetric spaces of non-compact type and Euclidean buildings of spherical type $A_1^n$ are products of rank-one factors, and hence quasiisometric to finite-dimensional CAT(0) cube complexes of asymptotic rank $n$. The general case requires only minor adaptations of our arguments, beyond replacing orthants with Weyl cones and using the appropriate quasiflat rigidity theorems. The details, together with the analogues of the relevant results, are proved in \Cref{sec:general_spherical_type}. 

Although these results are already known \cite{fisherwhyte:quasiisometric,nguyen:quasiisometric}, this illustrates the flexibility of our approach. In fact, the branching properties found in symmetric spaces and Euclidean buildings tend to be considerably stronger than is needed for the definitions given above. The analogue of Theorems~\ref{mthm:sketch} and~\ref{mthm:strong_ndim} is that a quasiisometric embedding between symmetric spaces of non-compact type or thick (not necessarily irreducible) Euclidean buildings, provided that the domain and codomain have the same spherical type, sends singular flats of every dimension within uniformly bounded Hausdorff distance of singular flats. In particular, this yields a genuine embedding of the Tits boundary of the domain into that of the codomain.

\subsection{Strategy and main tools in the proofs} \label{subsec:strat}

As is the case in \cite{kleinerleeb:induced}, one of the main tools in the proofs of Theorems~\ref{mthm:subflats} and~\ref{mthm:strong_ndim} is the asymptotic cone, which is a way of looking at a space from infinity. The point is that a quasiisometric embedding $f\colon  X \to Y$ induces a biLipschitz embedding between asymptotic cones $f_\omega \colon X_\omega \to Y_\omega$. This removes coarseness and enables the use of topological and analytic arguments. (This is a standard trick when working in the coarse setting.)

The basic idea is to show that if $X$ and $Y$ are finite-dimensional CAT(0) cube complexes of asymptotic rank $n$, and if $F \subset X$ is a directionally branch-complemented $n$--flat, then $f_\omega$ sends its ultralimit $F_\omega \subset X_\omega$ to a genuine $n$--flat of $Y_\omega$, and then to bring this information back to the original spaces. This is done in three steps. 

\smallskip\noindent\textbf{Step 1.}
Since $F$ is an $n$--flat, its ultralimit $F_\omega \subseteq X_\omega$ is again an $n$--flat. More importantly, the fact that $F$ is directionally branch-complemented transfers to $F_\omega$ but instead of having coarsely dense sets of parallels as in \cref{def:intro_branching}, we have a property of \emph{all} parallels. More precisely, $F_\omega$ is spanned by geodesics $\gamma_1,\dots,\gamma_n$ such that, for each $i$, the geodesic $\gamma_i$ admits a transverse $(n-1)$--flat $H_i \subset F_\omega$, and moreover all parallels of $\gamma_i$ and $H_i$ inside $F_\omega$ are branching. This first step is to prove that the $f_\omega$-image of each $\gamma_i$ is equal to a genuine geodesic of $Y_\omega$.

To do so, we use the fact that the asymptotic cone $Y_\omega$ carries the structure of a \emph{median metric space}. Since $Y$ has asymptotic rank $n$, the median rank of $Y_\omega$ is at most $n$ \cite{bowditch:coarse}. A key input is then a structure theorem of Bowditch \cite{bowditch:large:mapping} on top-dimensional biLipschitz flats in median metric spaces; see \Cref{prop:Bowditch_bilipflat_cubulated}. Applied to the restriction of $f_\omega$ to $F_\omega$, it implies that outside a subset $S \subset F_\omega$ of codimension $2$, every point of $F_\omega$ has a neighbourhood on which $f_\omega$ is \emph{flat}, in the sense that its image is contained in a Euclidean $n$--cube. The same theorem also implies that the images of all branching flats are \emph{cubulated}, namely locally finite unions of Euclidean cubes of the corresponding dimension, endowed with the $\ell^1$ metric. 

Using a tool from geometric measure theory, see \Cref{thm:rectifiable-level-sets}, one shows that almost every parallel of each $\gamma_i$ avoids $S$. Fix such a parallel $\gamma$, and let $x \in \gamma$. Let $H$ be the transverse parallel of $H_i$ passing through $x$. Near $x$, the image of $\gamma$ is a finite union of ``straight'' segments, and the image of $H$ is a finite union of $(n-1)$--dimensional cubes, while both lie inside a single Euclidean $n$--cube because $x$ avoids $S$. At this point, a local topological argument enters: since $H$ separates $\gamma$ inside $F_\omega$, the same separation phenomenon must persist in the image, which prevents $f_\omega(\gamma)$ from ``turning'' locally; see \Cref{sec:top_arguments}. It follows that $f_\omega(\gamma)$ is locally a CAT(0) geodesic around every point, and hence is a global geodesic. By continuity of $f_\omega$, the same conclusion holds for all parallels of the $\gamma_i$.

\smallskip\noindent\textbf{Step 2.}
The second step is to show that $F_\omega$ is mapped to a genuine $n$--flat. Since all parallels of $\gamma_1$ and $\gamma_2$ are sent to geodesics, one can translate the image of $\gamma_1$ along the image of $\gamma_2$ to show that the $2$--flat spanned by $\gamma_1$ and $\gamma_2$ is sent to a genuine $2$--flat by $f_\omega$. One then translates this $2$--flat along the image of $\gamma_3$, and so on, continuing inductively. These first two steps are concluded in \cref{sec:flats_rigidity}.

\smallskip\noindent\textbf{Step 3.}
Once one knows that $f_\omega(F_\omega) = (f(F))_\omega \subset Y_\omega$ is an $n$--flat, the final step is to transfer this information back down to $Y$. This is carried out in \Cref{sec:back_from_cone}. In this part of the argument, things are considerably simpler in the case that $Y$ is $n$--dimensional, so let us outline that first.

When $Y$ is $n$--dimensional, it follows from Huang's quasiflats theorem, \cite{huang:top}, that $f(F) \subset Y$ lies at finite Hausdorff distance from a union of finitely many $n$--orthants, and the fact that $(f(F))_\omega \subset Y_\omega$ is a genuine flat then forces this union to consist of exactly $2^n$ orthants. One then passes to the ``support set'' (see \cref{sec:back_from_cone}) of the quasiflat $f(F)$ and uses a sharp volume-growth argument, again following \cite{huang:top}, to conclude that this support set is itself an $n$--flat. We thus conclude that $f(F)$ is Hausdorff-close to an actual $n$--flat in $Y$. This argument is \cref{prop:quasiflats-from-bilip-flats}.

In the more general setting where $Y$ only has asymptotic rank $n$, the argument is more involved, for two main reasons. Firstly, one no longer has support sets for $n$--quasiflats when the ambient space is higher-dimensional. And secondly, in this setting the output of Huang's quasiflats theorem is no longer true. The correct statement here is Bowditch's quasiflats theorem \cite{bowditch:quasiflats}, which instead of providing a genuine union of orthants in $Y$ gives only a coarse map $\Omega\to Y$, where $\Omega$ is a \emph{panel complex} (see \cref{def:panel}). In this case, carrying out Step~3 decomposes into three sub-steps.

\smallskip\noindent\textbf{Step 3a.}
First we analyse Bowditch's construction of the panel complex $\Omega$. Up to finite Hausdorff distance, $\Omega$ is built as a union of orthants inside a larger $n$--dimensional CAT(0) panel complex $\Psi$. By using the same strategy as in the $n$--dimensional case of Step~3, we show that $\Omega$ can actually be taken to be a flat subcomplex inside $\Psi$. Bowditch's quasiflats theorem now tells us that $f(F)$ lies at finite Hausdorff distance from the image $\phi(\Omega)$ of $\Omega$ under a quasiisometric embedding that is \emph{quasimedian} on each orthant of $\Omega$.

\smallskip\noindent\textbf{Step 3b.}
So far we have replaced the quasiflat $f(F)$ by another quasiflat $\phi(\Omega)$ at finite Hausdorff distance. But now we have extra median information on the orthants of $\Omega$ that we did not have for $F$. It suffices to find a flat at finite Hausdorff distance from $\phi(\Omega)$, and in fact, by an argument similar to that used in Step~2, for this it suffices to show that every singular geodesic of $\Omega$ is mapped at finite Hausdorff distance from a geodesic in $Y$.

With that goal in mind, in this sub-step we upgrade the median information we have about $\phi$, by showing that it is globally quasimedian, and not merely quasimedian when restricted to its orthants. To do this, we show that the $\phi$-image of each orthant is \emph{coarsely median-convex}, and then iteratively apply a gluing result for the quasimedian property to build up larger and larger subspaces of $\Omega$ on which $\phi$ is quasimedian. After a finite number of steps, we obtain that $\phi$ is globally quasimedian.

\smallskip\noindent\textbf{Step 3c.}
Once $\phi$ is known to be quasimedian, we can show that if $\gamma\subset\Omega$ is a singular geodesic, then $\phi(\gamma)$ lies at finite Hausdorff distance from a convex subcomplex of $Y$. That convex subcomplex is itself a finite-dimensional CAT(0) cube complex quasiisometric to a line, so \cref{prop:lines} shows that it contains a CAT(0) geodesic, which is necessarily at finite Hausdorff distance from $\phi(\gamma$. As described above, we can now conclude from an argument similar to Step~2 that $\phi(\Omega)$, and hence $f(F)$, is at finite Hausdorff distance from a flat in $Y$.

\smallskip\noindent\textbf{Concluding.}
The three steps above establish \cref{mthm:sketch}. Let us describe how to complete the proofs of Theorems~\ref{mthm:subflats} and~\ref{mthm:strong_ndim}.

To complete the proof of \Cref{mthm:subflats} and show that the image of a branching subflat of $F$ lies within finite Hausdorff distance of a singular flat, one has to control coarse intersections of finite unions of coarse orthants. Indeed, if $H \subset F$ is a branching $k$--dimensional subflat, we show that its image $f(H)$ is at finite Hausdorff distance from a union $\cal O$ of coarse $k$--orthants, obtained as a coarse intersection of the coarse $n$--orthants provided by \cite{bowditch:quasiflats}. Since $f(F)$ is at finite Hausdorff distance from an actual $n$--flat $F'$ by \cref{mthm:sketch}, the Tits boundary of $\cal O$ forms a round $(k-1)$--sphere inside $\partial_T F'$. Hence $f(H)$ is at finite Hausdorff distance from a $k$--dimensional subflat of $F'$. The details of these arguments are in \Cref{sec:structure_quasi_flats}. 

For \cref{mthm:strong_ndim}, the stronger branching assumption allows the last step of the previous argument to be strengthened: every branching subflat of $F$ is mapped within uniformly bounded Hausdorff distance of a singular flat. In particular, all singular geodesics of $F$ are mapped within uniformly bounded Hausdorff distance of singular geodesics. This uniformity then allows us to apply an argument similar to the one in Step~2 to produce the desired flat. Essentially the same argument also shows that orthants are mapped within uniformly bounded Hausdorff distance of orthants. The details of this step are in \Cref{sec:singular_quasiflats}.

\subsection*{Outline of the article}

\begin{itemize}
\item   \cref{sec:prelim} contains background material on asymptotic cones, CAT(0) spaces, and CAT(0) cube complexes.
\item   In \cref{sec:singular} we introduce the notion of semisingular flats and state a couple of basic facts about them.
\item   \cref{sec:singular_quasiflats} is where we carry out the argument described in Step~2 of \cref{subsec:strat}. Namely, we give methods for building flats and orthants near a quasiflat under assumptions on the 1--dimensional subspaces of that quasiflat.
\item   The purpose of \cref{sec:structure_quasi_flats} is to show that certain structural properties of quasiflats pass down to intersections. These results are used in ``Concluding'' part of \cref{subsec:strat}.
\item   \cref{sec:qm_cc} develops various statements about quasimedian maps between orthants and CAT(0) cube complexes that are needed for carrying out Step~3 of \cref{subsec:strat} in the general case. It also contains analogues of some of the results of \cref{sec:structure_quasi_flats} under more general, coarser hypotheses.
\item   \cref{sec:back_from_cone} is about ``coming back'' from the asymptotic cone. That is, it shows that if one wants to find a flat near a quasiflat, then it suffices to do so in the asymptotic cone. This corresponds to Step~3 in \cref{subsec:strat}.
\item   The short \cref{sec:top_arguments} contains the topological argument described in Step~1 of \cref{subsec:strat}.
\item   In \cref{sec:flats_rigidity}, we introduce the various branching conditions on flats and carry out Step~1 from \cref{subsec:strat}. That is, we show that the branching conditions are stable under passing to asymptotic cones, and that their images on the level of asymptotic cones are flat.
\item   Our main results are proved in \cref{sec:fully_branching_theorems}.
\item   Finally, \cref{sec:appendix} is devoted to the analogous statements for symmetric spaces and Euclidean buildings.
\end{itemize}

The reader who is interested only in the $n$--dimensional case of the main results can avoid essentially all of the quasimedian-related aspects of the paper and need not concern themself with semisingularity. They can therefore completely ignore \cref{sec:qm_cc}, and from \cref{sec:back_from_cone} only \cref{subsec:back_dim_n} is needed.


\subsection*{Acknowledgements}
SB and OB are grateful to Stephan Stadler for several stimulating discussions at the Max Planck Institute for Mathematics in Bonn in 2024, from which this project grew and to which some of its initial ideas are due. OB is grateful to the Max Planck Institute for Mathematics in Bonn for its financial support, and acknowledges support from the FWO and F.R.S.-FNRS under the Excellence of Science (EOS) programme (project ID 40007542). SB and HP thank the Isaac Newton Institute for their hospitality during the programme \emph{Operators, Graphs, Groups}, where some of the work on this paper took place (EPSRC grant EP/Z000580/1).

Claude was used to check for typos in the first version of this paper.

\section{Preliminaries}\label{sec:prelim}

Let $(X,d)$ be a metric space. We write $B(x,r)$ and $S(x,r)$ for the closed ball and sphere of radius $r$ centred on $x$, respectively. That is,
$$
B(x,r) := \{ z \in X \mid d(x,z) \leq r \}, \qquad
S(x,r) := \{ z \in X \mid d(x,z) = r \}.
$$
For $A\subseteq X$ and $D\geq 0$, the \emph{$D$--neighbourhood} of $A$ is $A^{+D}=\bigcup_{a\in A}B(a,D)$. The \emph{Hausdorff distance} between subsets $A,B\subset X$ is
$$
d_{\mathrm{Haus}}(A,B) \,=\, \inf\{r\geq 0 \mid A\subseteq B^{+r}\text{ and }B\subseteq A^{+r}\}.
$$
When two subsets lie at finite Hausdorff distance from one another, we sometimes say that they are \emph{Hausdorff-close}. We say that $B$ \emph{coarsely contains} $A$ if $A\subseteq B^{+D}$ for some $D\geq 0$. If $B$ is a subset of $A$ that coarsely contains $A$, then we say that $B$ is \emph{coarsely dense} in $A$.



\subsection{Asymptotic cones}

Asymptotic cones provide a way to ``zoom out'' from a space. This allows one to work with fine objects and maps rather than coarse ones.

\begin{definition}[Ultrafilter]
An \emph{ultrafilter} on $\N$ is a set $\omega$ of subsets satisfying:
\begin{itemize}
\item   for each $A\subset\N$, exactly one of $A\in\omega$ and $\N\ssm A$ belongs to $\omega$;
\item   if $A\in\omega$ and $A\subseteq B$, then $B\in\omega$;
\item   if $A,B\in\omega$, then $A\cap B\in\omega$.
\end{itemize}
The ultrafilter $\omega$ is \emph{nonprincipal} if it contains no finite sets.
\end{definition}

We shall only consider nonprincipal ultrafilters in this article, so we drop the word ``nonprincipal''.

\begin{definition}[Asymptotic cone]
Let $\omega$ be a (nonprincipal) ultrafilter. Given a sequence $(x_n)$ in $\R$, if there exists some $x\in\R$ such that for each $\eps>0$, the set $\{n\,:\,|x-x_n|<\eps\}$ lies in $\omega$, then we call $x$ the \emph{ultralimit} of the sequence, and write $x=\lim_\omega(x_n)$. Each sequence has at most one ultralimit.

Let $(\lambda_n)$ be a divergent sequence of positive numbers, called a \emph{scaling sequence}. For a metric space $(X,\dist)$ with a sequence of basepoints $(o_n)$, the \emph{asymptotic cone} $\lim_\omega(X,(\lambda_n),(o_n))$ is the complete metric space obtained as follows. Let $Y$ be the set of all sequences $(x_n)$ with $x_n\in X$ such that $\lim_\omega\frac1{\lambda_n}\dist(o_n,x_n)$ exists. The function $\hat\dist$ on $Y\times Y$ given by $\hat\dist((x_n),(y_n))=\lim_\omega\frac1{\lambda_n}\dist(x_n,y_n)$ is a pseudometric. We define $\lim_\omega(X,(\lambda_n),(o_n))$ to be the metric quotient of $Y$.

If $(Z_n)$ is a sequence of subsets of $X$, then we get a subset $\lim_\omega(Z_n)\subset\lim_\omega(X,(\lambda_n),(o_n))$ by considering only sequences whose $n^\mathrm{th}$ term lies in $Z_n$. If $\lim_\omega(Z_n)$ is empty, then we say that the ultralimit of $(Z_n)$ \emph{does not exist}. Otherwise, we refer to $\lim_\omega(Z_n)$ as \emph{the ultralimit} of $(Z_n)$. If $(Z_n)$ is a constant sequence $(Z)$, then we simply refer to it as the ultralimit of $Z$.
\end{definition}


For any ultrafilter $\omega$, scaling sequence $(\lambda_n)$, and sequence $(o_n)$ of basepoints in $X$, if $f\colon X\to Y$ is a $q$--quasiisometric embedding of metric spaces, then $f$ induces a biLipschitz embedding $\hat f\colon \lim_\omega(X,(\lambda_n),(o_n))\to \lim_\omega(Y,(\lambda_n),(f(o_n)))$ of asymptotic cones. 

The following definition comes from \cite{wenger:asymptotic}, see Proposition~3.1 thereof. It provides a large-scale notion of rank for general metric spaces.

\begin{definition}[Asymptotic rank]\label{def:asymptotic_rank}
Let $X$ be a metric space. The \emph{asymptotic rank} of $X$ is the supremal $n$ such that there exists an asymptotic cone of $X$ and a sequence $(B_k)$ of subsets of $X$ whose ultralimit is isometric to the unit ball in some normed space $(\R^n,\|\cdot\|)$.
\end{definition}

\subsection{CAT(0) spaces}\label{sec:CAT(0)_spaces}

We refer the reader to \cite{ballmann:lectures} and \cite{bridsonhaefliger:metric} for background on CAT(0) spaces. The main classes of CAT(0) spaces considered in this paper are \emph{CAT(0) cube complexes} (see \cref{subsec:ccc}) and their asymptotic cones. 
Our results hold for symmetric spaces of non-compact type and Euclidean buildings, and the appendix will deal with those. 

\begin{definition}[Flats and orthants]\label{def:flats_in_CAT(0)}
Let $X$ be a CAT(0) space. A \emph{$k$--flat} in $X$ is the image of an isometric embedding $(\mathbb R^k,\|\cdot\|_2)\to X$. A \emph{$k$--orthant} is the image of an isometric embedding $([0,\infty)^k,\|\cdot\|_2)\to X$. The \emph{cone point} of an orthant is the image of $(0,\dots,0)$. 

When the dimension is implied or not important, we will just say \emph{flat} and \emph{orthant}.
\end{definition}

For example, $\R^n$ has $2^n$ $n$--orthants whose cone points are the origin.

\begin{definition}[Parallels] \label{def:parallel_set}
Let $X$ be a complete CAT(0) space, and suppose that $f_1,f_2:A\to X$ are two isometric embeddings of a metric space $A$ with closed, convex images. Let $\pi_1:X\to A_1$ and $\pi_2:X\to A_2$ denote the closest-point projection maps. We say that $f_1(A)$ and $f_2(A)$ are \emph{parallel} if $f_2=\pi_2f_1$ and $f_1=\pi_1f_2$ and the function $a\mapsto d(f_1(a),f_2(a))$ is constant on $A$.

The \emph{parallel set} of a closed, convex subset $B\subset X$, denoted $P(B)$, is the union of all subsets parallel to $B$. It is a closed convex subset of $X$, and it admits a canonical splitting as a metric product $P(B)=B\times Y$ for some complete CAT(0) space $Y$. See \cite[Section~2.3.3]{kleinerleeb:rigidity}.
\end{definition}

Note that any two flats that are at finite Hausdorff distance are parallel, by convexity of the metric. Moreover, they necessarily have the same dimension.

\begin{definition}[Angles]
Let $\gamma_1,\gamma_2:[0,1]\to X$ be geodesics in a CAT(0) space with $\gamma_1(0)=\gamma_2(0)=x$. The \emph{angle} $\angle_x(\gamma_1,\gamma_2)$ between $\gamma_1$ and $\gamma_2$ at $x$ is the minimal nonnegative number $\theta$ such that $\cos\theta \,=\, 1-\lim_{t\to0}\frac{\dist(\gamma_1(t),\gamma_2(t))^2}{2t^2}$. This expression is derived from the cosine law. For $a,b,x\in X$, we define $\angle_x(a,b)=\angle_x([x,a],[x,b])$. 
\end{definition}

\begin{definition}[Tits boundary]
We denote by $\partial_TX$ the Tits boundary of a CAT(0) space $X$. As a set, $\partial_TX$ is the set of equivalence classes of geodesic rays, where two rays are defined to be equivalent if they are Hausdorff-close. Unless otherwise stated, we shall equip $\partial_TX$ with the \emph{angular metric $\angle$}, which, given $\xi,\eta\in\partial_TX$, is defined by setting
$$
\angle(\xi,\eta)=\sup_{p\in X}\angle_p(\xi,\eta),
$$
where $\angle_p(\xi,\eta)=\angle_p([p,\xi),[p,\eta))$. See \cite[Prop.~II.9.5]{bridsonhaefliger:metric}. 
\end{definition}


\begin{definition}[Link] \label{def:link}
Let $X$ be a CAT(0) space and let $x\in X$. The \emph{link} of $x$, denoted $\Sigma_xX$, is the space of directions of $X$ at $x$. That is, it is the set of equivalence classes of geodesics $\gamma\colon [0,1]\to X$ with $\gamma(0)=x$, where we declare $\gamma_1$ and $\gamma_2$ to be equivalent if $\angle_x(\gamma_1,\gamma_2)=0$. The angle function $\angle_x(\cdot,\cdot)$ descends to a metric on $\Sigma_xX$ \cite[II.3.18]{bridsonhaefliger:metric}.
\end{definition}

\begin{definition}[Euclidean cone] \label{def:Euclidean_cone}
    The \emph{Euclidean cone} over a metric space $\Delta$ is a uniquely geodesic metric space, denoted $C(\Delta)$, with a specified \emph{cone point}. As a set it is obtained from $\Delta\times[0,\infty)$ by identifying all points at height zero; see \cite[Def.~I.5.6]{bridsonhaefliger:metric} for the definition of the metric. If $\Delta$ is \emph{CAT(1)}, then $C(\Delta)$ is CAT(0), and its boundary with the angle metric is $\Delta$ \cite[II.3.14]{bridsonhaefliger:metric}. 
\end{definition}

The following lemma will be used in several places. Part of it is stated in \cite[Lem.~10.6]{kleiner:local}. 

\begin{lemma}\label{lem:embedding-Titsboundary-into_link_and_Tits_boundary}
Let $X$ be a CAT(0) space, and let $\hat X$ be an asymptotic cone of $X$ with respect to a fixed basepoint $x$. Write $o = (x)\in\hat X$. For each $\xi\in\partial_T X$, let $r_\xi \subseteq X$ denote the geodesic ray $[x,\xi)$, and let $\hat r_\xi \subseteq \hat X$ denote its ultralimit. The following three maps are all isometric embeddings.
\[
\begin{aligned}
\varphi_{o}&\colon(\partial_T X,\angle)\to (\Sigma_o\hat X,\angle_o), \\
\varphi_T&\colon(\partial_T X,\angle)\to (\partial_T\hat X,\angle), \\
\Phi&\colon C(\partial_T X)\to \hat X,
\end{aligned}
\;
\begin{aligned}
&\text{ given by setting } \varphi_{o}(\xi) \text{ to be the initial direction of } \hat r_\xi. \\
&\text{ given by } \varphi_T(\xi)=\hat r_\xi(+\infty). \\
&\text{ given by } \Phi(\xi,t)=\hat r_\xi(t).
\end{aligned}
\]
\end{lemma}

\begin{proof}
For every $\xi\in\partial_TX$, the ultralimit $\hat r_\xi$ is a geodesic ray based at $o$, so the maps are all well defined. Given $\xi,\eta\in\partial_T X$, set 
\[ 
c=\lim_{s\to\infty}\frac{1}{s}\dist\bigl(r_\xi(s),r_\eta(s)\bigr).
\]
According to \cite[Thm~II.4.4]{ballmann:lectures}, the angle $\angle(\xi,\eta)$ is equal to the angle opposite the side $c$ in the Euclidean triangle with side-lengths $1$, $1$, and $c$.

Let $\Lambda=(\lambda_n)$ be the scaling sequence associated with $\hat X$. Given $t>0$, consider the points $p_t=(r_\xi(t\lambda_n))\in\hat r_\xi$ and $q_t=(r_\eta(t\lambda_n))\in\hat r_\eta$. By construction we have $\dist(o,p_t)=\dist(o,q_t)=t$. We also have
\[
\dist(p_t,q_t) 
    \,=\, \lim_\omega \frac1{\lambda_n}\dist\big(r_\xi(t\lambda_n),r_\eta(t\lambda_n)\big)
    \,=\, t\lim_{n\to\infty}\frac1{t\lambda_n}\dist\big(r_\xi(t\lambda_n),r_\eta(t\lambda_n)\big)
    \,=\, tc.
\]
Hence, for every $t>0$, the comparison angle at $o$ of the triangle $(o,p_t,q_t)$ is equal to $\angle(\xi,\eta)$.

Letting $t\to 0$, we obtain $\angle_o(\hat r_\xi,\hat r_\eta)=\angle(\xi,\eta)$, which shows that $\varphi_o$ is an isometric embedding. Instead letting $t\to\infty$ and applying \cite[Prop.~II.9.8]{bridsonhaefliger:metric}, we obtain $\angle\big(\hat r_\xi(\infty),\hat r_\eta(\infty)\big)=\angle(\xi,\eta)$, so $\varphi_T$ is also an isometric embedding.

Finally we consider $\Phi$. We have shown that $\angle_o(\hat r_\xi,\hat r_\eta)=\angle\big(\hat r_\xi(\infty),\hat r_\eta(\infty)\big)=\angle(\xi,\eta)$. By the Flat Sector Theorem \cite[Cor.~II.9.9]{bridsonhaefliger:metric}, if $\angle(\xi,\eta)<\pi$ then the rays $\hat r_\xi$ and $\hat r_\eta$ bound a Euclidean sector of angle $\angle(\xi,\eta)$. If $\angle(\xi,\eta)=\pi$, then $\hat r_\xi\cup \hat r_\eta$ is a geodesic. Therefore, for all $s,t\geq 0$, the points $\Phi(\xi,s)$ and $\Phi(\eta,t)$ lie in a Euclidean sector of angle $\angle(\xi,\eta)$. The Euclidean law of cosines now yields
\[
\dist\bigl(\Phi(\xi,s),\Phi(\eta,t)\bigr)^2 \,=\, s^2+t^2-2st\cos\angle(\xi,\eta).
\]
This is exactly the distance formula in the Euclidean cone $C(\partial_T X)$, so $\Phi$ is an isometric embedding.
\end{proof}

\begin{remark}
The maps in Lemma~\ref{lem:embedding-Titsboundary-into_link_and_Tits_boundary} are not in general surjective, even if $X$ is a simplicial tree with bounded valence and no leaves. For instance, let $X$ be obtained from $\Z$ by attaching a ray at each integer. Its Tits boundary is countable, but $\hat X$ has uncountable valence.
\end{remark}





\subsection{CAT(0) cube complexes and median metric spaces} \label{subsec:ccc}

We refer the reader to \cite{wise:structure,bowditch:median:book,genevois:algebraic} for background on CAT(0) cube complexes. 

One of the many equivalent ways to define them is as follows. A simplicial graph is \emph{median} if for every triple $v_1,v_2,v_3$ of vertices there is a unique vertex $\mu$ such that $\dist(v_i,v_j)=\dist(v_i,\mu)+\dist(\mu,v_j)$ for all $i,j$. A CAT(0) cube complex is the cell complex obtained from a median graph by attaching, in the obvious way and with obvious identifications, a unit cube $[0,1]^n$ to every subgraph isometric to the Cartesian product of $n$ edges, for each $n\ge2$. 

The \emph{dimension} of a CAT(0) cube complex is the supremal dimension among cubes involved in its construction. If a CAT(0) cube complex is not finite-dimensional then it is not a complete metric space.

If $X$ is a CAT(0) cube complex, then the length metric $d_2$ obtained by equipping each cube with the $\ell^2$ metric makes $X$ into a CAT(0) space. If one instead equips the cubes with the $\ell^1$ metric, then one obtains a metric $d_1$ on $X$ that makes $X$ into a \emph{median metric space}. If $X$ is finite-dimensional then these metrics are both complete. See \cite{bowditch:median:book} for an authoritative account of median metric spaces and \emph{median algebras}. 

\begin{definition}[Median metric space] \label{def:median_metric_space}
Let $(M,d)$ be a metric space. For $a,b\in M$ set $[a,b] = \{x \in M \mid d(a,b)= d(a,x) + d(x,b)\}$.

A \emph{median metric space} is a metric space $(M,d)$ such that for every triple $a,b,c\in M$, the intersection $[a,b]\cap[b,c]\cap[c,a]$ consists of exactly one point. We denote this point $\mu(a,b,c)$ and call it the \emph{median} of the triple $(a,b,c)$. The map $\mu\colon M^3\to M$ is a ternary operation that is 1--Lipschitz in each coordinate and makes $M$ into a \emph{median algebra}.

The \emph{median rank} of $M$ is the largest $n\in\mathbb{N}\cup\{\infty\}$ such that $M$ contains a subset median-isomorphic to $\{0,1\}^n$ with the product median. 
\end{definition}

A median graph is precisely a graph whose edge metric makes the vertex set into a median metric space. If $X$ is a CAT(0) cube complex, then $(X,d_1)$ and all of its asymptotic cones are median metric spaces as well. If $X$ is a finite-dimensional CAT(0) cube complex, then corresponding asymptotic cones of $(X,d_2)$ and $(X,d_1)$ are biLipschitz equivalent. More generally, a result of Bowditch shows that every complete, connected median metric space of finite \emph{median rank} admits a biLipschitz equivalent CAT(0) metric \cite{bowditch:some}.

\begin{definition}[Hyperplanes]
Every edge $xy$ of a CAT(0) cube complex has a \emph{hyperplane} $h$ dual to it, which consists of all points that are equidistant from $x$ and $y$ in the $\ell^1$ metric. If two edges are opposite in some square then they define the same hyperplane. Each hyperplane $h$ separates $X$ into exactly two connected components, called \emph{halfspaces}. We say that $h$ \emph{separates} a point $z$ from a point $w$ if they do not lie in the same halfspace of $h$. 
\end{definition}

The metric $\dist_1$ on the vertex set of a CAT(0) cube complex can be characterised with hyperplanes: $d_1(x,y)$ is equal to the number of hyperplanes that separate $x$ from $y$.

\begin{definition}[Crossing]
Let $h_1$ and $h_2$ be hyperplanes of a CAT(0) cube complex $X$, with halfspaces $h_1^+$, $h_1^-$, $h_2^+$, $h_2^-$. We say that $h_1$ and $h_2$ \emph{cross} if all four intersections $h_1^\pm\cap h_2^\pm$ are nonempty. If $h_1$ and $h_2$ cross, then there is a 2--cell of $X$ in which they can be seen to cross.
\end{definition}

CAT(0) cube complexes are CAT(0) spaces when given the metric $d_2$, so we can talk about subcomplexes being convex with respect to $\dist_2$. In fact, convexity of subcomplexes interacts strongly with both the median and the hyperplane structures.

\begin{remark} \label{rem:ccc_convexity}
Let $X$ be a finite-dimensional CAT(0) cube complex. The following are equivalent for a subcomplex $A\subset X$.
\begin{itemize}
\item   $A$ is convex as a subspace of the CAT(0) space $(X,\dist_2)$.
\item   $A$ is convex as a subspace of the metric space $(X,\dist_1)$.
\item   $\mu(a,b,x)\in A$ for every $a,b\in A$ and every $x\in X$.
\item   $A$ is a nonempty intersection of halfspaces of $X$.
\item   For every $x\in X$ there is a unique point $\pi_A(x)\in A$ with the property that $\mu(x,\pi_A(x),a)=\pi_A(x)$ for all $a\in A$.
\end{itemize} 
The point $\pi_A(x)$ is the unique closest point in $A$ to $x$ in the metric $\dist_1$. It can be obtained from $x$ by crossing exactly the hyperplanes that separate $x$ from $A$.
\end{remark}

More generally, a subset $A$ of a median space $M$ is \emph{median-convex} if $\mu(a,x,b)\in A$ for all $a,b\in A$, $x\in M$.

\begin{definition}[Hull] \label{def:hull}
Let $X$ be a finite-dimensional CAT(0) cube complex, and let $A\subset X$. The \emph{hull} of $A$ is the convex subcomplex $\Hull(A)$ obtained by intersecting all halfspaces that contain $A$. If $A=\{a,b\}$ consists of two points then we write $[a,b]=\Hull\{a,b\}$.
\end{definition}

\begin{remark} \label{rem:d_infty}
There is a third metric on a CAT(0) cube complex $X$ that is often useful, which is the metric $\dist_\infty$ obtained by equipping the cubes of $X$ with the $\ell^\infty$ metric. This makes $X$ into an \emph{injective} metric space. See \cite{lang:injective} for a discussion of injective spaces. We will not need anything from the theory of injective spaces; we only require two facts about $d_\infty$. Firstly, if $\dim X=n$, then we have 
\[
d_\infty \,\le\, d_2 \,\le\, \sqrt nd_\infty.
\]
Secondly, if $x$ is a vertex of $X$, then the ball $B_\infty(x,m)\subset(X,d_\infty)$ is a convex subcomplex of $X$ in the sense of \cref{rem:ccc_convexity} for all $m\in\mathbb N$. Note that it need not be convex for the metric $d_\infty$. One can also consider $\ell^p$ metrics on $X$ for other values of $p$, see \cite{haettelhodapetyt:lp}, but we shall not need those here.
\end{remark}

Unless otherwise stated, the metric we consider on a CAT(0) cube complex is the CAT(0) metric $d=d_2$.

The following relates the asymptotic rank of a CAT(0) cube complex to its cubical structure; its statement is essentially contained in \cite[\S3]{munropetyt:coarse}.

\begin{proposition} \label{lem:rank_cone_CCC}
Let $X$ be a CAT(0) cube complex and let $\hat X$ be an asymptotic cone of $X$. The rank of the median algebra $\hat X$ is bounded above by the asymptotic rank of $X$, which is equal to the supremal $n$ such that $X$ contains a subcomplex of the form $[0,m]^n$ for every positive integer $m$.
\end{proposition}

\begin{proof}
By \cite[Thm~2.3]{bowditch:coarse}, the rank of $\hat X$ is bounded above by the \emph{coarse median rank} of $X$. As noted in \cite[Rem.~3.3]{munropetyt:coarse}, the construction of \cite[Prop.~3.1]{munropetyt:coarse} shows that if the coarse median rank of $X$ is $n$, then there is a sequence $Q_m$ of subsets of $X$ such that $Q_m$ is the image of $[0,m]^n$ under a uniform-quality \emph{quasimedian} quasiisometric embedding. This relies on results of Bowditch \cite[\S9]{bowditch:quasiflats}. Since $X$ is a CAT(0) cube complex, this implies the existence of subcomplexes isometric to $[0,m]^n$ for all $m$.

Clearly the existence of a sequence of subcomplexes $[0,m]^n$ for all $m$ implies that $X$ has asymptotic rank at least $n$. It remains to see that the asymptotic rank bounds the dimension of such a sequence.
From \cite[Thm~6.9]{bowditch:large:mapping} and \cite[Cor.~3.7]{haettel:higher}, it is known that the rank of an asymptotic cone $X_\omega$ is equal to the \emph{separation dimension} of $X_\omega$, and it is an immediate consequence that the asymptotic rank of $X$ is bounded above by the supremal rank among asymptotic cones of $X$. Since this is bounded above by the coarse median rank of $X$, this completes the proof.
\end{proof}

\section{Singular flats} \label{sec:singular}

The metric $d_1$ on a CAT(0) cube complex is usually not uniquely geodesic, and the metric $d_2$ does not respect the cubical structure of $X$. When we refer to a flat or a geodesic in $X$ or an asymptotic cone of $X$, we will be considering the CAT(0) metric on $X$ unless specified otherwise. A special role is played by subsets that are flat when considered with both metrics.

\begin{definition}[Singular] \label{def:ccc_singular}
Let $X$ be a CAT(0) cube complex or an asymptotic cone of one, and let $A\subset(X,d_2)$ be a $k$--flat or $k$--orthant. We say that $A$ is \emph{singular} if it is the image of an isometric embedding $(\R^k,\dist_1)\to(X,\dist_1)$ or $([0,\infty)^k,\dist_1)\to(X,\dist_1)$, respectively. 

When $k=1$, we refer to these as \emph{singular geodesics} and \emph{singular geodesic rays}.
\end{definition}

Equivalently, a flat or orthant $A$ is singular if $\mu(a,b,x)\in A$ for all $a,b\in A$ and all $x\in X$.

Singular geodesics may be thought of as the ``axis-parallel'' ones. For instance, in $\mathbb{R}^2$ (with its standard cubulation), the singular geodesics are precisely the horizontal and vertical lines. They can be characterised in terms of hyperplanes: a geodesic $\gamma$ is singular if no two hyperplanes crossed by $\gamma$ cross each other.

\begin{lemma} \label{lem:top_dim_singular}
Let $X$ be an $n$--dimensional CAT(0) cube complex. If $F\subset X$ is an $n$--flat, then $F$ is singular.
\end{lemma}

\begin{proof}
If $n=1$, then $X$ is a tree, and every geodesic in a tree is singular. Now suppose that we have proved the lemma in dimensions less than $n$. Let $F\subset X$ be an $n$--flat. Take an arbitrary geodesic in $F$, and let $h$ be a hyperplane of $X$ that separates its endpoints. Recall that $h$ is itself a CAT(0) cube complex of dimension $n-1$. The intersection $h\cap F$ is an $(n-1)$--flat in an $(n-1)$--dimensional CAT(0) cube complex, so it is singular by induction. 

Let $\gamma\subset F$ be a geodesic orthogonal to the subflat $h\cap F$. The parallel set $P(h\cap F)$ splits as $P(h\cap F)=(h\cap F)\times Y$ for some CAT(0) cube complex $Y$. Since $X$ is $n$--dimensional, $Y$ must be 1--dimensional. Hence $\gamma\subset Y$ is singular. This shows that $F$ is spanned by singular geodesics, and hence it is singular.
%
\end{proof}

Since we will be considering CAT(0) cube complexes whose dimension is greater than their asymptotic rank, we shall need the following generalisation of singular flats and orthants.

\begin{definition}[Semisingular] \label{def:semisingular}
Let $X$ be a finite-dimensional CAT(0) cube complex. A flat, orthant, or ray in $X$ is \emph{semisingular} if its ultralimit is singular in every asymptotic cone of $X$ where it exists. 
\end{definition}

\begin{example} \label{eg:corner_to_corner}
For each $i\in\Z$, let $S_i$ be a unit square. Let $x_i$ and $y_i$ be a pair of opposite vertices in $S_i$. Let $X$ be the CAT(0) square complex obtained from $\bigcup_{i\in\Z}S_i$ by gluing $x_i$ to $y_{i-1}$ for all $i$, as in \cref{fig:semisingular_geod}. The asymptotic rank of $X$ is one. It contains no singular geodesics, but it contains a semisingular geodesic, drawn in \cref{fig:semisingular_geod}. Similarly, $X^n$ contains no singular $n$--flats, but does contain a semisingular $n$--flat.
\end{example}

\begin{figure}[ht]
    \centering
    \includegraphics[height=1.5cm]{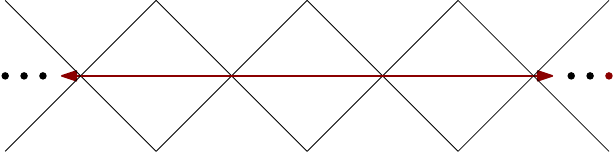} 
    \caption{A semisingular geodesic in a CAT(0) square complex with no singular geodesics.} \label{fig:semisingular_geod}
\end{figure}

Note that every flat or orthant that is \emph{coarsely median-convex} in the sense of \cref{def:cmc} is semisingular. The following example shows that the converse does not always hold.

\begin{example}
Let $H_n=\Hull\{(2^n,n),(2^{2n},2n)\}\subset\R^2$ be a rectangle in the plane whose height is logarithmic in its length. Let $Y=\bigcup_{n=1}^\infty H_{2^n}$. There is a unique geodesic ray $\gamma\subset Y$ emanating from $(2,1)$. In every asymptotic cone, the ultralimit $\hat\gamma$ is singular, but $\gamma$ is not coarsely median-convex. Note that $Y$ has asymptotic rank two, even though every asymptotic cone with fixed basepoint is a ray. 
\end{example}

This example also shows that not all semisingular geodesics are Morse.

\begin{lemma} \label{lem:top_rank_semisingular}
Let $X$ be a finite-dimensional CAT(0) cube complex of asymptotic rank $n$. If $F\subset X$ is an $n$--flat, then $F$ is semisingular.
\end{lemma}

\begin{proof}
The proof is similar to that of \cref{lem:top_dim_singular}. Let $\hat X=\lim_\omega(X,(\lambda_m),(o_m))$ be an asymptotic cone for which the ultralimit $\hat F$ of $F$ exists.

If $n=1$, then $\hat X$ is an $\R$--tree, and all geodesics in $\R$--trees are singular, so $\hat F$ is singular. For the inductive step, take a geodesic $\gamma\subset F$, and let $h_m$ be a hyperplane that separates its endpoints and lies at distance at most $r\lambda_m$ from $o_m$ for some fixed constant $r$. Note that $\gamma$ is not contained in any finite neighbourhood of any halfspace of any $h_m$, by convexity of the CAT(0) metric. The ultralimit $\lim_\omega(h_m\cap F)$ is therefore an $(n-1)$--flat in $\lim_\omega(h_m)$, hence is singular by induction. Taking a geodesic of $\hat F$ orthogonal to $\lim_\omega(h_m\cap F)$ and considering the parallel set of $\lim_\omega(h_m\cap F)$, we find that $\hat F$ is spanned by singular geodesics.
\end{proof}

\cref{lem:semisingular_subflat_boundary} will show that semisingular subflats of semisingular flats can be detected using the Tits boundary. We will need the following observation about the Tits boundary of a semisingular flat.

\begin{observation} \label{rem:boundary_semisingular}
Let $X$ be a finite-dimensional CAT(0) cube complex. If $F\subset X$ is a semisingular $n$--flat, then $\partial_TF$ admits a natural simplicial structure isomorphic to that of the Tits boundary of $\R^n$ with the standard cubulation, as we now describe.

Because $F$ is a semisingular $n$--flat, if we take any asymptotic cone of $X$ for which the ultralimit of $F$ exists, then that ultralimit is median-preservingly isometric to $\R^n$ with the standard median structure. In particular, every such ultralimit is median-preservingly isometric to the ultralimit $\hat F$ of $F$ in an asymptotic cone $\hat X$ of $X$ taken with respect to a fixed basepoint.

By \cref{lem:embedding-Titsboundary-into_link_and_Tits_boundary}, for such an asymptotic cone there is an isometric embedding $\partial_TF\to\partial_T\hat F$. As $\hat F$ is median-preservingly isometric to $\R^n$, its boundary $\partial_T\hat F$ is naturally identified with $\partial_T\R^n$. This gives the desired structure on $\partial_TF$.
\end{observation}

\begin{lemma} \label{lem:semisingular_subflat_boundary}
Let $X$ be a finite-dimensional CAT(0) cube complex, and let $F\subset X$ be a semisingular $n$--flat. If $H\subset F$ is a subflat such that $\partial_TH$ is a subcomplex of $\partial_TF$, then $H$ is semisingular. 
\end{lemma}

\begin{proof}
Let $\hat X$ be an asymptotic cone of $X$ for which the ultralimit $\hat H$ of $H$ exists. As $F$ contains $H$, its ultralimit exists. Also, $F$ is a semisingular $n$--flat, so by \cref{rem:boundary_semisingular}, its ultralimit, $\hat F$, is median-preservingly isometric to $\R^n$ with the standard median. Since $H$ is a subflat of $F$, its ultralimit $\hat H$ is a subflat of $\hat F$. To show that $H$ is semisingular, we just need to show that $\partial_T\hat H$ is a subcomplex of $\partial_T\hat F$. But this is automatic from the definition of the simplicial structure on $\partial_TF$.
\end{proof}

\section{Constructing flats and orthants} \label{sec:singular_quasiflats}

In this section, we study quasiisometric embeddings of flats with the standard cubulation into CAT(0) cube complexes. We show that if singular geodesics are mapped uniformly close to (semi)singular geodesics, then singular flats and orthants are mapped uniformly close to (semi)singular flats and orthants (\Cref{prop:singular_quasi_flats} and \Cref{lem:orthants_to_semisingular_orthants}, respectively).
This will later be applied in situations where we have good control on singular geodesics, generally as a consequence of control on top-dimensional flats, in order to gain control on singular flats of intermediate dimension. 

This is done by induction on dimension. For flats, we use parallel sets to build the target flat from the images of singular geodesics. We then handle orthants using a coarse-separation argument.

Afterwards, in \cref{subsec:non_uniform_case}, we strengthen the connection between (semi)singularity and the Tits boundary. Then, in the $n$--dimensional case and under an additional assumption on asymptotic cones, we extend \cref{lem:orthants_to_semisingular_orthants} to singular 2--orthants where we do not have uniform control on Hausdorff distances (\Cref{prop:2_orthant_in_target}).

\subsection{Finding flats close to quasiflats}

We begin with two simple lemmas. The first says that if a quasiflat lies in a finite neighbourhood of a flat, then the two are actually at finite Hausdorff distance. The second says that a geodesic contained in a bounded neighbourhood of a closed convex subset admits a parallel inside that subset.

\begin{lemma}\label{lem:quasiisom_quasi_surj}
Let $X$ be a CAT(0) space, let $F \subseteq X$ be a $k$--flat, and let $f\colon  \mathbb{R}^k \to X$ be a $q$-quasiisometric embedding. For every $D\geq 0$, there exists $D'=D'(q,k,D)$ such that if $f(\mathbb{R}^k) \subseteq F^{+D}$, then
$\dist_{\mathrm{Haus}}(F, f(\mathbb{R}^k))\leq D'$.
\end{lemma}

\begin{proof}
Let $\pi_F \colon X \to F$ be the closest-point projection map. The composition $\pi_F\circ  f\colon \R^k\to F$ is a $(q,q+2D)$--quasiisometric embedding. Since $F$ is isometric to $\mathbb{R}^k$, such a quasiisometric embedding is uniformly coarsely onto; see \cite[Lem.~10.84]{drutukapovich:geometric}. That is, there exists $M=M(k,q,D)$ such that $F \subseteq (\pi_F \circ f(\mathbb R^k))^{+M}$. Since $\pi_F (f(\mathbb R^k)) \subseteq f(\mathbb R^k)^{+D} $, it follows that 
$F \subseteq f(\mathbb R^k)^{+D+M}$ and so $\dist_{\mathrm{Haus}}(F,f(\mathbb{R}^k))\leq D+M=D'$.
\end{proof}

\begin{lemma}\label{lem:parallel_geod_in_convex}
Let $X$ be a complete CAT(0) space, let $A \subseteq X$ be a closed, convex subset, and let $\gamma \subseteq X$ be a geodesic. If $\gamma \subseteq A^{+D}$, for some $D \geq 0$, then there exists a parallel geodesic $\gamma' \subseteq A$ such that $d_{\mathrm{Haus}}(\gamma,\gamma') \leq D$.
\end{lemma}

\begin{proof}
Let $\pi_A\colon  X \to A$ be the closest-point projection, and for $t \in \mathbb R$ set $\gamma'(t)=\pi_A(\gamma(t))$. The map $t \mapsto d(\gamma(t),A) = d(\gamma(t), \gamma'(t))$ is convex because $A$ is convex. Indeed, given $t_1<t_2<t_3$, the geodesic $\eta[t_1,t_2]\to X$ from $\pi_A\gamma(t_1)$ to $\pi_A\gamma(t_3)$ is contained in $A$, and by convexity of the metric we have 
\[
\dist(\gamma(t_2),\eta(t_2)) \,\le\, \frac{t_2-t_1}{t_3-t_1}\dist(\gamma(t_3),\eta(t_3))
    \,+\, \frac{t_3-t_2}{t_3-t_1}\dist(\gamma(t_1),\eta(t_1)).
\]
Moreover, $\dist(\gamma(t),\gamma'(t))$ is bounded by $D$. It is therefore constant, equal to some $r \leq D$. 

We must show that $\gamma'$ is a geodesic. Given $s<t$, consider the quadrilateral $Q$ with vertices $\gamma(s), \gamma(t), \gamma'(t), \gamma'(s)$. Since $\gamma'(s)=\pi_A(\gamma(s))$, we have $\angle_{\gamma'(s)}(\gamma(s), \gamma'(t)) \geq \pi/2$. Similarly, $\angle_{\gamma'(t)}(\gamma(t), \gamma'(s)) \geq \pi/2$. Now, note that $\gamma(s)$ is also the closest-point projection of $\gamma'(s)$ to $\gamma$. Indeed, since $\gamma$ is convex, if this were not the case, then there would exist $s' \ne s$ such that $d(\gamma(s'),\gamma'(s)) < r$. But then $d(\gamma(s'),A)<r$, which is a contradiction. Therefore, the angles in $Q$ at $\gamma(s)$ and $\gamma(t)$ are also at least $\pi/2$. By the flat quadrilateral theorem \cite[Thm~II.2.11]{bridsonhaefliger:metric}, all four angles in $Q$ are $\frac\pi2$, and the convex hull of $Q$ is a Euclidean rectangle. In particular, $d(\gamma'(s),\gamma'(t)) = d(\gamma(s),\gamma(t)) = t-s$. Thus $\gamma'$ is a geodesic.
\end{proof}


The following useful proposition gives a way to show that certain quasiflats are Hausdorff close to flats, given information about a spanning collection of quasigeodesics.

\begin{proposition}\label{prop:singular_quasi_flats}
Let $Y$ be a complete CAT(0) space, and let $f\colon \mathbb{R}^k\to Y$ be a $q$-quasiisometric embedding. Let $\gamma_1,\dots,\gamma_k\subseteq \mathbb{R}^k$ be geodesics spanning $\mathbb{R}^k$. For every $D\geq 0$ there exists $D'=D'(q,k,D)$ such that the following holds.
    
Suppose that for every geodesic $\gamma$ that is parallel to some $\gamma_i$, the image $f(\gamma)$ lies at Hausdorff distance at most $D$ from a geodesic in $Y$. Then $f(\mathbb{R}^k)$ lies at Hausdorff distance $D'$ from a $k$--flat $F\subset Y$. 

If $D=0$ and $f$ is biLipschitz, then we can take $D'=0$.

If $Y$ is a finite-dimensional CAT(0) cube complex or the asymptotic cone of one, and if each $f(\gamma_i)$ lies within Hausdorff distance $D$ of a singular geodesic, then $F$ is singular.

If $Y$ is a finite-dimensional CAT(0) cube complex and each $f(\gamma_i)$ lies within Hausdorff distance $D$ of a semisingular geodesic, then $F$ is semisingular.
\end{proposition}

\begin{proof}
We prove the statement by induction on $k$. For $k=1$ there is nothing to prove.

Assume the statement holds in dimension $k-1$, with constant $D''=D'(q,k-1,D)$. Let $H\subset\R^k$ be the $(k-1)$--flat spanned by $\gamma_1,\dots,\gamma_{k-1}$. 
By the induction assumption there exists a $(k-1)$--flat $F\subseteq Y$ such that $\dist_{\mathrm{Haus}}(f(H),F)\leq D''$.
Similarly, if $H'$ is a parallel of $H$, then it is spanned by parallels of $\gamma_1,\dots,\gamma_{k-1}$, so by the induction assumption there exists a $(k-1)$--flat $F'\subseteq Y$ such that $d_{\mathrm{Haus}}(f(H'),F')\leq D''$. As $f$ is a quasiisometric embedding, $F'$ is parallel to $F$.

Let $\alpha\subseteq Y$ be a geodesic satisfying $\dist_{\mathrm{Haus}}(f(\gamma_k),\alpha)\leq D$.
Since every point in $\gamma_{k}$ lies in a parallel of $H$, we have that $f(\gamma_{k})$ is contained in the $D''$--neighbourhood of the parallel set $P(F)$ of $F$. 
Therefore, $\alpha \subseteq P(F)^{+D''+D}$. Since $Y$ is complete and $P(F)$ is closed and convex, \Cref{lem:parallel_geod_in_convex} shows that there exists a geodesic $\alpha' \subseteq P(F)$ parallel to $\alpha$, and at Hausdorff distance $\leq D''+D$ from it. We have 
$$
d_{\mathrm{Haus}}(f(\gamma_{k}),\alpha') \,\leq\, D''+2D.
$$
Recall that the parallel set splits isometrically as 
$ P(F) = F \times T$, for some complete CAT(0) space $T$.
Let $\beta$ be the projection of $\alpha'$ to $T$. It is clear that $\beta$ is unbounded, because $\gamma_{k}$ does not lie in a finite neighbourhood of $H$. Moreover, as the projection of a geodesic in a product to one factor, $\beta$ is a bi-infinite geodesic (after a scalar reparametrisation). Thus, $F\times \beta\subseteq P(F)$ is a $k$--flat.

Next we show that $f(\mathbb{R}^k)\subseteq (F\times \beta)^{+3D''+2D}$. Since $\mathbb{R}^k$ is covered by parallels of $H$, it is enough to show that for every parallel $H'$ of $H$, there is a subflat of $F\times \beta$ such that $f(H')$ lies in the $3D''+2D$ neighbourhood of it.

Let $H'$ be parallel to $H$, and let $\{z\} = H' \cap \gamma_{k}$. Let $F' \subseteq Y$ be a $(k-1)$--flat parallel to $F$ such that $\dist_{\mathrm{Haus}}(f(H'),F') \leq D''$. In particular, $f(z)\in {F'}^{+D''}$. Since $z\in \gamma_k$, we have
$$
f(z) \,\in\, {\alpha'}^{+D''+2D} \,\subseteq\, (F\times \beta)^{+D''+2D}.
$$
Let $y=(x,b)\in \alpha'\subseteq F\times \beta$ be a point such that $d(f(z),y)\leq D''+2D$, and let $F'' = F\times \{b\}$. Then $F''$ is a $(k-1)$--flat contained in $F\times \beta$, parallel to $F$, and $f(z)\in F''^{+D''+2D}$.

The flats $F''$ and $F'$ are parallel, so the Hausdorff distance between them is at most the distance between two points on them. Hence, by comparing with $f(z)$, we get that $\dist_{\mathrm{Haus}}(F', F'')\leq 2D''+2D$. From this we see that $d_{\mathrm{Haus}}(f(H'),F'')\le3D''+2D$. As described above, this shows that $f(\R^k)\subseteq(F\times\beta)^{+3D''+2D}$.

We can now apply \cref{lem:quasiisom_quasi_surj}, which gives a constant $D'=D'(q,k,D,D'')=D'(q,k,D)$ such that $\dist_{\mathrm{Haus}}(f(\mathbb{R}^k),F\times\beta)\leq D'$.

In the case where $D=0$ and $f$ is bilipschitz, by induction $D''=0$. In this case, we showed that for each parallel $H'$ of $H$, the image $f(H'')$ is actually equal to a subflat $F''\subset F\times\beta$ parallel to $F$. Since $f$ is continuous and coarsely surjective, it follows that $f(\R^k)=F\times\beta$.

Now consider the case where $Y$ is a finite-dimensional CAT(0) cube complex or an asymptotic cone of one, and assume that each $f(\gamma_i)$ lies at Hausdorff distance at most $D$ from a singular geodesic. By the induction assumption, the $(k-1)$--flat $F$ is singular and by hypothesis the geodesic $\alpha$ is singular. Consequently, $\alpha'$ is a singular geodesic in the product $P(F)=F\times T$ that does not lie in a finite neighbourhood of $F$, so we have $\beta=\alpha'$. Thus $\beta$ is singular, so $F\times\beta$ is singular.

Finally, suppose that $Y$ is a CAT(0) cube complex and each $f(\gamma_i)$ lies at Hausdorff distance at most $D$ from a semisingular geodesic. Passing to asymptotic cones, we have that $\hat f\colon \R^k\to\hat Y$ is a bilipschitz embedding such that each $\hat f(\hat\gamma_i)$ is a singular geodesic. By the previous case, $\hat f(\R^k)$ is a singular flat. But we also know that $f(\R^k)$ lies at finite Hausdorff distance from $F\times\beta$, so $\widehat{F\times\beta}=\hat f(\R^k)$ is singular, and hence $F\times\beta$ is semisingular.
\end{proof}

\begin{remark} \label{rem:non_uniform_stable_not_good_for_flats}
In Proposition~\ref{prop:singular_quasi_flats}, if one makes the weaker assumption that for every singular geodesic $\gamma$, the image $f(\gamma)$ lies at finite Hausdorff distance from a singular geodesic without a uniform bound, then the conclusion of the proposition can fail, even for CAT(0) cube complexes. 

For example, let $Y=X_{C_6}$ be the universal cover of the Salvetti complex of the right-angled Artin group $A_{C_6}$ on the 6--cycle. Let $s_0,\dots,s_5$ be the standard generators for $A_{C_6}$. There is a quasiflat $Q\subseteq Y$ consisting of a cyclic union of twelve 2--orthants whose boundary geodesic rays are, in order, $\sgen{s_0}^+,\dots,\sgen{s_5}^+,\sgen{s_0}^-,\dots,\sgen{s_5}^-$. If $f\colon \R^2\to Q$ is a quasiisometry sending the $x$--axis to $\sgen{s_0}$ and the $y$--axis to $\sgen{s_3}$, then $f(\R^2)$ does not lie at finite Hausdorff distance from a flat in $Y$.
\end{remark}

The following is really a combination of two statements: in one, every appearance of ``(semi)singular'' is read as ``semisingular'', and in the other every instance is read as ``singular''.

\begin{lemma} \label{lem:orthants_to_semisingular_orthants}
Let $X$ and $Y$ be finite-dimensional CAT(0) cube complexes, and let $f\colon X\to Y$ be a $q$--quasiisometric embedding. Let $E\subset X$ be a (semi)singular $k$--flat. Assume that there exists $D\ge0$ such that for every (semi)singular geodesic $\gamma\subset E$, the image $f(\gamma)$ lies at Hausdorff distance at most $D$ from a (semi)singular geodesic in $Y$.

If $O\subset E$ is a (semi)singular $p$--orthant, for some $p\le k$, then $f(O)$ lies at Hausdorff distance at most $D'=D'(q,k,p,D,\dim Y)$ from a (semi)singular $p$--orthant of $Y$.
\end{lemma}

\begin{proof}
We shall prove the lemma by induction on the pair $(k,p)$ with the lexicographic order. The base cases are the pairs $(k,1)$. In such a case, $O\subset E$ is a (semi)singular geodesic ray. Let $\gamma$ be the unique geodesic in $E$ that contains $O$. Since $E$ and $O$ are (semi)singular, so is $\gamma$. By assumption, $f(\gamma)$ lies at Hausdorff distance at most $D$ from a (semi)singular geodesic $\gamma$. We can take $O'\subset\gamma'$ to be a subray at Hausdorff distance at most $D$ from $f(O)$.

Now suppose that we have a pair $(k,p)$ with $k,p\ge2$, and suppose that we have established the lemma for all pairs $(k',p')$ with $k'<k$ or with $k'=k$ and $p'<p$. Let $o$ denote the cone point of $O$, and let $\Omega\subset E$ be the $p$--flat containing $O$. Let $O_1,\dots,O_p$ denote the $(p-1)$--orthants that bound $O$, and let $\Omega_i\subset\Omega$ denote the $(p-1)$--flat containing $O_i$. Again, the fact that $E$ and $O$ are (semi)singular implies that each of $\Omega$, $O_i$, and $\Omega_i$ is (semi)singular.

According to \cref{prop:singular_quasi_flats}, there is a constant $D'=D'(q,p,D)$ such that $f(\Omega)$ lies at Hausdorff distance at most $D'$ from a (semi)singular $p$--flat $\Omega'$, and moreover each $f(\Omega_i)$ lies at Hausdorff distance at most $D'$ from a (semi)singular $(p-1)$--flat $\Omega'_i$. \cref{rem:boundary_semisingular} tells us that $\partial_T\Omega'$ is isomorphic to $\partial_T\R^p$, and the $\partial_T\Omega'_i$ are subcomplexes whose union contains all vertices of $\partial_T\Omega'$. By the inductive hypothesis, each $f(O_i)$ lies at Hausdorff distance at most $D'(q,k,p-1,D,\dim Y)$ from a (semi)singular orthant $O'_i\subset\Omega'$. Moreover, we have that $\bigcup_{i=1}^p\partial_TO'_i$ is the boundary of one of the $(p-1)$--cells $C\subset\partial_T\Omega'$. Also note that every $O'_i$ contains a point at distance at most $D'$ from $f(o)$.

Let $O'$ be an orthant of $\Omega'$ with cone point at distance at most $D'\dim Y$ from $f(o)$ and with Tits boundary equal to $C$. If the $O'_i$ are singular, then $O'$ is singular. Otherwise, since $C\subset\partial_T\Omega'$ is a $(p-1)$--cell, the ultralimit of $O'$ in $\hat Y$ is a singular $p$--orthant, so $O'$ is semisingular. We shall prove that $f(O)$ lies at uniformly bounded Hausdorff distance from $O'$.


The union $\bigcup_{i=1}^pO_i$ separates $O$ from $\Omega\ssm O$, and hence $\bigcup_{i=1}^pf(O_i)$ must uniformly coarsely separate $f(O)$ from $f(\Omega\ssm O)$, with constant depending only on $q$. From this it follows that $\bigcup_{i=1}^p\Omega'_i$ is a union of orthants that uniformly coarsely separates $f(O)$ from $\Omega'\ssm f(O)$, where the constant now depends on $D$, $\dim Y$, and $D'(q,k,p-1,D,\dim Y)$ as well. By the choice of $O'$, this shows that $f(O)$ is contained in a uniform neighbourhood of $O'$. Applying the same argument with a quasiinverse $\hat f\colon \Omega'\to\Omega$ bounds the Hausdorff distance between $f(O)$ and $O'$. This completes the proof.
\end{proof}

\subsection{2--orthants} \label{subsec:non_uniform_case}

The following can be thought of as a variant of \cite[Prop.~2.1]{leeb:characterization} for singular orthants.

\begin{lemma} \label{lem:orthants_are_weyl_cones}
Let $X$ be a finite-dimensional CAT(0) cube complex, and suppose that $\xi_1,\xi_2\in\partial_TX$ are represented by singular geodesics. If $\xi_1\ne\xi_2$, then $\angle(\xi_1,\xi_2)\ge\frac\pi2$, and if $\angle(\xi_1,\xi_2)=\frac\pi2$, then there is a 2--orthant subcomplex $O\subseteq X$ such that $\partial_TO$ is an arc of length $\frac\pi2$ from $\xi_1$ to~$\xi_2$.
\end{lemma}

\begin{proof}
Let $\gamma_i$ be a singular geodesic ray in $X$ that represents $\xi_i$ for $i=1,2$. By replacing $\gamma_i$ with a parallel copy, we can assume that $\gamma_i$ lies in the 1--skeleton of $X$ and begins at some vertex $x_i$. We can also make a choice of $x_1$ and $x_2$ that minimises $\dist_1(x_1,x_2)$. There is a unique hyperplane $h_{i,j}$ separating $\gamma_i(j-1)$ from $\gamma_i(j)$ for each $j$. Since $\gamma_1$ and $\gamma_2$ are singular, no two $h_{1,j}$ cross, and no two $h_{2,j}$ cross.

Observe first that if there exist $j$ and $k'$ such that $h_{1,j}$ does not cross any $h_{2,k}$ with $k\geq k'$, then $\angle(\xi_1,\xi_2)>\frac\pi2$. 
Indeed, let $p=\gamma_1(1+j)$. The geodesic ray $[p,\xi_2)$ must cross $h_{1,j}$, otherwise $\gamma_2(k)$ and $[p,\xi_2)$ would be contained in different halfspaces separated by $h_{1,j}$ for every $k\geq k'$. As $k$ grows, $\gamma_2(k)$ gets unboundedly far from $h_{1,j}$, in contradiction to the fact that $\gamma_2$ and $[p,\xi_2)$ converge to the same boundary point.
Since $[p,\xi_2)$ crosses $h_{1,j}$, we have $\angle(\xi_1,\xi_2)\geq \angle_p(\xi_1,\xi_2)>\frac\pi2$. 

Suppose that $x_1\ne x_2$. Let $C$ denote the smallest convex subcomplex of $X$ that contains $x_1$ and $x_2$, which can equivalently be described as the convex subcomplex obtained by intersecting all halfspaces that contain both $x_1$ and $x_2$. By definition, every hyperplane dual to an edge of $C$ separates $x_1$ from $x_2$. Since $\dist_1(x_1,x_2)<\infty$, there are finitely many such hyperplanes. Additionally, the fact that $\gamma_i$ is singular and $\dist(x_1,x_2)$ is minimal implies that the projection of $\gamma_i$ to $C$ is $x_i$, so every hyperplane dual to an edge of $C$ separates $\gamma_1$ from $\gamma_2$.

If $h'$ is a hyperplane of $C$ dual to an edge containing $x_1$, then $h'$ must cross $h_{11}$, for otherwise we could extend $\gamma_1$ to reduce $\dist_1(x_1,x_2)$.
Moreover, $h'$ cannot cross every $h_{1,j}$, for then we could move $\gamma_1$ to a parallel ray to reduce $\dist_1(x_1,x_2)$. Thus there exists some $j$ such that none of the finitely many hyperplanes dual to edges in $C$ that contain $x_1$ cross $h_{1,j}$.
Because the projection of $\gamma_1$ to $C$ is $x_1$, each hyperplane of $C$ not adjacent to $x_1$ is separated from $\gamma_1$ by a hyperplane of $C$ adjacent to $x_1$. Thus, none of the hyperplanes dual to $C$ can cross $h_{1,j}$. 
Similarly, there exists $k'$ such that none of the hyperplanes dual to $C$ can cross $h_{2,k'}$. We have that $h_{1,j}$ cannot cross any $h_{2,k}$ with $k\geq k'$, because they are separated by the hyperplanes dual to $C$. Thus, as observed above, $\angle(\xi_1,\xi_2)>\frac{\pi}{2}$.


Now suppose instead that $x_1=x_2=x$. As $[x,\xi_1)$ and $[x,\xi_2)$ are singular, $\angle_x(\xi_1,\xi_2)\in\{0,\frac\pi2\}$. As $\xi_1\neq \xi_2$, there exists $t\in \mathbb{N}$ such that $\gamma_1(t)=\gamma_2(t)$ and $\gamma_1(t+1)\neq \gamma_2(t+1)$. The angle between $\xi_1$ and $\xi_2$ at $\gamma_1(t)=\gamma_2(t)$ is not 0, so it must be $\frac{\pi}{2}$. We obtain the desired inequality because $\angle(\xi_1,\xi_2)\geq \angle_{\gamma_1(t)}(\xi_1,\xi_2)=\frac{\pi}{2}$.

We have shown that $\angle(\xi_1,\xi_2)\ge\frac\pi2$. Moreover, if $\angle(\xi_1,\xi_2)=\frac\pi2$, then we can take $\gamma_1$ and $\gamma_2$ to start at a common point $x$ such that $\angle_x(\xi_1,\xi_2)=\frac\pi2$. With this choice, if $h_{1,j}$ does not cross ${h_{2,k'}}$ for some $j$ and $k'$, then $h_{1,j}$ does not cross $h_{2,k}$ for any $k\geq k'$, because $h_{2,k'}$ separates $h_{1,j}$ from $h_{2,k}$. As observed above, this would imply that $\angle(\xi_1,\xi_2)>\frac{\pi}{2}$.
Hence $h_{1,j}$ crosses $h_{2,k}$ for every $j,k\in \mathbb{N}$ and so the singular geodesic rays $\gamma_1$ and $\gamma_2$ span an orthant whose Tits boundary is an arc of length $\frac\pi2$ from $\xi_1$ to $\xi_2$.
\end{proof}


\begin{proposition}\label{prop:2_orthant_in_target}
Let $Y$ be a finite-dimensional CAT(0) cube complex. Let $Y_\omega$ be an asymptotic cone of $Y$ with fixed basepoint. Let $E=\mathbb R^2$, equipped with the standard cubulation, and let $f\colon E\to Y$ be a $q$-quasiisometric embedding. Assume that:
\begin{enumerate}
    \item for every singular geodesic $\gamma\subseteq E$, the image $f(\gamma)$ lies at finite Hausdorff distance from a singular geodesic in $Y$;
    \item the ultralimit $(f(E))_\omega$ is a $2$--flat in $Y_\omega$.
\end{enumerate}
Then for every singular $2$--orthant $Q\subseteq E$, the two axis rays of $Q$ are mapped by $f$ within finite Hausdorff distance of the two axis rays of a singular $2$--orthant of $Y$.
\end{proposition}

\begin{proof}
Let $Q\subseteq E$ be a singular $2$--orthant, and let $\alpha_E^+$ and $\beta_E^+$ be its axis rays, with common initial point $v$. Extend $\alpha_E^+$ and $\beta_E^+$ to singular geodesics $\alpha_E$ and $\beta_E$ of $E$. Since asymptotic cones of $Y$ with fixed basepoints and the same scaling sequence are canonically isometric, we may assume that the fixed basepoint is $y=f(v)$. 

Let $\alpha_Y$ and $\beta_Y$ be singular geodesics of $Y$ at finite Hausdorff distance from $f(\alpha_E)$ and $f(\beta_E)$, respectively. Let $\xi^+, \xi^- \in \partial_TY$ be the endpoints of $\alpha_Y$, chosen so that $\xi^+$ is the endpoint of the subray at finite Hausdorff distance from $f(\alpha_E^+)$, and let $\eta \in \partial_T Y$ denote the endpoint of the subray of $\beta_Y$ at finite Hausdorff distance from $f(\beta_E^+)$. Since $f(\beta_E^+)$ is at infinite Hausdorff distance from $\alpha_Y$, the points $\xi^+, \xi^-, \eta$ are pairwise distinct. 

    
Set $\hat F=(f(E))_\omega$, and let $o=(y)\in Y_\omega$. First notice that, since $\alpha_Y$ and $\beta_Y$ lie at finite Hausdorff distance from $f(E)$, their ultralimits $(\alpha_Y)_\omega$ and $(\beta_Y)_\omega$ belong to the $2$--flat $\hat F$. Since $\alpha_Y$ and $\beta_Y$ are singular, so are $(\alpha_Y)_\omega$ and $(\beta_Y)_\omega$. Moreover, as $v\in \alpha_E\cap\beta_E$, we have that $o\in (\alpha_Y)_\omega\cap (\beta_Y)_\omega\subset\hat F$. Let $\varphi_o:(\partial_T Y,\angle)\to (\Sigma_{o} \hat Y,\angle_{o})$ be the isometric embedding given by \Cref{lem:embedding-Titsboundary-into_link_and_Tits_boundary}. Since $(\alpha_Y)_\omega$ and $(\beta_Y)_\omega$ are singular geodesics through $o$ contained in the singular flat $\hat F$, we must have $\angle_{o}((\alpha_Y^{\pm})_\omega,(\beta_Y^{\pm})_\omega)\in\{0,\frac\pi2,\pi\}$. But $\xi^+$, $\xi^-$, and $\eta$ are pairwise distinct, so $(\alpha_Y)_\omega$ and $(\beta_Y)_\omega$ must meet at angle $\pi/2$. In particular,
$$
\angle_{o}\bigl(\varphi_o(\xi^-),\varphi_o(\eta)\bigr)=\angle_{o}\bigl(\varphi_o(\xi^+),\varphi_o(\eta)\bigr)=\pi/2.
$$
Since $\varphi_o$ is an isometric embedding, it follows that $\angle(\xi^+, \eta) = \pi/2$. Therefore, by \cref{lem:orthants_are_weyl_cones}, there is a singular 2--orthant in $Y$ whose boundary rays are at finite Hausdorff distance from $f(\alpha_E)$ and $f(\beta_E)$.
\end{proof}

We finish this section by proving an analogue of \cref{lem:orthants_are_weyl_cones} for semisingular orthants. It will not be used in the proofs of our main theorems, but can be used to give an alternative definition to \cref{def:singular_boundary_graph}.

\begin{lemma} \label{lem:asymptotic_Weyl_cone}
Let $X$ be a finite-dimensional CAT(0) cube complex, let $\hat X$ be an asymptotic cone with respect to a fixed basepoint. Suppose that $\xi_1,\xi_2\in\partial_TX$ are represented by semisingular geodesic rays. 

If $\xi_1\ne\xi_2$, then $\angle(\xi_1,\xi_2)\ge\frac\pi2$. If $\angle(\xi_1,\xi_2)=\frac\pi2$ then there is a singular orthant $\hat O\subset\hat X$ such that $\partial_T\hat O$ is an arc of length $\frac\pi2$ from $\varphi_T(\xi_1)$ to $\varphi_T(\xi_2)$ in $\partial_T(\hat X)$.
\end{lemma}

\begin{proof}
Let $\hat X$ be an asymptotic cone of $X$ with fixed basepoint $o$. By \cref{lem:embedding-Titsboundary-into_link_and_Tits_boundary}, there is an isometric embedding $\varphi_T\colon \partial_TX\to\partial_T\hat X$. If $\gamma_1$ and $\gamma_2$ are semisingular geodesic rays in $X$ that represent $\xi_1$ and $\xi_2$, respectively, then their ultralimits $\hat\gamma_1\subset\hat X$ and $\hat\gamma_2\subset\hat X$ are singular geodesic rays that represent $\varphi_T\xi_1$ and $\varphi_T\xi_2$, respectively. Moreover, $\hat\gamma_1$ and $\hat\gamma_2$ emanate from the same point $\hat o=(o)\in\hat X$. Since they are both singular, we must have $\angle_{\hat o}(\hat\gamma_1,\hat\gamma_2)\in\{0,\frac\pi2,\pi\}$.

\cref{lem:embedding-Titsboundary-into_link_and_Tits_boundary} also gives an isometric embedding $\varphi_o:\partial_TX\to\Sigma_{\hat o}\hat X$, where $\Sigma_{\hat o}\hat X$ is the link at $\hat o\in\hat X$. In particular, if $\xi_1\ne\xi_2$, then the angle $\angle_{\hat o}(\hat\gamma_1,\hat\gamma_2)$ is positive, and consequently is either $\frac\pi2$ or $\pi$. Hence $\angle(\xi_1,\xi_2)=\angle(\varphi_T(\xi_1),\varphi_T(\xi_2))\ge\frac\pi2$.

Now suppose that $\angle(\xi_1,\xi_2)=\frac\pi2$. Let $(\lambda_n)$ be the scaling sequence for the asymptotic cone $\hat X$. Since $\hat\gamma_1$ and $\hat\gamma_2$ are singular, we have that $\mu(\hat\gamma_1(t),\hat o,\hat\gamma_2(t))=\hat o$ for all $t\ge0$. Furthermore, since $\hat\gamma_1$ and $\hat\gamma_2$ are singular and $\angle_{\hat o}(\hat\gamma_1,\hat\gamma_2)=\frac\pi2$, there must exist $\eps>0$ such that $\hat\gamma_1(\eps)$, $\hat\gamma_2(\eps)$, and $\hat o$ are three vertices of a square in $\hat X$. Let $\hat x$ be the fourth vertex. We have 
\begin{align}
\mu(\hat o,\hat x,\hat\gamma_i(\eps)) \,=\, \hat\gamma_i(\eps), 
    \quad \mu(\hat\gamma_1(\eps),\hat x,\hat\gamma_2(\eps)) \,=\, \hat x. \label{eq:square}
\end{align}
In other words, if $(x_n)$ is a sequence that represents $\hat x$, then
\begin{align}\begin{split}
\lim_\omega \frac1{\lambda_n}
    \dist\big(\mu(o,x_n,\gamma_i(\lambda_n\eps)),\,\gamma_i(\lambda_n\eps)\big) \,=\, 0, \\
\lim_\omega \frac1{\lambda_n}
    \dist\big(\mu(\gamma_1(\lambda_n\eps),x_n,\gamma_2(\lambda_n\eps)),\,x_n) \,=\, 0. 
    \label{eq:limit}
\end{split}\end{align}

Any point $\hat x$ that satisfies the identities in~\eqref{eq:square} spans a square with $\hat o,\hat a_1,\hat a_2$. Moreover, it follows from the five-point condition of median algebras that for any two such squares, one contains the other (this is analogous to the fact that median graphs cannot contain $K_{2,3}$ as a subgraph; see for example the proof of \cite[Lem.~4.1]{munropetyt:coarse}). Thus, in order to find a singular orthant in $\hat X$ spanned by $\hat\gamma_1$ and $\hat\gamma_2$, it suffices to show that they have arbitrarily long initial subsegments that span squares, which we will do by showing they satisfy \cref{eq:square}.

Let $(x_n)$ be a sequence representing $\hat x$. Given $t>0$, there is a sequence $(k_n)$ such that $\lim_\omega\frac{\lambda_{k_n}}{\lambda_n}=t$. Consider the point $\hat x_t=(x_{k_n})$. Using~\eqref{eq:limit}, we can derive
\begin{align*}
\lim_\omega&\frac1{\lambda_n}
        \dist\big(\mu(o,x_{k_n},\gamma_i(\lambda_{k_n}\eps)),\gamma_i(\lambda_{k_n}\eps)\big) 
    \,=\, t\lim_\omega \frac1{\lambda_{k_n}}
        \dist\big(\mu(o,x_{k_n},\gamma_i(\lambda_{k_n}\eps)),\gamma_i(\lambda_{k_n}\eps)\big) \,=\, 0,
\end{align*}
and similarly for the other expression. This shows that $\hat x_t$ satisfies the identities in~\eqref{eq:square}. By taking increasingly large values of $t$, we see that $\hat\gamma_1$ and $\hat\gamma_2$ have arbitrarily long initial subsegments that span squares. As described above, this completes the proof.
\end{proof}

Note that there are CAT(0) square complexes that contain two semisingular geodesics at angle $\frac\pi2$ but do not contain any orthants; see \cref{eg:singular_boundary}. On the other hand, \cref{lem:orthants_to_semisingular_orthants} rules out this type of behaviour in certain quasiflats.


\section{Structures inherited by intersections of flats} \label{sec:structure_quasi_flats}

In this section, we are interested in the images of intersections of flats in CAT(0) spaces under quasiisometric embeddings. The goal is to show that rigidity of flats passes down to their intersections. More precisely, \cref{prop:intersection_orthants} says that if some sets have images Hausdorff-close to unions of singular orthants, then the image of their intersection is also Hausdorff-close to a union of singular orthants.
If, in addition, the intersection is a flat, then \cref{prop:boundary_quasi_flat_sphere}, gives us that the boundary of the union of orthants is a sphere.

This will later be applied in situations where we have already established good control on top-dimensional flats, by using \cref{thm:fully_branching_general}, for example, and want to gain similar control over their intersections.

The proof of \cref{prop:intersection_orthants} is a combination of the following two slogans: ``the image of the intersection is close to the intersection of the images'' (\cref{cor:distance_from_intersection}), and  ``a coarse intersection of orthants is essentially an orthant'' (\cref{lem:intersection_weyl_cones_is_cone}).

The bound on Hausdorff distance in \cref{prop:intersection_orthants} is finite but not uniform. A statement with uniform bounds is given in \cref{prop:intersection_semisingular}: if some flats intersect in a flat and their images are uniformly Hausdorff-close to flats, then the image of their intersection is also uniformly Hausdorff-close to a flat.

\subsection{The Tits boundary of certain quasiflats}

In this subsection, we show that if a quasiflat lies at finite Hausdorff distance from a finite union of orthants, then the Tits boundary of that union is homeomorphic to a sphere. 

The following lemma is a reformulation of \cite[Thm~4.4]{bjorner:let}. We include a direct proof for the reader's convenience.

\begin{lemma}\label{lem:Bjorner}
Let $K$ be a finite, $n$-dimensional simplicial complex. Assume that $K$ is homotopy equivalent to $\mathbb S^n$, and that for every simplex $\sigma \in K$ of dimension $d<n$, the link $\operatorname{Lk}_K(\sigma)$ is homotopy equivalent to $\mathbb S^{n-d-1}$. Then $K$ is homeomorphic to $\mathbb S^n$.
\end{lemma}

\begin{proof}
The proof is by induction on $n$. The case $n=0$ is immediate. 

Assume $n\ge 1$, and that the statement is true for all smaller dimensions. Let $\sigma$ be a simplex of $K$, with $\dim(\sigma)=d<n$, and set $L=\operatorname{Lk}_K(\sigma)$. By hypothesis, $L$ is homotopy equivalent to $\mathbb S^{n-d-1}$. If $\tau$ is a simplex of $L$, then
$$
\operatorname{Lk}_L(\tau) \,=\, \operatorname{Lk}_K(\sigma * \tau),
$$
so every link in $L$ is again homotopy equivalent to a sphere of the right dimension. Therefore $L$ satisfies the same assumptions as $K$, but in smaller dimension. By induction, $L$ is homeomorphic to $\mathbb S^{n-d-1}$. In particular, every vertex link of $K$ is homeomorphic to $\mathbb S^{n-1}$, hence the open star of every vertex is homeomorphic to $\mathbb R^n$. Since every point of $K$ lies in the open star of some vertex, $K$ is a closed topological $n$-manifold. Since $K$ is homotopy equivalent to $\mathbb S^n$, it is homeomorphic to $\mathbb S^n$ by the $n$--dimensional Poincaré conjecture.
\end{proof}

We expect that the following statement remains true for more general simplicial complexes, but it is sufficient for our purposes as-is. 

\begin{proposition}\label{prop:Euclid_cone_bilip_Rn}
Let $K$ be a finite simplicial complex, equipped with a piecewise-spherical simplicial metric. If the Euclidean cone $C(K)$ over $K$ is biLipschitz equivalent to $\mathbb{R}^n$, then $K$ is homeomorphic to $\mathbb S^{n-1}$.
\end{proposition}

\begin{proof}
Since each simplex of $K$ is spherical, the Euclidean cone over each simplex is isometric to a Euclidean sector. Hence $C(K)$ is obtained by gluing finitely many Euclidean sectors along subsectors. Choose a triangulation $T$ of $C(K)$ compatible with this decomposition. Then $T$ is a locally finite simplicial complex with a piecewise Euclidean metric, and the link $L:=\operatorname{Lk}_T(v)$ is a triangulation of $K$ for every vertex $v\in T$. In particular, $L$ is homeomorphic to $K$. By hypothesis, $C(K)$ is biLipschitz equivalent to $\mathbb{R}^n$. Therefore $T$ is a Lipschitz $n$-manifold. By \cite[Thm~1]{siebenmannsullivan:oncomplexes} $L$ is homotopy equivalent to $\mathbb S^{n-1}$, and the link of each simplex of $L$ is homotopy equivalent to a sphere of the appropriate dimension. By \Cref{lem:Bjorner}, $L$, and therefore $K$, is homeomorphic to $\mathbb S^{n-1}$.
\end{proof}

Let $X$ be a finite-dimensional CAT(0) cube complex. If a subset $A \subseteq X$ lies at finite Hausdorff distance from a finite union of orthants in $X$, then we denote by $\partial_T A$ the union of the Tits boundaries of these orthants. This is well defined, since whenever $A$ lies at finite Hausdorff distance from two finite unions $\bigcup_{i=1}^m O_i$ and $\bigcup_{j=1}^{\ell} O'_j$ of orthants, one has $\bigcup_{i=1}^m \partial_T O_i=\bigcup_{j=1}^{\ell} \partial_T O'_j$.

The following proposition applies to every $n$--dimensional quasiflat in an $n$--dimensional CAT(0) cube complex, by \cite[Thm~1.1]{huang:top}. Note that the orthants in its statement are not required to be singular. 

\begin{proposition}\label{prop:boundary_quasi_flat_sphere}
Let $X$ be a finite-dimensional CAT(0) cube complex, and let $f\colon  \mathbb R^k \to X$ be a quasiisometric embedding. If $f(\mathbb R^k)$ lies at finite Hausdorff distance from a finite union of orthants of $X$, then $(\partial_T (f(\mathbb R^k)),\angle)$ is homeomorphic to $\mathbb S^{k-1}$.
\end{proposition}

\begin{proof}
Let $Q = \cup_{i=1}^n O_i$ be a union of orthants at finite Hausdorff distance from $f(\mathbb R^k)$. By a volume-growth argument, $\dim O_i\leq k$ for every $i$. If there exists $i$ such that $\dim O_i < k$, then $O_i$ must be coarsely contained in the union of the others, for otherwise arbitrarily large balls of $\mathbb R^k$ would be mapped into a uniform neighbourhood of $O_i$, contradicting volume-growth again. This also implies that $\partial_T O_i \subseteq \cup_{j\neq i}\partial_T O_j$. Thus we may assume that $\dim O_i=k$ for all $i$, in which case every $\partial_T O_i$ is a spherical $(k-1)$--simplex.

Let $\hat X$ be an asymptotic cone of $X$ with respect to some fixed basepoint, and let $\hat Q$ be the ultralimit of $Q$. Let $o \in\hat X$ denote the ultralimit of the basepoint. Each ultralimit $\hat O_i$ is a $k$--orthant based at $o$, and $\hat Q=\bigcup_{i=1}^n\hat O_i$. In particular, $\partial_TQ$ and $\partial_T\hat Q$ are isometric, so $\hat Q$ is isometric to the Euclidean cone $C(\partial_TQ)$. Because $f$ induces a biLipschitz map of asymptotic cones, $\mathbb R^k$ and $\hat Q$ are biLipschitz equivalent. Using \Cref{prop:Euclid_cone_bilip_Rn}, we deduce that $\partial_TQ$ is homeomorphic to $\mathbb S^{k-1}$, and the conclusion follows because $\partial_T(f(\mathbb R^k)) = \partial_TQ$.
\end{proof}

\subsection{Approximating images of intersections of flats} \label{sec:coarse_intersection}

Here we prove the two main statements of the section, showing that the images of intersections of certain subsets of a CAT(0) cube complex under a quasiisometric embedding are Hausdorff-close to unions of orthants or flats. We start with a simple consequence of the Helly property.


\begin{lemma}\label{prop:coarse_intersection_subcomplexes_ccx}
Let $X$ be an $n$--dimensional CAT(0) cube complex. Let $F_1, \cdots, F_s$ be convex subcomplexes such that $H= \bigcap_{i=1}^s F_i$ is nonempty. For every $D \geq 0$ we have
\[
\bigcap_{i=1}^s F_i^{+D} \,\subseteq\, H^{+D \sqrt{n}}.
\]
\end{lemma}

\begin{proof}
Let $d_1$ denote the median metric on $X$, and let $d_\infty$ denote the injective metric on $X$, see \Cref{subsec:ccc} and \cref{rem:d_infty}. Since $X$ is $n$--dimensional, one has $d_\infty \leq d \leq \sqrt nd_\infty$. Every ball in $(X,d_\infty)$ is convex for $d_1$.

If $x\in\bigcap_{i=1}^sF_i^{+D}$, then for every $i$ the intersection $B_{d_\infty}(x,D)\cap F_i$ is nonempty. Hence $\{B_{d_\infty}\left(x,D\right), F_1, \cdots, F_s\}$ is a family of pairwise-intersecting convex subsets of the median space $(X,d_1)$. By the Helly property, their total intersection is nonempty; see \cite[Thm~2.2]{roller:poc}. In particular, $B_{d_\infty}\left(x,D\right) \cap H \ne \varnothing$. Therefore $x$ is in the $D$--neighbourhood of $H$ for the metric $d_\infty$. Since $d \leq \sqrt nd_\infty$, we conclude that $x \in H^{+D\sqrt n}$.
\end{proof}

\begin{proposition} \label{cor:distance_from_intersection}
Let $X$ be a finite-dimensional CAT(0) cube complex. Let $F_1, \cdots, F_s$ be convex subcomplexes such that $H= \bigcap_{i=1}^s F_i$ is nonempty. Let $Y$ be any metric space. If $f\colon X \to Y$ is a $q$--quasiisometric embedding, then for each $D\ge0$ there exists $T=T(q,D,\dim X)\geq 0$ such that
\[
f(H) \;\subseteq\; \bigcap_i f(F_i) \;\subseteq\; \bigcap_i f(F_i)^{+D} \;\subseteq\; f(H)^{+T}.
\]
If $f$ is biLipschitz and $D=0$, then we can take $T=0$.
\end{proposition}

\begin{proof}
The first two inclusions are immediate. For the last one, let $y \in \bigcap_i f(F_i)^{+D}$. For each $i$, choose $x_i \in F_i$ such that $d_Y(y,f(x_i)) \leq D$. Since $f$ is a $q$--quasiisometric embedding, we have $\dist(x_i,x_j)\le2qD+q^2$ for all $i,j$. In particular, $x_1 \in \bigcap_i F_i^{+2qD+q^2}$. It follows from \Cref{prop:coarse_intersection_subcomplexes_ccx} that $x_1 \in H^{+(2qD+q^2)\sqrt{\dim X}}$. Therefore, $f(x_1) \in f(H)^{+q(2qD+q^2)\sqrt{\dim X}+q}$. This suffices, because $d(f(x_1),y)\le D$. Note that if $f$ is bilipschitz and $D=0$, then the constant $T$ given by this argument is 0.
\end{proof}

Next we show that a coarse intersection of singular orthants is coarsely a singular orthant.

\begin{lemma}\label{lem:intersection_weyl_cones_is_cone}
Let $X$ be a finite-dimensional CAT(0) cube complex, and let $O_1,O_2\subseteq X$ be singular orthants of dimensions $p$ and $q$, respectively. For every $s,t \geq 0$, the coarse intersection $O_1^{+s} \cap O_2^{+t}$ is either empty, bounded, or at finite Hausdorff distance from a singular orthant of dimension $k \leq \min\{p,q\}$.
\end{lemma}

\begin{proof}
Let $O_1$ and $O_2$ be singular orthants, and assume that $O_1^{+s}\cap O_2^{+t}$ is unbounded. There are diverging sequences $y_n \in O_1$ and $y_n' \in O_2$ such that $d(y_n,y_n') \leq s+t$ for all $n$. Since $O_1$ and $O_2$ are proper, after passing to subsequences, we may assume that $(y_n)$ and $(y_n')$ converge at infinity. Because they stay at bounded distance, they converge to the same point of $\partial_TX$. Thus, if $\sigma = \partial_T O_1$ and $\sigma' = \partial_T O_2$, then $\sigma \cap \sigma' \ne \varnothing$.  

Let $x \in O_1^{+s} \cap O_2^{+t}$. Choose $z_0 \in O_1$ and $z' \in O_2$ such that $d(x,z_0) \leq s$ and $d(x,z') \leq t$, and let $z \in O_1$ be a vertex at distance at most $\sqrt{p}$ from $z_0$. Let $Q \subseteq O_1$ be the union of all geodesic rays $[z,\xi)$, where $\xi \in \sigma \cap \sigma'$. Since $O_1$ and $O_2$ are singular, $\sigma \cap \sigma'$ is a face of the simplex $\sigma$, so $Q$ is a singular suborthant of $O_1$. We claim that $O_1^{+s} \cap O_2^{+t}$ lies at finite Hausdorff distance from $Q$.

Given $\xi\in\sigma\cap\sigma'$ and a point $w\in X$, let $\gamma_w$ be the geodesic ray from $w$ to $\xi$. By convexity of the distance function, we have $\dist_{\mathrm{Haus}}(\gamma_{z_0},\gamma_x)\le s$ and $\dist_{\mathrm{Haus}}(\gamma_x,\gamma_{z'})\le t$. In particular, $\gamma_x\subset O_1^{+s}\cap O_2^{+t}$. We also have $\dist_{\mathrm{Haus}}(\gamma_x,\gamma_z)\le s+\sqrt p$. This shows that
\[
Q \,\subset\, \big(O_1^{+s}\cap O_2^{+t}\big)^{+s+\sqrt p}.
\]

We now wish to show that $O_1^{+s}\cap O_2^{+t}$ lies in a finite neighbourhood of $Q$. Since $Q$ is parallel to a face of $O_1$, there is a singular orthant $Q'\subset O_1$ based at $z$ such that $Q\times Q'\subset O_1$ is $r$--coarsely dense in $O_1$ for some $r$. If $O_1^{+s}\cap O_2^{+t}$ is not contained in any finite neighbourhood of $Q$, then there is a sequence $(a_n)\subset (Q\times Q')\cap O_2^{+s+t+r}$ such that $\dist(a_n,Q)>n$. 

Writing $z=(z_Q,z_{Q'})\in Q\times Q'$ and $a_n=(a_{n,Q},a_{n,Q'})\in Q\times Q'$, consider the point $b_n=(z_Q,a_{n,Q'})$. It satisfies $b_n=\mu(z,b_n,a_n)$, where $\mu$ denotes the median on $X$. Let $a'_n\in O_2$ be a point with $\dist(a_n,a'_n)\le s+t+r$. Since $\mu$ is 1--Lipschitz in each coordinate, we have 
\[
\dist(b_n,\mu(z',b_n,a'_n)) \,\le\, \dist(z,z')+\dist(a_n,a'_n) \,\le\, (s+t+\sqrt p)+(s+t+r).
\]
As $O_2$ is a convex subcomplex, we have $\mu(z',b_n,a'_n)\in O_2$, and so $b_n\in O_2^{+2s+2t+r+\sqrt p}$.

Because $Q\times Q'$ is proper, the sequence $(b_n)$ subconverges to a point $\zeta\in\partial_T(Q\times Q')=\sigma$. As the $a_n$ get arbitrarily far from $Q$, the choice of $b_n$ ensures that $\zeta\notin\partial_TQ=\sigma\cap\sigma'$. But $b_n\in O_1\cap O_2^{+2s+2t+r+\sqrt p}$ for all $n$, so $\zeta\in\sigma\cap\sigma'$. This is a contradiction.
\end{proof}

We can now prove the first of our results about the images of intersections under quasiisometric embeddings.

\begin{proposition}\label{prop:intersection_orthants}
Let $X$ and $Y$ be finite-dimensional CAT(0) cube complexes, and let $f \colon X \to Y$ be a $q$--quasiisometric embedding. Let $F_1, \dots, F_s \subseteq X$ be convex subcomplexes such that $H=\bigcap_{i=1}^s F_i$ is nonempty.

If $f(F_i)$ lies at finite Hausdorff distance from a finite union of singular orthants for all $i$, then $f(H)$ lies at finite Hausdorff distance from a finite union of singular orthants.
\end{proposition}

\begin{proof}
For each $i$, let $O_1^i,\dots,O_{p_i}^i$ be singular orthants of $Y$ and let $D\geq 0$ be such that $d_{\mathrm{Haus}} \left(f(F_i),\bigcup_{i=1}^{p_i} O_i\right)\le D$. It follows that
\begin{equation}
\begin{split}
   & f(H)  \,\subseteq\, \bigcap_{i=1}^s f(F_i) \,\subseteq\,
\bigcup_{j_1=1}^{p_1}\cdots  \ \bigcup_{j_s=1}^{p_s} \left( {(O_{j_1}^1)}^{+D} \cap \cdots \cap {(O_{j_s}^s)}^{+D} \right) \,\subseteq\, \bigcap_{i=1}^sf(F_i)^{+2D}.
\end{split}
\end{equation}
By \cref{cor:distance_from_intersection}, $f(H)$ and  $\bigcap_{i=1}^s f(F_i))$ are Hausdorff-close.
Therefore, $f(H)$ is Hausdorff-close to
$$
\bigcup_{j_1=1}^{p_1}\cdots  \ \bigcup_{j_s=1}^{p_s} \left( {(O_{j_1}^1)}^{+D} \cap \cdots \cap {(O_{j_s}^s)}^{+D} \right).
$$
By \cref{lem:intersection_weyl_cones_is_cone}, each term ${(O_{j_1}^1)}^{+D} \cap \cdots \cap {(O_{j_s}^s)}^{+D}$ is either empty, bounded, or Hausdorff-close to a singular orthant. Hence $f(H)$ lies at finite Hausdorff distance from a finite union of singular orthants.
\end{proof}

We conclude this section with a variant of \cref{prop:intersection_orthants} with flats in place of orthants.

\begin{proposition} \label{prop:intersection_semisingular}
Let $X$ and $Y$ be $n$-dimensional CAT(0) cube complexes, and let $f\colon X\to Y$ be a $q$--quasiisometric embedding. Let $F_1,\dots,F_s$ be singular flats whose intersection $H=\bigcap_{i=1}^sF_i$ is a $k$--flat. 

For each $D$ there exists $D'$ such that if $f(F_i)$ lies at Hausdorff distance at most $D$ from a singular flat $E_i\subset Y$ for all $i$, then $f(H)$ lies at Hausdorff distance at most $D'$ from a singular $k$--flat.
\end{proposition}

\begin{proof}
We can assume that $k\ge1$, for there is nothing to prove otherwise. Let $C\subset E_1$ be the convex subset $C=E_1\cap\Big(\bigcap_{i=2}^sE_i^{+D}\Big)^{+D}$. Note that we have
\[
\bigcap_{i=1}^sE_i^{+D} \,\subset\, C^{+D} \qquad\text{and}\qquad 
    C \,\subset\, \Big(\bigcap_{i=1}^sE_i^{+D}\Big)^{+D}.
\]
In other words, $C$ and $\bigcap_{i=1}^sE_i^{+D}$ are at Hausdorff distance at most $D$. By the choice of the flats $E_i$, we have that
\[
f(H) \,\subset\, \bigcap_{i=1}^sf(F_i) \,\subset\, \bigcap_{i=1}^sE_i^{+D} \,\subset\, \bigcap_{i=1}^sf(F_i)^{+2D}.
\]
Moreover, because the $F_i$ are singular, they are convex, and hence \cref{cor:distance_from_intersection} provides a constant $T=T(q,2D,\dim X)$ such that
\[
\bigcap_{i=1}^sf(F_i)^{+2D} \,\subset\, f(H)^{+T}.
\]
Combining these comparisons, we find that
\[
f(H) \,\subset\, \bigcap_{i=1}^sE_i^{+D} \,\subset\, C^{+D} \,\subset\, 
    \big(\bigcap_{i=1}^sE_i^{+D}\big)^{+2D} \,\subset\, f(H)^{+T+2D}.
\]
In particular, $f(H)$ and $C$ are at Hausdorff distance at most $T+2D$.

We will show that $\partial_TC$ is isometric to $\mathbb{S}^{k-1}$ for some $k$, deduce that it bounds a $k$--flat $P$, and show that $P$ lies at uniformly bounded Hausdorff distance from $f(H)$.

The set $C$ is a convex subspace of the $n$--flat $E_1$. As each $E_i$ is singular, its boundary is a round simplicial sphere (even semisingular suffices, see \cref{rem:boundary_semisingular}). Moreover, $C$ lies at Hausdorff distance at most $D$ from $\bigcap_{i=1}^sE_i^{+D}$, so its Tits boundary is a subcomplex of $\partial_TE_1$. This implies that $C$ lies at finite Hausdorff distance from a finite union of orthants of $E_1$, and hence so does $f(H)$. \cref{prop:boundary_quasi_flat_sphere} now tells us that $\partial_Tf(H)$ is homeomorphic to $\mathbb S^{k-1}$. Hence $\bigcap_{i=1}^s\partial_TE_i=\partial_T\bigcap_{i=1}^sE_i^{+D}$ is also homeomorphic to $\mathbb S^{k-1}$. Let us write $B=\bigcap_{i=1}^s\partial_TE_i$.

Let $S\subset\partial_TE_1$ be the smallest round sphere that contains $B$. The fact that the $E_i$ are semisingular implies that $B$ is a subcomplex of $\partial_TE_1$. Since they are round spheres, $B$ contains every geodesic between its pairs of points at distance less than $\pi$. In other words, $B$ is ``$\pi$-convex''. Hence either $B=S$ or $B$ is contained in a closed hemisphere $N\subset S$.

Suppose $B\subset N$ for some closed hemisphere $N$. We cannot have $B\subset\partial N$, for that would contradict minimality of $S$. Hence $B$ meets the interior of $N$. But now $\pi$--convexity of $B$ implies that it is contractible, which contradicts the fact that $B\cong\mathbb S^{k-1}$. Thus $B=S$. 

We have shown that $\partial_TC=B=S$ is a round sphere subcomplex of $\partial_TE_1$. Hence there exists a singular $k$--flat $P\subset C\subset E_1$ with $\partial_TP=S$.

It suffices to bound the Hausdorff distance between $P$ and $f(H)$. We already know that $P\subset C\subset f(H)^{+T+2D}$. Let $\pi:P\to f(H)$ denote a closest-point projection, and let $\bar f\colon f(H)\to H$ be a quasiinverse of $f$. The composition $\bar f\pi:P\to H$ is a $(q+2T+4D)$--quasiisometric embedding, so by \cref{lem:quasiisom_quasi_surj} it is $D''=D''(q,T,D,k)$ coarsely surjective. Thus $f(H)$ lies in a uniform neighbourhood of $P$.
\end{proof}

\begin{remark}
Note that in \cref{prop:intersection_semisingular}, since we are only considering finitely many flats $F_1,\dots,F_s$, if we assume that each $f(F_i)$ is at finite Hausdorff distance from a singular flat in $Y$, then there certainly exists \emph{some} constant $D$ such that each $f(F_i)$ is at Hausdorff distance at most $D$ from a singular flat in $Y$. The extra utility comes from situations where we have an external bound on the quantity $D$ that does not depend on the $F_i$. This will be the case in \cref{sec:fully_branching_theorems}.
\end{remark}

\section{Quasimedian maps and coarse convexity} \label{sec:qm_cc}

In this section, we develop some of the technical machinery that we need in order to handle CAT(0) cube complexes whose asymptotic rank is greater than their dimension. This involves treating such cube complexes as being ``coarsely $n$--dimensional''. Although it will not play an explicit role here, this is the perspective of \emph{coarse median spaces}, introduced by Bowditch in \cite{bowditch:coarse}. We reiterate that this section is not needed for proving our main results in the case that the dimension and rank agree.

Going forwards, the key statements from this section are Propositions~\ref{claim:convex_orthants} and \ref{prop:subflat_image_union_orthants}. The former states that the images of convex unions of orthants in certain quasiflats are ``coarsely median-convex'', in the sense of \cref{def:cmc}. The latter is an analogue of \cref{prop:intersection_orthants} that applies to non-singular flats.

Throughout this section, if $X$ is a CAT(0) cube complex, then we shall consider $X$ equipped with the median metric $d_1$. The following definition is from \cite[\S8]{bowditch:coarse}.

\begin{definition}[Quasimedian]
Let $M$ and $N$ be median metric spaces. A map $\phi\colon M\to N$ is \emph{$m$--quasimedian} if $\dist\big(\phi(\mu_M(a,b,c)),\,\mu_N(\phi(a),\phi(b),\phi(c))\big)\le m$ for all $a,b,c\in M$.
\end{definition}

The following provides an appropriate coarse notion of convexity in this coarse setting; compare \cref{rem:ccc_convexity}.

\begin{definition}[Coarsely median-convex] \label{def:cmc}
Let $X$ be a median metric space. A subset $A\subset X$ is \emph{$C$--coarsely median-convex} if $\mu(a,x,b)\in A^{+C}$ for every $a,b\in A$ and $x\in X$.
\end{definition}

The next lemma shows that coarsely median-convex subsets are close to actual convex subcomplexes. We refer to \cref{def:hull} for the definition of the hull of a subset of a CAT(0) cube complex.

\begin{lemma} \label{lem:coarsely_convex_coarsely_convex}
Let $X$ be an $n$--dimensional CAT(0) cube complex. If $A\subset X$ is $C$--coarsely median-convex, then $A$ is $2^nC$--coarsely dense in $\Hull(A)$.
\end{lemma}

\begin{proof}
According to \cite[Lem.~2.18]{haettelhodapetyt:coarse}, the hull of $A$ can be obtained by taking medians $n$ times, in the following sense. Let $A_0=A$. Given $A_i$, set $A_{i+1}=\{\mu(a_1,x,a_2)\,:\,a_1,a_2\in A_i,\,x\in X\}$. The cited lemma shows that $A_n=\Hull(A)$. 

Note that $A_i\subset A_{i+1}$ for all $i$, because $\mu(a,a,a)=a$. The assumption of the lemma is equivalent to saying that $A_1\subset A^{+C}$. Suppose that we know that $A_i\subset A^{(2^i-1)C}$. Given $a_1,a_2\in A_i$, we can let $a_1',a_2'\in A$ have $\dist(a_1,a'_1)\le (2^i-1)C$ and $\dist(a_2,a'_2)\le(2^i-1)C$. We then have that $\dist(\mu(a_1,x,a_2),\mu(a'_1,x,a'_2))\le2(2^i-1)C$ for all $x\in X$, because the median is 1--Lipschitz in each factor. Since $A$ is $C$--coarsely median-convex, there exists $a\in A$ with $\dist(a,\mu(a'_1,x,a'_2))\le C$. Consequently, $\dist(a,\mu(a_1,x,a_2))\le(2^{i+1}-1)C$. Taking $i=n$ proves the lemma.
\end{proof}

\subsection{Quasimedian embeddings of unions of orthants} \label{subsec:quasimedian_orthants}

Here we show \cref{claim:convex_orthants}, which states that if one has a convex union of orthants in $\R^n$ and applies a quasimedian, quasiisometric embedding in a CAT(0) cube complex of asymptotic rank $n$, then the image is coarsely median-convex. 

\begin{observation} \label{obs:quasi_median_observation}
Suppose that $\phi\colon (M,\mu)\to X$ is an $m$--quasimedian map from a median metric space to a CAT(0) cube complex $X$.
If $\mu(a,b,c)=b$, then as $\phi$ is $m$-quasimedian, $\mu(\phi(a),\phi(b),\phi(c))$ is at distance at most $m$ from $\phi(b)$. Hence all but at most $m$ hyperplanes that separate $\phi(a)$ from $\phi(b)$ also separate $\phi(a)$ from $\mu(\phi(a),\phi(b),\phi(c))$, so all but at most $m$ such hyperplanes separate $\phi(a)$ from $\{\phi(b),\phi(c)\}$.
\end{observation}

Informally, the following technical lemma says that the quasimedian image of a box is essentially a box.

\begin{lemma}[Box lemma] \label{lem:quasimedian_cube}
Let $\prod_{i=1}^n [0,a_i]\subseteq \mathbb{R}^n$ be a box, equipped with the standard median structure. Let $V_j$ denote the set of corner vertices of $\prod_{i=1}^n[0,a_i]$ whose $j^\mathrm{th}$ coordinate is zero. Let $e_j\in\R^n$ be the point whose $j^\mathrm{th}$ coordinate is $a_j$ and whose other coordinates are zero.

Let $X$ be a CAT(0) cube complex, and let $\phi\colon \prod_{i=1}^n [0,a_i]\to X$ be an $m$--quasimedian map.
For each $j$, let $W_j$ be the set of all hyperplanes in $X$ that separate $\phi(0)$ from $\phi(e_j)$. Let $W$ be the set of all hyperplanes in $X$ that separate $\phi(\bar 0)$ from $\phi((a_1,\dots, a_n))$.

\begin{enumerate}[label={\arabic*)}]
\item   For each $j$ there exists $W_j'\subseteq W_j$ such that $|W_j'|\geq |W_j|-2^nm$ and all hyperplanes in $W_j'$ separate $\phi(V_j)$ from $\phi(V_j+e_j)$.
In particular, for $i\neq j$, every element of $W_i'$ crosses every element of $W_j'$.

\item There exists $W'\subset W$ such that $|W'|\geq |W|-m(n-1)-\binom n22^nm$ and $W'\subseteq \bigcup_{j=1}^n W_j$.
\end{enumerate}
\end{lemma}

\begin{proof}
$1)\quad$ By applying a median symmetry of $\mathbb{R}^n$, it is enough to consider $j=1$.

Fix $v\in V_1$. Note that $\mu(\bar 0, e_1, v+e_1)=e_1$ and $\mu(e_1,\bar0,v)=\bar 0$. Thus, by \cref{obs:quasi_median_observation}, all but at most $m$ of the hyperplanes that separate $\phi(\bar 0)$ from $\phi(e_1)$ actually separate $\phi(\bar0)$ from $\{\phi(e_1),\phi(v+e_1)\}$. Similarly, all but at most $m$ elements of $W_1$ separate $\phi(e_1)$ from $\{\phi(\bar0),\phi(v)\}$. 
Thus, all but at most $2m$ of the hyperplanes in $W_1$ separate $\phi(v)$ from $\phi(v+e_1)$.

By considering all $v\in V_1$, we get that all but at most $2m\cdot |V_1|=2m\cdot 2^{n-1}$ hyperplanes in $W_1$ separate $\phi(v)$ and $\phi(v+e_1)$ for all $v\in V_1$. We let $W'_1\subset W_1$ be the complement of those hyperplanes.

The ``in particular'' statement follows, because if $i\neq j$ then every $h\in W'_i$ separates $\phi(V_i)$ from $\phi(V_i+e_i)$ and every $h'\in W'_j$ separates $\phi(V_j)$ from $\phi(V_j+e_j)$, and hence each of the four points $\phi(\bar 0), \phi(e_i),\phi(e_j),\phi(e_i+e_j)$ lies in a different intersection of the halfspaces defined by $h$ and $h'$. 

\medskip

$2)\quad$ Let $c=(a_1,\dots,a_n)$. We can write $c$ as a sequence of medians as follows (\emph{cf.} \cite[Lem.~2.1]{petytzalloum:constructing}):
\[
c \,=\, \mu\big(e_n,c,\mu(e_{n-1},c,\mu(\dots,\mu(e_2,c,e_1)\dots\big).
\]
This expression involves only $n-1$ medians, so the fact that $\phi$ is $m$--quasimedian implies that $\phi(c)$ lies at distance at most $(n-1)m$ from a point $\beta$ obtained by taking medians similarly in $X$. By construction, the point $\beta$ lies in $\Hull\{\phi(\bar 0),\phi(e_1),\dots,\phi(e_n)\}$. Any hyperplane that separates $\phi(\bar 0)$ from $\phi(c)$ but does not separate any two elements of $\{\phi(\bar 0),\phi(e_1),\dots,\phi(e_n)\}$ must separate $\phi(c)$ from $\Hull\{\phi(\bar 0),\phi(e_1),\dots,\phi(e_n)\}$, and in particular from $\beta$. Hence there are at most $m(n-1)$ such hyperplanes. 

We have shown that there is a subset $W''\subset W$ with $|W''|\ge|W|-m(n-1)$ such that every element of $W''$ separates some pair of elements of $\{\phi(\bar0),\phi(e_1),\dots,\phi(e_n)\}$. By the first part of the lemma and symmetry of the cube, if $i\neq j$ then all but at most $2^nm$ hyperplanes that separate $\phi(e_i)$ from $\phi(e_j)$ also separate $\phi(\bar 0)$ from $\phi(e_j)$. Thus there exists $W'\subset W''$ with $|W'|\ge|W''|-{\binom n2}2^nm$ such that $W'\subset\bigcup_{j=1}^nW_j$, as desired.
\end{proof}

Before proving \cref{claim:convex_orthants}, we first control the images of singular geodesics.

\begin{lemma} \label{claim:convex_axes}
Let $X$ be a finite-dimensional CAT(0) cube complex of asymptotic rank $n$, and let $\mathcal{O}=\bigcup_{i\in I}O_i\subseteq \mathbb{R}^n$ be a convex union of orthants with the subspace median structure. Suppose that $\phi\colon \mathcal{O}\to X$ is a $q$--quasiisometric embedding that is $m$--quasimedian. There exists $D'=D'(m,q,n,X)\ge0$ such that if $\alpha\subset\mathcal{O}$ is a singular geodesic, then $\phi(\alpha)$ is $D'$--coarsely median-convex. 
\end{lemma}

It is enough to prove the statement for every finite singular geodesic segment $\gamma\colon [a,b]\to \mathcal{O}$. We uniformly approximate $\phi(\gamma)$ by a connected, $d_1$-isometrically embedded subcomplex $L\subset[\phi(\gamma(a)),\phi(\gamma(b))]$. By the definition of $\mathcal{O}$, the geodesic $\gamma$ is an edge of an $n$--cube in $\cal O$. Using this and the asymptotic rank assumption, we show that $[\phi(\gamma(a)),\phi(\gamma(b))]$ is uniformly hyperbolic, and hence $L$ lies at uniformly finite Hausdorff distance from $[\phi(\gamma(a)),\phi(\gamma(b))]$.

\begin{proof}[Proof of \cref{claim:convex_axes}]
To show that $\phi(\alpha)$ is coarsely median-convex, it suffices to show that $\phi(\gamma)$ is uniformly coarsely median-convex for every finite subsegment $\gamma\subset\alpha$. 

By the definition of asymptotic rank, there exists $R>0$ such that $X$ does not contain an isometric image of $[0,R]^{n+1}$.
Let $R'=Rq\dim X+q^2+2^nmq$. If $\gamma$ has length at most $R'$, then the diameter of $\phi(\gamma)$ is bounded by $qR'+q$, and hence is $(qR'+q)$--coarsely median-convex. So suppose that $\gamma$ has length $C> R'$.

Up to a median isometry of $\mathbb{R}^n$ (rotations by $\frac{\pi}{2}$ or $\pi$ about an axis, reflections in the hyperplanes normal to axes, and translations), we can assume the endpoints of $\gamma$ are $\bar 0$ and $(C,0,\dots,0)$, and that $\mathcal{O}$ contains $[0,\infty)^n$.
For $j\in\{1,\dots,n\}$, let $e_j\in \mathbb{R}^n$ be the point whose $j^\mathrm{th}$ coordinate is $C$ and whose other coordinates are 0. Let $\gamma_j$ be the singular geodesic between $0$ and $e_j$. In particular, $\gamma=\gamma_1$.
The hull in $\mathcal{O}$ of $\bigcup_{j=1}^n\gamma_j$ is a cube of side-length $C$. 


Each $\gamma_j$ is closed under the median operation in $\mathcal{O}$, so $\phi(\gamma_j)$ is $q$--coarsely connected and $m$--coarsely closed under the median operation of $X$. Thus \cite[Prop.~2.8]{hagenpetyt:projection} (which follows from either \cite[Prop.~4.1]{bowditch:convex} or \cite[Prop.~4.1]{fioravanti:coarse}) shows that there exists $D=D(m,q,\dim X)$ such that $\phi(\gamma_j)$ lies within Hausdorff distance $D$ of a CAT(0) cubical subcomplex $L_j\subseteq X$. Note that $L_j$ is isometrically embedded with respect to the cubical metric $d_1$, but may not be with respect to CAT(0) metric. 

Next we bound the Hausdorff distance between $L_1$ and $[\phi(\bar0),\phi(e_1)]$. 
For each $j$, let $W_j$ denote the set of hyperplanes of $X$ that separate $\phi(\bar0)$ from $\phi(e_j)$, and let $W'_j\subset W_j$ be as in \cref{lem:quasimedian_cube}. In particular, $|W_j\ssm W'_j|\le2^nm$, and if $j_1\ne j_2$, then every element of $W'_{j_1}$ crosses every element of $W'_{j_2}$.

Suppose that $[\phi(\bar0),\phi(e_1)]$ contains an isometric copy of $[0,R+2^nm]^2$. In this case, $W'_1$ must contain two chains $c_0,c_1$ of hyperplanes in $X$ such that every element of $c_0$ crosses every element of $c_1$ and both have length at least $R$. (Recall that a chain of hyperplanes is a sequence $h_1,\dots,h_t$ such that $h_i$ separates $h_{i-1}$ from $h_{i+1}$ for all $i$.) For each $j>1$, we have that $|W_j|\ge\dist_1(\phi(\bar0),\phi(e_j))\ge\frac{R'}q-q$. Hence, if $j>1$, then $W'_j$ contains a chain $c_j$ of hyperplanes of length at least $\frac1{\dim X}(\frac{R'}q-q-2^nm)=R$. But now the chains $c_0,\dots,c_n$ provide a copy of $[0,R]^{n+1}$ inside $X$, contrary to our assumptions. Thus $[\phi(\bar0),\phi(e_1)]$ does not contain an isometric copy of $[0,R+2^nm]^2$.

From the proof of \cite[Lem.~7.14]{hagen:weak}, the fact that $[\phi(\bar0),\phi(e_1)]$ does not contain an isometric copy of $[0,R+2^nm]^2$ implies that it is $(R+2^nm)$--hyperbolic in the metric $\dist_1$. Every point of $[\phi(\bar0),\phi(e_1)]$ lies on a $\dist_1$-geodesic from $\phi(\bar0)$ to $\phi(e_1)$, and the fact that $L_1$ is isometrically embedded implies that it contains a $\dist_1$-geodesic from $\phi(\bar0)$ to $\phi(e_1)$. Hyperbolicity thus implies that the Hausdorff distance between $L_1$ and $[\phi(\bar0),\phi(e_1)]$ is at most $R+2^nm$. It follows that the Hausdorff distance between $\phi(\gamma_j)$ and $[\phi(\bar0),\phi(e_1)]$ is at most $D+R+2^nm$.

We can now complete the proof of the lemma. Given $a,b\in\phi(\gamma)$, there exist $a',b'\in[\phi(\bar0),\phi(e_1)]$ with $\dist(a,a'),\dist(b,b')\le D+R+2^nm$. Because $\mu$ is 1--Lipschitz in each factor, for every $x\in X$ we have $\dist\big(\mu(a,x,b),\mu(a',x,b')\big)\le2D+2R+2^{n+1}m$. But $\mu(a',x,b')\in[\phi(\bar0),\phi(e_1)]$, so we conclude that $\mu(a,x,b)$ lies at distance at most $3D+3R+3\cdot2^nm$ from $\phi(\gamma)$. We have shown that $\phi(\gamma)$ is $D'$--coarsely median-convex, where $D'=\max\{qR'+q,3D+3R+3\cdot2^nm\}$. 
\end{proof}

Note that in the above proof we did not strictly need $\cal O$ to be a union of orthants, only that every point is the corner of a large $n$--cube with side-length depending on the space $X$. 

We now bootstrap \cref{claim:convex_axes} to get the same conclusion for median-convex subsets of $\mathcal{O}$.

\begin{proposition} \label{claim:convex_orthants}
Let $X$ be a finite-dimensional CAT(0) cube complex of asymptotic rank $n$, and let $\mathcal{O}=\bigcup_{i\in I}O_i\subseteq \mathbb{R}^n$ be a convex union of orthants with the subspace median structure. Suppose that $\phi\colon \mathcal{O}\to X$ is a $q$--quasiisometric embedding that is $m$--quasimedian. There exists $D=D(m,q,n,X)>0$ such that if $A\subset\mathcal{O}$ is median-convex, then $\phi(A)$ is $D$--coarsely median-convex.
\end{proposition}

\begin{proof}
We must show that if we take $x\in X$ and $\phi(a),\phi(c)\in\phi(A)$, then $\mu(\phi(a),x,\phi(c))$ lies uniformly close to $\phi(A)$. Note that since $\mu\big(\phi(a),\,\mu(\phi(a),x,\phi(c)),\,\phi(c)\big)=\mu(\phi(a),x,\phi(c))$, there is no loss in assuming that $x$ lies in the cubical hull $[\phi(a),\phi(c)]$. 

Let $a,c\in A$ and let $x\in [\phi(a),\phi(c)]$. Up to a median isometry of $\R^n$ we can assume $a=\bar 0$ and $c$ is in the positive orthant of $\mathbb{R}^n$. Let $e_1,\dots,e_n$ be the projections of $c$ to the different coordinate axes and let $\gamma_j$ be the singular geodesic from $\bar 0$ to $e_j$.

By \cref{claim:convex_axes}, there exists $D'=D'(m,q,n,X)$ such that each $\phi(\gamma_j)$ is $D'$--coarsely median-convex. Thus, if we let $x_j=\mu(\phi(\bar 0),x,\phi(e_j))$, then there is some $x'_j\in\gamma_j$ such that $\dist(\phi(x'_j),x_j)\leq D'$. Let $x'=\sum_{j=1}^n x_j'\in[\bar 0,c]$. We shall bound $\dist(x,\phi(x'))$.

Since $\phi$ is $m$-quasimedian and $\mu(x',\bar0, c)=x'$, the projection of $\phi(x')$ to $[\phi(\bar0),\phi(c)]$ is at distance at most $m$ from $\phi(x')$. Let $y$ be that projection. It is enough to bound $\dist(x,y)$. Equivalently, it is enough to bound the number of hyperplanes that separate $x$ from $y$.

Since $x,y\in [\phi(\bar 0),\phi(c)]$, all hyperplanes of $X$ that separate $x$ from $y$ also separate $\phi(\bar0)$ from $\phi(c)$. 
By the second statement of \cref{lem:quasimedian_cube}, all but at most $m(n-1)+{\binom n2}2^nm$ hyperplanes separating $\phi(\bar 0)$ from $\phi(c)$ also separate $\phi(\bar0)$ from some $\phi(e_j)$. Thus it is enough to bound for each $j$ the number of hyperplanes that separate $\phi(\bar 0)$ from $\phi(e_j)$ and also $x$ from $y$.

No such hyperplane can separate $x$ from $x_j$, since, by definition, $x_j=\mu(x,\phi(\bar 0),\phi(e_j))$. Thus any such hyperplane must separate $x_j$ from $y$ and either $\phi(\bar 0)$ or $\phi(e_j)$. It therefore separates $x_j$ from $\mu(\phi(\bar0),\phi(e_j),y)$, so the number of such hyperplanes is bounded above by the distance from $x_j$ to $\mu(\phi(\bar0),\phi(e_j),y)$.

Because $\dist(y,\phi(x'))\le m$, the fact that $\mu$ is 1--Lipschitz in each factor implies that $\mu(\phi(\bar0),\phi(e_j),y)$ is $m$--close to $\mu(\phi(\bar0),\phi(e_j),\phi(x'))$. The map $\phi$ is $m$--quasimedian, so the latter point is $m$--close to $\phi\big(\mu(\bar0,e_j,x')\big)=\phi(x'_j)$. By the choice of $x'_j$, this shows that the distance from $x_j$ to $\mu(\phi(\bar0),\phi(e_j),y)$ is at most $D'+2m$.

In total, we have shown that 
\[
\dist(x,\phi(x')) \,\le\, m+\dist(x,y) \,\le\, m+m(n-1)+{\binom n2}2^nm+n(D'+2m).
\]
This proves the proposition with $D=3mn+{\binom n2}2^nm+nD'$.
\end{proof}

\subsection{Quasilines, quasiorthants, and images of intersections of flats}

In this subsection we prove \cref{prop:subflat_image_union_orthants}, which controls the images of intersections of top-rank flats under quasiisometric embeddings. Along the way, we show Propositions~\ref{prop:lines} and~\ref{prop:quasiorthant}, which may be of independent interest. The latter says that any CAT(0) cube complex quasimedian quasiisometric to an orthant contains a coarsely dense orthant. The former states that a CAT(0) cube complex is a quasiline, then it contains a biinfinite CAT(0) geodesic. 

It is noteworthy that these statements do not require properness of the CAT(0) cube complex. Although \cite[Prop.~2.1]{leeb:characterization} gives a way to find flats in proper CAT(0) spaces, properness is used in an essential way in the argument, and we do not know whether the conclusion holds without it. 

\begin{proposition} \label{prop:lines}
If $X$ is a finite-dimensional CAT(0) cube complex quasiisometric to a line, then $X$ contains a biinfinite CAT(0) geodesic.

Similarly, if $X$ is a finite-dimensional CAT(0) cube complex quasiisometric to $[0,\infty)\subset\R$, then $X$ contains a geodesic ray.
\end{proposition}

\begin{proof}
Suppose that $X$ is a finite-dimensional CAT(0) cube complex quasiisometric to a line. Let $H$ denote the set of hyperplanes of $X$, and let $\cal O=\prod_{h\in H}\{h^-,h^+\}$ denote the set of all possible orientations of the hyperplanes, where $h^\pm$ are the two halfspaces of the hyperplane $h\in H$. That is, an element $\phi\in\cal O$ is a choice $\phi(h)\in\{h^-,h^+\}$ for each $h\in H$. By Tychonoff's theorem, $\cal O$ is compact. Recall that $\phi\in\cal O$ is called an \emph{ultrafilter} if $\phi(h_1)\cap\phi(h_2)\ne\varnothing$ for every $h_1,h_2\in H$. Each vertex $x\in X$ has a corresponding ultrafilter $\phi_x\in\cal O$, where for each hyperplane we select the halfspace containing $x$.  

The definition of the Tychonoff topology implies that the set of ultrafilters is closed. Indeed, for each $h_1,h_2\in H$, let $I(h_1,h_2)\subset\{h_1^\pm\}\times\{h_2^\pm\}$ be the set of pairs with nonempty intersection. It is a nonempty subset of a set of cardinality four, hence closed. Let $\pi_{h_1,h_2}\colon \cal O\to\{h_1^\pm\}\times\{h_2^\pm\}$ be the projection map. The set of ultrafilters is precisely the intersection of all sets $\pi_{h_1,h_2}^{-1}I(h_1,h_2)$, and hence is an intersection of closed sets.

Let $f\colon \R\to X$ be a quasiisometry, and let $q\ge1$ be a quasiisometry constant. By compactness of $\cal O$, the sequences $(\phi_{f(n)})$ and $(\phi_{f(-n)})$ have convergent subnets $x_\bullet$ and $z_\bullet$, respectively. Let $\xi$ and $\zeta$ be respective limit points. They are ultrafilters because the set of ultrafilters is closed. By the definition of subnets, $x_\bullet$ and $z_\bullet$ escape every bounded subset of $X$, so $\xi$ and $\zeta$ are ``at infinity'', in the sense that they do not define points of $X$. (They define points in the \emph{Roller boundary} of $X$.)

Let $H'\subset H$ be the set of all hyperplanes $h$ such that $\xi(h)\ne\zeta(h)$, and let $X'$ be the CAT(0) cube complex dual to $H'$. Identifying the hyperplanes of $X'$ with $H'\subset H$ gives a canonical isometric embedding $X'\to X$, whose image is the convex subcomplex $\bigcap_{h\in H\ssm H'}\xi(h)$. It therefore suffices to find a biinfinite CAT(0) geodesic inside $X'$.

We can give $H'$ a partial order by declaring $h<h'$ whenever $\xi(h)\subset\xi(h')$. An antichain for this partial order is a set of pairwise-crossing hyperplanes, which has cardinality at most $\dim X$. According to Dilworth's theorem, we can therefore partition $H'$ into at most $\dim X$ chains - see also \cite[P.355]{chatterjiruane:some}. This partition gives a $d_1$-isometric embedding $X'\to\R^{\dim X}$. In particular, $X'$ is proper.

Now we show that $H'$ contains a biinfinite chain of hyperplanes, or in other words a sequence $(h_i)_{i\in\Z}$ of hyperplanes such that $h_i$ separates $h_{i-1}$ from $h_{i+1}$ for all $i$. 

Because $X$ is a $q$--quasiline, there is a number $r=r(q)$ such that, for each $p\in X$, the complement of each $r$--ball $B(p,r)$ in $(X,d_1)$ contains exactly two unbounded connected components, which we denote $C_{p}^+$ and $C_{p}^-$. If $a,b,\in C_{p}^+$, then any path from $a$ to $b$ that leaves $C_{p}^+$ must enter $B(p,r)$, and hence every $d_1$-geodesic from $a$ to $b$ is contained in the $2r$--neighbourhood of $C_{p}^+$ because its first and last points in $B(p,r)$ are at distance at most $2r$.
This shows that $\dist(\mu(a,b,y),C_{p}^+)\le2r$ for all $y\in X$. According to \cref{lem:coarsely_convex_coarsely_convex}, the convex subcomplex $\Hull(C_{p}^+)$ is contained in the $r'$--neighbourhood of $C_{p}^+$, where $r'=2^{1+\dim X}r$. The same argument applies to $C_{p}^-$.

Because of this, if $\dist(p,p')>2r'$ then, up to a change of signs, the convex subcomplexes $\Hull(C_{p}^-)$ and $\Hull(C_{p'}^+)$ are disjoint. Moreover, $\Hull(C_{p}^-)\subset\Hull(C_{p'}^-)$ and $\Hull(C_{p'}^+)\subset\Hull(C_{p}^+)$. 

Because $f$ is a $q$--quasiisometry, if for $i\in\Z$ we let $m_i=(2qr'+q^2+1)i$, then we have that $\dist(f(m_i),f(m_{i+1}))>2r'$. In particular, up to changing signs, the convex sets $\Hull(C_{f(m_i)}^-)$ and $\Hull(C_{f(m_{i+1})}^+)$ are disjoint. Since each is an intersection of halfspaces (see \cref{rem:ccc_convexity}), there must exist a hyperplane $h_i\in H$ that separates them. Because $\Hull(C_{f(m_{i+2})}^+)\subset\Hull(C_{f(m_{i+1})}^+)$ and $\Hull(C_{f(m_{i-1})}^-)\subset\Hull(C_{f(m_i)}^-)$, the hyperplane $h_i$ separates $h_{i-1}$ from $h_{i+1}$ as desired. 

We have found a biinfinite chain $(h_i)\subset H$. To see that it is contained in $H'$, note that for each $i$, every sufficiently large $k\in\R$ satisfies $f(k)\in C_{f(m_{i+1})}^+$ and $f(-k)\in C_{f(m_i)}^-$. Thus $h_i$ separates every accumulation point of $(\phi_{f(n)})$ in $\cal O$ from every accumulation point of $(\phi_{f(-n)})$ in $\cal O$, and in particular $\xi(h_n)\ne\zeta(h_n)$.

We have shown that $X'$ is a proper CAT(0) cube complex that is a convex subcomplex of the quasiline $X$ and has a biinfinite chain $(h_n)$ of hyperplanes. We can take two diverging sequences in $X'$ and use properness to obtain a pair of geodesic rays $\gamma_1$ and $\gamma_2$ representing two points in $\partial_TX'\subset\partial_TX$. If $\eta_n$ denotes the geodesic in $X'$ from $\gamma_1(n)$ to $\gamma_2(n)$, then since $X'$ is proper, we can use Arzelà--Ascoli to show that the sequence $(\eta_n)$ converges to a biinfinite geodesic $\eta$. This proves the first statement. The proof of the second statement is similar. 
\end{proof}

We now show a similar statement for orthants under the additional assumption that the quasiisometry is quasimedian. This assumption is necessary: consider a sector in $\R^2$ of angle $\frac\pi4$. On the other hand it is a natural addition compared to \cref{prop:lines}, because every quasiisometric embedding between hyperbolic spaces is quasimedian \cite[Lem.~2.9]{petyt:mapping}. Note that we cannot conclude that the orthant in $X$ is singular; see \cref{eg:corner_to_corner}.

\begin{proposition} \label{prop:quasiorthant}
If $X$ is a finite-dimensional CAT(0) cube complex with an $m$--quasimedian $q$--quasiisometry $f\colon O\to X$ for some orthant $O=[0,\infty)^n$ equipped with the standard median then $X$ contains a coarsely dense orthant.
\end{proposition}

\begin{proof}
Asymptotic rank is preserved by quasiisometry, so $X$ has asymptotic rank equal to $n$. \cref{claim:convex_orthants} therefore provides a constant $D=D(m,q,n,X)$ such that $f(Z)$ is $D$--coarsely median-convex for every convex subcomplex $Z\subset O$. \cref{lem:coarsely_convex_coarsely_convex} then tells us that $f(Z)$ is $2^{\dim X}D$--coarsely dense in $\Hull(f(Z))$.

Let $r_1,\dots,r_n$ denote the axis-rays of $O$, each of which is median-convex. We shall prove by induction on $k$ that if $A$ is parallel to a suborthant of $O$ spanned by $k$ elements of $\{r_1,\dots,r_n\}$, then $\Hull(f(A))$ contains a $D_k$--coarsely dense CAT(0) $k$--orthant, for some constant $D_k$ depending only on $D$, $k$, $q$, $m$, and $X$. The idea is similar to \cref{prop:singular_quasi_flats}. 

For $k=1$, we have that $\Hull(f(A))$ is a finite-dimensional CAT(0) cube complex quasiisometric to $[0,\infty)$, so the second statement of \cref{prop:lines} tells us that it contains a CAT(0) geodesic ray, which is $D_1$--coarsely dense for some $D_1$ depending only on $D$, $q$, and $X$.

Next, let $k>1$ and suppose that we have established all cases up to $k-1$. We shall think of $O$ as a subset of $\R^n$, allowing us to use vector notation. Let $A$ be a parallel of the suborthant of $O$ spanned by $k$ elements of $\{r_1,\dots,r_n\}$. After relabelling, we can assume that $A$ is spanned by parallels of the rays $r_1,\dots,r_k$. Let $a$ be the cone point of $A$. Let $B_0\subset A$ be the $(k-1)$--suborthant spanned by $a+r_1,\dots,a+r_{k-1}$, and let $B_1$ be a parallel of $B_0$ inside $A$. The cone point of $B_0$ is $a\in A$. Let $c$ denote the cone point of $B_1$. We have $c\in a+r_k$.

The projection of $a+r_k$ to $B_0$ is precisely the cone point $a\in B_0$. This can be characterised by saying that for every $x,y\in B_0$ such that $a\in[x,y]$ and for every $p\in a+r_k$, we have $\mu(x,p,y)=a$. Since $f$ is $m$--quasimedian, it follows that $\mu(f(x),f(p),f(y))$ lies at distance at most $m$ from $f(a)$. But $f(B_0)$ is $2^{\dim X}D$--coarsely dense in $\Hull(f(B_0))$ and the median operation is 1--Lipschitz, so the projection of $f(p)$ to $\Hull(f(B_0))$ lies at distance at most $m+2^{1+\dim X}D$ from $f(a)$, for all $p\in a+r_k$. This applies to $c\in a+r_k$ in particular.

Now, for $i\in\{0,1\}$ we have that $\Hull(f(B_i))$ contains a $D_{k-1}$--coarsely dense $(k-1)$--orthant $Q_i$, by the inductive assumption. The cone point $a'$ of $Q_0$ lies at distance at most $D_{k-1}$ from $f(a)$. The cone point $c'$ of $Q_1$ lies at distance at most $D_{k-1}$ from $f(c)$, and hence its projection to $\Hull(f(B_0))$ lies at distance at most $D_{k-1}+m+2^{1+\dim X}D$ from $f(a)$ by the previous paragraph. 

Because $f$ is a quasiisometry, the orthants $Q_0$ and $Q_1$ are at finite Hausdorff distance, and therefore they are almost parallel, in the sense that the projection from one to the other is coarsely onto. More precisely, the fact that the projection to $\Hull(f(B_0))$ of $c'$ lies at distance at most $2D_{k-1}+m+2^{1+\dim X}D$ from $a'$, regardless of the choice of parallel $B_1$, and the fact that $Q_0$ is $D_{k-1}$--coarsely dense in $\Hull(f(B_0))$ imply that there is a uniformly coarsely dense suborthant $Q'_0\subset Q_0$ such that, for every possible choice of $B_1$, the projection of $Q_1$ to $Q_0$ contains $Q'_0$. Note that $Q'_0$ is uniformly coarsely dense in $\Hull(f(B_0))$, where the constant depends only on $D$, $k$, $q$, $m$, and $X$.

Consider the parallel set $P(Q'_0)\subset\Hull(f(A))$ of $Q'_0$, which splits metrically as $Q'_0\times Y$ for some CAT(0) space $Y$. We shall prove that $P(Q'_0)$ contains a $k$--orthant $Q$. We know by induction that $\Hull(f(a+r_k))\subset\Hull(f(A))$ contains a geodesic ray $\gamma$. We also saw that the projection of $f(a+r_k)$ to $\Hull(f(B_0))$ has diameter at most $2m+2^{2+\dim X}D$, so since $f(a+r_k)$ is $2^{\dim X}D$--coarsely dense in $\Hull(f(a+r_k))$, it follows that the diameter of the projection of $\gamma$ to $\Hull(f(B_0))$ is uniformly bounded. Now consider the projection of $\gamma$ to $P(Q'_0)$. It lies at bounded Hausdorff distance from $\gamma$, and hence $\gamma$ has a parallel geodesic $\gamma'\subset P(Q'_0)$ at bounded Hausdorff distance by \cref{lem:parallel_geod_in_convex}. As a geodesic ray in a product, $\gamma'$ must have trivial projection to $Q'_0$, because it is parallel to $\gamma$, whose projection to $\Hull(f(B_0))\supset Q'_0$ is bounded. We let $Q=Q'_0\times\gamma'$.

It remains to show that $Q$ is uniformly coarsely dense in $\Hull(f(A))$, with the constant depending only on $D$, $k$, $q$, $m$, and $X$. For this, it suffices to show that $f(A)$ lies in a uniform neighbourhood of $Q$, because $f(A)$ is $2^{\dim X}D$--coarsely dense in $\Hull(f(A))$. Each point in $A$ can be written as $(x,p)$, where $x\in B_0$ and $p\in a+r_k$. We know that $\gamma$ is $D_1$--coarsely dense in $\Hull(f(a+r_k))$, and also that $Q'_0$ is uniformly coarsely dense in $\Hull(f(B_0))$. In particular, $\gamma$ is at uniformly bounded Hausdorff distance from $\gamma'$. Thus there exists a point of $Q$ that is uniformly close to $f(x,p)$. This completes the proof.
\end{proof}

The next lemma is a coarsification of \cref{lem:intersection_weyl_cones_is_cone}.

\begin{lemma} \label{lem:intersect_quasiorthants}
Let $X$ be a finite-dimensional CAT(0) cube complex of asymptotic rank $n$. Suppose that for all $i\in\{1,\dots,m\}$ we have quasimedian quasiisometric embeddings $f_i\colon O_i\to X$, where $O_i$ is an $n$--orthant equipped with the standard median. 

There is a constant $s_0$ such that if $s\ge s_0$, then $\bigcap_{i=1}^mf_i(O_i)^{+s}$ is at finite Hausdorff distance from the coarsely median-convex image of a quasimedian quasiisometric embedding of a $k$--orthant for some $k\in\{0,\dots,n\}$.
\end{lemma}

\begin{proof}
Let $s_1$ be sufficiently large that $\bigcap_{i=1}^mf_i(O_i)^{+s_1}$ is nonempty. By \cref{claim:convex_orthants}, every $f_i(O_i)$ is coarsely median-convex, and hence so is every $f_i(O_i)^{+s_1}$, because the median is 1--Lipschitz in each factor. Hence, by \cref{lem:coarsely_convex_coarsely_convex}, there exists $D$ such that each $f_i(O_i)^{+s_1}$ lies at Hausdorff distance at most $D$ from $\Hull(f_i(O_i)^{+s_1})$. In particular, there exists $s_0>s_1$ such that if $s\ge s_0$, then $f_i(O_i)^{+s}$ contains $\Hull(f_i(O_i)^{+s_1})$, for each $i$. 

Let $Q=\bigcap_{i=1}^m\Hull(f_i(O_i)^{+s_1})$, which is nonempty. Because $\Hull(f_i(O_i)^{+s_1})$ is convex, \cref{prop:coarse_intersection_subcomplexes_ccx} tells us that $\bigcap_{i=1}^m\Hull(f_i(O_i)^{+s_1})^{+r}\subset Q^{+r\sqrt{\dim X}}$ for all $r$. In particular, for every $s\ge s_0$ we have 
\[
Q \,\subset\, \bigcap_{i=1}^mf_i(O_i)^{+s} \,\subset\, Q^{+s\sqrt{\dim X}}.
\]
That is, $\bigcap_{i=1}^sf_i(O_i)^{+s}$ lies at finite Hausdorff distance from $Q$.

Let $g$ be a quasiinverse of $f_1$, which is necessarily quasimedian. Every $x\in Q$ lies at distance at most $D+s_1$ from $f_1(O_1)$, so if $\pi:Q\to f_1(O_1)$ denotes a closest-point projection, then $g\pi$ is a quasimedian quasiisometric embedding. As $Q$ is convex, its intersection with $\Hull(f_1(O_1)^{+s_1})$ is convex. Since $f_1(O_1)$ is $(D+s_1)$--coarsely dense in $\Hull(f_1(O_1)^{+s_1})$, it follows that $g\pi(Q)$ is coarsely median-convex inside the orthant $O_1$. By \cref{lem:coarsely_convex_coarsely_convex}, $g\pi(Q)$ lies at finite Hausdorff distance from its hull. Being a convex subcomplex of an orthant, $\Hull(g\pi(Q))$ lies at finite Hausdorff distance from a suborthant $O\subset O_1$, of dimension $k\in\{0,\dots,n\}$. 

We have shown that $\bigcap_{i=1}^mf_i(O_i)^{+s}$ lies at finite Hausdorff distance from $Q$, which lies at finite Hausdorff distance from $f_1g\pi(Q)$, and hence from $f_1(O)$. The fact that $f_1(O)$ has coarsely median-convex image is given by \cref{claim:convex_orthants}.
\end{proof}

We need one more technical statement about the ultralimits of flats in CAT(0) spaces in asymptotic cones. Recall that a flat is always finite-dimensional in this paper. 

\begin{lemma} \label{lem:intersection_flats_cone}
Let $X$ be a CAT(0) space. Let $F_1,\dots,F_m\subset X$ be flats such that $H=\bigcap_{i=1}^mF_i$ is a flat. If $\hat X$ is an asymptotic cone of $X$ for which the ultralimit $\hat H$ exists, then $\hat H=\bigcap_{i=1}^m\hat F_i$, where $\hat F_i$ is the ultralimit of $F_i$.
\end{lemma}

\begin{proof}
Let $(\lambda_n)$ be the scaling sequence and $(o_n)$ be the sequence of basepoints associated with an asymptotic cone $\hat X$ in which $\hat H$ exists. Up to an isometry of asymptotic cones, we can assume that $o_n\in H$ for all $n$, because $\hat H$ exists. 

We clearly have $\hat H\subset\bigcap_{i=1}^m\hat F_i$. Assume towards a contradiction that there exists a point $x\in\bigcap_{i=1}^m\hat F_i$ that is not in $\hat H$. Let $(x_n)$ be a sequence in $X$ that represents $x$. For each $i$, let $x_{n,i}$ be the closest point in $F_i$ to $x_n$. Since $x\in\bigcap_{i=1}^mF_i$, for every $\eps>0$ there is a subset $N_\eps\in\omega$ such that $\dist(x_n,x_{n,i})\le\eps\lambda_n$ for all $n\in N$ and all $i$.

On the other hand, the fact that $x\notin\hat H$ implies that for every positive $\eps<\frac12\dist(x,\hat H)$ there is a subset $M_\eps\in\omega$ such that
\[
\dist(x_n,H)\,\in\,\big((\dist(x,\hat H)-\eps)\lambda_n,\,(\dist(x,\hat H)+\eps)\lambda_n\big) 
\]
for all $n\in M_\eps$.

For each $i$, let $p_{n,i}$ be the closest-point projection of $x_{n,i}$ to the subflat $H\subset F_i$. Given a positive number $\eps<\frac14\dist(x,\hat H)$, if $n\in N_\eps\cap M_\eps\in\omega$, then both of the above distance estimates hold. Since closest-point projection is 1--Lipschitz, it follows from the triangle inequality that
\[
\dist(x_{n,i},x_{n,j}) \,\le\, 2\eps\lambda_n \quad \text{and} \quad 
    \dist(x_{n_i},p_{n,i}) \,\ge\, (\dist(x,\hat H)-2\eps)\lambda_n
\]
for all $n\in N_\eps\cap M_\eps$ and all $i,j$.

For each $i$, consider the sequence of geodesic rays $\gamma_{n,i}$ based at $p_{n,i}$ and passing through $x_{n,i}$. Since $p_{n,i}\in H\subset F_i$ is the closest-point projection of $x_{n,i}\in F_i$ to $H$, the geodesic $\gamma_{n,i}$ is contained in $F_i$ and the point of $\partial_TF_i$ that it defines is at angle $\frac\pi2$ from $\partial_TH$. Since $\partial_TF_i$ is compact, after passing to a subsequence the endpoints of the $\gamma_{n,i}$ converge to a point $\xi_i\in\partial_TF_i$ at angle $\frac\pi2$ from $\partial_TH$.

We shall derive a contradiction by showing that $\xi_i\in\partial_TH$. For this, consider the above distance estimates for each term of the sequence $\eps_k=\frac1k$. By passing to a diagonal subsequence $(n_k)$, we can ensure that $\dist(x_{{n_k},i},x_{{n_k},j})\le\frac1k\lambda_{n_k}$ for all $k$. But now, for this subsequence we have that $\dist(x_{n_k,i},x_{n_k,j})$ grows sublinearly compared to $\dist(x_{n_k,i},p_{n_k},i)$. Consequently we have $\xi_i=\xi_j$ for all $i,j$. In other words, $\xi_1\in\bigcap_{i=1}^m\partial_TF_i$. 

But now let $p\in H$ and consider the geodesic ray $\gamma$ based at $p$ that represents $\xi_1$. Since $H\subset F_i$ for all $i$, we have $\gamma\subset F_i$ for all $i$, so $\gamma\subset H$. But this shows that $\xi_1\in\partial_TH$, contradicting the fact that it is at angle $\frac\pi2$ from $\partial_TH$.
\end{proof}

The next lemma is a version of \cref{cor:distance_from_intersection} the works for flats even when they are not subcomplexes. 

\begin{lemma} \label{lem:image_of_intersection_of_flats_top_rank}
Let $X$ and $Y$ be CAT(0) spaces and let $f\colon X\to Y$ be a $q$--quasiisometric embedding. Let $F_1,\dots,F_m\subset X$ be flats. If $H=\bigcap_{i=1}^mF_i$ is a flat, then $f(H)$ is Hausdorff-close to $\bigcap_{i=1}^mf(F_i)$.
\end{lemma}

\begin{proof}
Obviously $f(H)\subseteq \bigcap_{i=1}^m f(H_i)$. Assume towards contradiction that they are not Hausdorff-close. That is, there is a sequence $(x'_\ell)\subset\bigcap_{i=1}^mf(F_i)$ such that $\dist(x'_\ell,f(H))\ge \ell$. For each $i$ there exists $x_{\ell,i}\in F_i$ such that $f(x_{\ell,i})=x'_\ell$. Since $f$ is a $q$--quasiisometric embedding we have $\dist(x_{\ell,i},x_{\ell,j})\le q^2$ for all $i,j$. In particular, $x_{\ell,j}\in F_i^{+q^2}$ for all $i,j$.

Let $y_\ell\in H$ be such that $f(y_\ell)$ is a closest point in $f(H)$ to $x'_\ell$, and consider the asymptotic cone $\hat X=\lim_\omega(X,(\dist(x_{\ell,1},y_\ell)),(y_\ell))$. For each $i$ we have $\lim_\omega(F_i)=\lim_\omega(F_i^{+q^2})$. Hence $(x_{\ell,1})_\omega\in\bigcap_{i=1}^m(\lim_\omega F_i)$. But by construction, $(x_{\ell,1})_\omega\notin\lim_\omega(H)$. This contradicts \cref{lem:intersection_flats_cone}. Thus, $f(H)$ and $\bigcap_{i=1}^mf(F_i)$ are Hausdorff-close.
\end{proof}


We can now prove \cref{prop:subflat_image_union_orthants}, which is an analogue of \cref{prop:intersection_orthants} for non-singular orthants. We believe the statement is of independent interest. Note that \cref{prop:intersection_orthants} does not require $H$ or the $F_i$ to be flat, but requires their images to be at finite Hausdorff distance from a finite union of orthants, while \cref{prop:subflat_image_union_orthants} does not make assumptions about the images of the $F_i$. 

\begin{theorem} \label{prop:subflat_image_union_orthants}
Let $X$ and $Y$ be finite-dimensional CAT(0) cube complexes of asymptotic rank $n$, and let $f\colon X\to Y$ be a $q$--quasiisometric embedding. If $F_1,\dots,F_m\subset X$ are $n$--flats and $H=\bigcap_{i=1}^mF_i$ is a $k$--flat, then $f(H)$ lies at finite Hausdorff distance from a finite union of semisingular $k$--orthants.
\end{theorem}

\begin{proof}
According to \cref{lem:rank_cone_CCC} and \cite[Rem.~3.3]{munropetyt:coarse}, $Y$ is a \emph{coarse median space of rank $n$}. By \cite[Thm~1.1]{bowditch:quasiflats}, for each $i$ there are finitely many $n$--orthants $O_{i,j}$, equipped with the standard median structure, and quasimedian quasiisometric embeddings $\phi_{i,j}\colon O_{i,j}\to Y$ such that $f(F_i)$ is Hausdorff-close to $\bigcup_j\phi_{i,j}(O_{i,j})$. Let $s_0\ge0$ be a constant given by \cref{lem:intersect_quasiorthants} for the $\phi_{i,j}$. 

By \cref{lem:image_of_intersection_of_flats_top_rank}, $f(H)$ is at finite Hausdorff distance from $\bigcap_{i=1}^mf(F_i)$, and so there exist constants $D_2\ge D_1\ge s_0$ such that
\begin{align}
f(H) \,\subset\, \bigcap_{i=1}^mf(F_i) 
	\,&\subset\, \bigcap_{i=1}^m\big(\bigcup_j\phi_{i,j}(O_{i,j})^{+D_1}\big) \nonumber \\
	&=\, \bigcup_{j_1}\cdots\bigcup_{j_m}\big(\phi_{1,j_1}(O_{1,j_1})^{+D_1}\cap\cdots\cap\phi_{m,j_m}(O_{m,j_m})^{+D_1}\big) \label{eq:big_union} \\
	&\subset\, \bigcap_{i=1}^mf(F_i)^{+2D_1} \,\subset\, f(H)^{+D_2}. \nonumber
\end{align}
According to \cref{lem:intersect_quasiorthants}, each of the terms in the union in \eqref{eq:big_union} is at finite Hausdorff distance from the coarsely median-convex image $I$ of a quasimedian quasiisometric embedding of a standard $k$--orthant for some $k\le n$. By \cref{lem:coarsely_convex_coarsely_convex}, $I$ is coarsely dense in $\Hull(I)$, and \cref{prop:quasiorthant} now tells us that $\Hull(I)$ contains a coarsely dense orthant. The quasimedian property implies that this orthant is semisingular.
\end{proof}

We conclude this section by spelling out the following corollary of \cref{prop:subflat_image_union_orthants}, which is the case $m=1$.

\begin{corollary} \label{thm:huang_for_asymptotic_rank}
Let $Y$ be a finite-dimensional CAT(0) cube complex of asymptotic rank $n$. Each $n$--quasiflat in $Y$ is at finite Hausdorff distance from a finite union of semisingular $n$--orthants.
\end{corollary}



\section{Coming back from the cone} \label{sec:back_from_cone}

In this section, we give sets of conditions under which an $n$--quasiflat in a finite-dimensional CAT(0) cube complex lies Hausdorff-close to a flat. All three are phrased in terms of properties of the induced biLipschitz flat in the asymptotic cone.

The first two, \cref{prop:quasiflats-from-bilip-flats} and \cref{prop:quasiflat_from_biLipschitz:asymptotic_rank}, are variations of the following idea: for a CAT(0) cube complex $X$, if a $q$--quasiflat is an $n$--flat in some asymptotic cone, then it is uniformly Hausdorff-close to an $n$--flat. In \cref{prop:quasiflats-from-bilip-flats} we assume that $X$ is $n$--dimensional. In \cref{prop:quasiflat_from_biLipschitz:asymptotic_rank} we allow $X$ to only have asymptotic rank $n$, but we obtain constants that depend on the specific space $X$.
We note that for proper CAT(0) spaces, without the assumption that the space is a cube complex, this follows from \cite[Cor.~6.11, Thm~6.12]{huangkleinerstadler:morse:II}.

These two results essentially reduce the problem of showing that a quasiflat is Hausdorff close to a flat to the same problem for biLipschitz flats in the asymptotic cone. The latter will be considered in \cref{sec:flats_rigidity}.


The third main result of this section is \cref{thm:top_quasiflats_union_parallel_codim1_flats}. It assumes the existence of a codimension-1 subflat whose image is an $(n-1)$--flat in the asymptotic cone. This result will be used in \cite{baderbensaidpetyt:quasiisometric:rigidity}.

Lastly, in \cref{thm:quasiflat_join_boundary} we give a condition on the image of a top-rank quasiflat in any asymptotic cone from which we deduce that the Tits boundary of the quasiflat is a join. This result will be used in \cite{baderbensaidpetyt:quasiisometric:products} to obstruct embeddings of right-angled Artin groups into products of trees.

\cref{thm:top_quasiflats_union_parallel_codim1_flats} and \cref{thm:quasiflat_join_boundary} will not be used in this paper, but they fit the theme and use the technical tools developed in this paper.

\subsection{Top-dimensional flats} \label{subsec:back_dim_n}

The proof of the following is based on \cite[Lem.~5.2]{huang:top}; we refer to \cite[\S3.1]{huang:top} for discussion of \emph{support sets} of quasiflats. To give the idea, let $X$ be a metric space that is $n$--dimensional in an appropriate sense, let $f\colon \R^n\to X$ be a quasiflat, and let $x\in f(\R^n)$. Roughly speaking, to determine whether $x$ is in the support set of $f$, one looks in a small neighbourhood $U$ of $x$ and considers how the sheets of $f^{-1}(U)$ map to $U$. If more map with one orientation than the opposite orientation, then $x$ is in the support set. More formally, the support set is defined by considering the \emph{proper homology} class obtained as the image of the fundamental class $[\R^n]$ under an induced map $f_*$ of proper homology.

\begin{proposition}\label{prop:quasiflats-from-bilip-flats}
For each $q$ there exists $D$ such that the following holds. Let $X$ be an $n$--dimensional CAT(0) cube complex, and let $f\colon \mathbb R^n \to X$ be a $q$--quasiisometric embedding. If the ultralimit $(f(\mathbb{R}^n))_{\omega}$ is an $n$--flat of $X_\omega$ for some asymptotic cone of $X$, then there exists an $n$--flat $F\subseteq X$ such that $\dist_{\mathrm{Haus}}(f(\mathbb{R}^n),F)\le D$.
\end{proposition}

\begin{proof}
By \cite[Thm~1.1]{huang:top}, there are $n$--dimensional orthants $O_1, \dots, O_p$ in $X$ such that 
\[
\dist_{\mathrm{Haus}}\Big(f(\mathbb{R}^n),\,\bigcup_{i=1}^p O_i\Big)< \infty.
\]
Since some ultralimit of $f(\R^n)$ is a flat, we must have $p=2^n$. Let $S$ be the support set of $f(\R^n)$. According to \cite[Lem.~4.3]{bestvinakleinersageev:quasiflats}, there is a constant $D$ depending only on $q$ such that $\dist_{\mathrm{Haus}}(S,f(\R^n))\le D$ (an alternative proof can be found in \cite[Lem.~7.3]{haettelhodapetyt:quasiflats}, using the fact that $X$ is an injective metric space when given the $\ell^\infty$ metric). It therefore suffices to show that $S$ is a flat. 

Since $p=2^n$, the boundary $\partial_TS$ is contained in a union of $2^n$ simplices of $\partial_TX$ of dimension $n-1$. Hence the Hausdorff measure $\HH^{n-1}(\partial_TS)$ is bounded above by the Hausdorff measure of the $(n-1)$--sphere $\mathbb S^{n-1}$. Fixing $x\in S$, there is a natural 1--Lipschitz map $\log_x$ from the euclidean cone $C(\partial_TS)$ over $\partial_TS$ to $X$, which sends the cone point $o$ to $x$. By \cite[Lem.~3.1]{bestvinakleinersageev:quasiflats}, we have $S\subset\log_xC(\partial_TS)$. We can therefore estimate Hausdorff measures of balls as follows:
\[
\frac{\HH^n(B(x,r)\cap S)}{r^n} \,\le\, \frac{\HH^n(B(x,r)\cap\log_xC(\partial_TS))}{r^n} 
    \,\le\, \frac{\HH^n(B(o,r)\cap C(\partial_TS))}{r^n}.
\]
Since $\HH^{n-1}(\partial_TS)\le\HH^{n-1}(\mathbb S^{n-1})$, it follows that $\frac{\HH^n(B(x,r)\cap S)}{r^n}$ is bounded above by the volume of the unit ball in $\R^n$. According to \cite[Thm~3.10]{huang:top}, this implies that $S$ is isometric to $\R^n$, completing the proof.
\end{proof}

\begin{remark} \label{rem:bad_orthants}
Note that it is not necessarily true that, in the notation of \cref{prop:quasiflats-from-bilip-flats}, if $A\subseteq\bbR^n$ is one of the $2^n$ orthants whose cone point is the origin and $(f(A))_\omega$ is an $n$--orthant in the $n$--flat $(f(\bbR^n))_\omega$ for all asymptotic cones, then $f(A)$ lies at finite Hausdorff distance from an $n$--orthant in $F$. For example, consider the map $f\colon \R^2\to\R^2$ defined with polar coordinates as follows:
\[
f(r,\theta) \,=\, 
\begin{cases}
(0,0) & r\le10; \\
(r,\theta(1+\frac1r)) & r>10,\,\theta\in[0,\frac\pi2]; \\
(r,\frac\pi r+\theta(1-\frac1r)) & r>10,\theta\in[\frac\pi2,\pi] \\
(r,\theta) & \text{otherwise}.
\end{cases}
\]
It is a quasiisometry that expands the first orthant by a sublinear amount, so on the level of asymptotic cones all orthants map to orthants, even though neither the first nor second orthant is mapped Hausdorff-close to an orthant by $f$.
\end{remark}

\subsection{Top-rank flats}

In this section, we prove \cref{prop:quasiflat_from_biLipschitz:asymptotic_rank}, which is a variation of \cref{prop:quasiflats-from-bilip-flats} assuming that the asymptotic rank of $X$ is $n$, but the dimension can be bigger.

The proof of this is more involved, for two essential reasons. Firstly, if $\dim X>n$, then we cannot use support sets inside $X$ to say anything about $n$--quasiflats, because their definition relies on the ambient space being $n$--dimensional. 
Secondly, although there is still a powerful quasiflats theorem for such cube complexes, due to Bowditch \cite[Thm~1.1]{bowditch:quasiflats}, it is weaker than Huang's theorem \cite[Thm~1.1]{huang:top} in one key way: instead of providing a \emph{subcomplex} formed by orthants, we only have a \emph{map} of a complex $\Omega$ formed by orthants into $X$, where the restriction to each orthant is quasimedian.

We will prove the proposition as follows. First, we use \cite[Thm~1.1]{bowditch:quasiflats} to find a map from a panel complex $\Omega$ to $X$ approximating our quasiflat that is quasimedian on each panel. That panel complex may not be a CAT(0) space, but it sits inside a bigger CAT(0) panel complex $\Psi$. By investigating the construction, and using the asymptotic cone assumption, we find a different panel complex $F\subseteq \Psi$ that: is a flat; is built out of $2^n$ orthants meeting at a point; and approximates our quasiflat via a map $\phi$ that is quasimedian on each orthant.

We show the image of $F$ is close to an $n$--flat using \cref{prop:singular_quasi_flats}. For that, we use \cref{lem:gluing_quasimedian_along_orthant} to show that $\phi$ is globally quasimedian, and then apply the results of \cref{subsec:quasimedian_orthants}. We prove \cref{lem:gluing_quasimedian_along_orthant} by an inductive application of \cref{claim:convex_orthants} and the following lemma. Recall that if $X$ is a CAT(0) cube complex, then $\mu:X^3\to X$ denotes the median operation on $X$. Throughout this section, we equip $\R^n$ with its standard median structure, which we also denote by $\mu$.

\begin{lemma} \label{lem:quasi_median_union}
Let $Y$ and $X$ be finite-dimensional CAT(0) cube complexes. Suppose that $A,B\subset Y$ are convex subcomplexes such that $A\cup B$ is convex.
For every $C,\lambda,q$ there exists $D=D(C,\lambda,q,\dim X)$ such that the following holds. 

If $\phi\colon Y\to X$ is a $q$--quasiisometric embedding such that $\phi|_A$ and $\phi|_B$ are $\lambda$--quasimedian, and both $\phi(A)$ and $\phi(B)$ are $C$--coarsely median-convex, then $\phi|_{A\cup B}$ is $D$--quasimedian. 
\end{lemma}

\begin{proof}
Let $a,b,c\in A\cup B$. After relabelling, we can assume that $a,b\in A$. If $c\in A$, then there is nothing to do, because $\phi_A$ is $\lambda$--quasimedian. So assume that $c\in B$. 

The point $m=\mu(a,b,c)$ lies in $A$, because $A$ is median-convex. Let $\pi_A\colon Y\to A$ be the closest point projection map. That is, for any $y\in Y$, the point $\pi_A(y)$ differs from $y$ on exactly the hyperplanes that separate $y$ from $A$. In particular, $\mu(z,\pi_A(c),c)=\pi_A(c)$ for all $z\in A$.
Note that $\mu(a,b,\pi_A(c))=m$, and also that $\pi_A(c)\in B$. To see the latter, observe that if $z\in A\cap B$ then $\mu(z,\pi_A(c),c)=\pi_A(c)$, because $z\in A$, and it lies in $B$ because $z,c\in B$.

We aim to show that $\mu(\phi(a),\phi(b),\phi(c))$ lies uniformly close to $\phi(m)$, and our strategy will be to show that both lie close to $m'=\mu(\phi(a),\phi(b),\phi(\pi_A(c)))$. We have
\begin{equation}\label{eq:quasimedian_first}
\begin{split}
d\big(\mu(\phi(a),\phi(b),\,\phi(c)),\phi(m)\big) 
    \,&\leq\, d\big(\mu(\phi(a),\phi(b),\phi(c)),\,m'\big)+d\big(m',\phi(m)\big) \\& 
\leq\, d\big(\mu(\phi(a),\phi(b),\phi(c)),\,m'\big)+\lambda,
\end{split}
\end{equation}
where the second inequality holds because $\phi|_A$ is $\lambda$--quasimedian. We are left with bounding the distance from $m'$ to $\mu(\phi(a),\phi(b),\phi(c))$.

Because $\phi(A)$ and $\phi(B)$ are $C$--coarsely median-convex, it follows from \cref{lem:coarsely_convex_coarsely_convex} that they are $(2^{\dim X}C)$--coarsely dense in convex subcomplexes $Q_A$ and $Q_B$ of $X$, respectively.

Let $p=\pi_{Q_A}(\phi(c))$ be the closest-point projection of $\phi(c)$ to $Q_A$. Note that $\mu(\phi(c),y,p)=p$ for all $y\in Q_A$. Also, since $\phi(a),\phi(b)\in\phi(A)\subseteq Q_A$, we have that $\mu(\phi(a),\phi(b),\phi(c))=\mu(\phi(a),\phi(b),p)$. By the fact that the median is 1--Lipschitz in each coordinate, we have 
\begin{equation}\label{eq:quasimedian_second}
d\big(\mu(\phi(a),\phi(b),\phi(c)),\,m'\big)
\,=\,
d\big(\mu(\phi(a),\phi(b),p),\,m'\big)
\,\le\, d\big(p\,,\phi(\pi_A(c))\big).
\end{equation}
It therefore suffices to bound $\dist(p,\phi(\pi_A(c)))$.

First we show that $p$ lies at distance at most $D'=D'(q,C,\dim X)$ from $\phi(A\cap B)$, by finding points $\phi(x_A)\in \phi(A)$ and $\phi(x_B)\in \phi(B)$ at uniformly bounded distance from $p$.

As $p\in Q_A$ and $\phi(A)$ is $2^{\dim X}C$--coarsely dense in $Q_A$, there exists $x_A\in A$ such that $\dist(\phi(x_A),p)\le 2^{\dim X}C$. Next, note that if $z\in A\cap B$, then $\phi(z)\in \phi(A)\subseteq Q_A$, so by the definition of $p$ we have that $\mu(\phi(z),p,\phi(c))=p$. But $\phi(z),\phi(c)\in Q_B$, so we must have $p\in Q_B$ by convexity. Similarly to the existence of $x_A$, there exists $x_B\in B$ such that $\dist(\phi(x_B),p)\le 2^{\dim X}C$. By the triangle inequality, $\dist(\phi(x_A),\phi(x_B))\le2\cdot 2^{\dim X}C$.

Since $\phi$ is a $q$--quasiisometric embedding, $\dist(x_A,x_B)\le q\cdot 2\cdot2^{\dim X}C+q^2$. Hence, by considering the interval $[x_A,x_B]$, there is a point $x\in A\cap B$ at distance at most $2q\cdot 2^{\dim X}C+q^2$ from both $x_A$ and $x_B$. By the triangle inequality, we then have
\begin{equation*}
    \begin{split}
        \dist(p,\phi(x)) \,\le\, \dist(p,\phi(x_A))&+\dist(\phi(x_A),\phi(x)) \,\le\, 2^{\dim X}C+q\dist(x_A,x)+q 
    \,\\&\le\,  (1+2q^2)2^{\dim X}C+q^3+q =D'.
    \end{split}
\end{equation*}
Thus $p$ lies at distance at most $D'$ from $\phi(x)\in\phi(A\cap B)$. We shall use the point $x$ to help bound $\dist(p,\phi(\pi_A(c)))$. Indeed, since $x\in A$ we have $\mu(x,\pi_A(c),c)=\pi_A(c)$. Because all three of $x$, $\pi_A(c)$, and $c$ lie in $B$ and $\phi|_B$ is $\lambda$--quasimedian, we therefore have 
\[
\dist\big(\mu(\phi(x),\phi(\pi_A(c)),\phi(c)),\,\phi(\pi_A(c))\big) \,\le\, \lambda.
\]
Similarly, since $\phi(\pi_A(c))\in Q_A$ and $p=\pi_{Q_A}(\phi(c))$, we have $\mu(p,\phi(\pi_A(c)),\phi(c))=p$.
Since $\mu$ is 1--Lipschitz in each coordinate, we can combine these and the triangle inequality to compute
\begin{align}
\dist(p,\phi(\pi_A(c))) 
    \,&\le\, \dist\big(\mu(p,\phi(\pi_A(c)),\phi(c)),\,\mu(\phi(x),\phi(\pi_A(c)),\phi(c))\big)
        \,+\, \dist\big(\mu(\phi(x),\phi(\pi_A(c)),\phi(c)),\,\phi(\pi_A(c))\big) \nonumber \\
    &\le\, \dist(p,\phi(x))+\lambda \,\le\, D'+\lambda. \label{eq:quasimedian_third}
\end{align}




We conclude from \eqref{eq:quasimedian_first}, \eqref{eq:quasimedian_second}, and \eqref{eq:quasimedian_third} that 
\[
d\big(\mu(\phi(a),\phi(b),\phi(c)),\phi(m)\big)\le D'+2\lambda=(1+2q^2)2^{\dim X}C+q^3+q+2\lambda=D.
\] 
This proves that $\phi|_{A\cup B}$ is $D$--quasimedian.
\end{proof}

The following lemma uses an iterative application of \cref{lem:quasi_median_union} and \cref{claim:convex_orthants} to show how to upgrade piecewise information about a quasiflat to global information.

\begin{lemma}\label{lem:gluing_quasimedian_along_orthant}
Let $X$ be a finite-dimensional CAT(0) cube complex of asymptotic rank $n$. Write $\mathbb{R}^n=\bigcup_{i=1}^{2^n}O_i$ where $\{O_i\}_{i=1}^{2^n}$ are the singular $n$--orthants with cone point~0. For every $q\ge1$ and $m\ge0$ there exists $m'=m'(m,q,n,X)$ such that the following holds. 

If $\phi\colon \mathbb{R}^n\to X$ is a $q$--quasiisometric embedding such that $\phi|_{O_i}$ is $m$-quasimedian for each $i\in \{1,\dots,2^n\}$, then $\phi$ is $m'$--quasimedian. 
\end{lemma}

\begin{proof}
Every orthant $O_i$ is convex in $\Omega$, and $\phi|_{O_i}$ is $m$--quasimedian. By \cref{claim:convex_orthants}, there exists $D>0$ such that $\phi(O_i)$ is $D$--coarsely median-convex. Given $i\in\{1,\dots,2^n\}$, we can write $O_i=\prod_{j=1}^n\eps_j[0,\infty)$, where each $\eps_j\in\{-1,1\}$ and we interpret $-[0,\infty)$ as meaning $(-\infty,0]$. 

At the first stage of the process, we pair orthants that disagree only in the first coordinate, obtaining subcomplexes $P_i$ of the form $\R\times\prod_{j=2}^n\eps_j[0,\infty)$, for $i\in\{1,\dots,2^{n-1}\}$. By \cref{lem:quasi_median_union}, there exists $m_1$ such that $\phi|_{P_i}$ is $m_1$--quasimedian for every such $P_i$. \cref{claim:convex_orthants} now tells us that there exists $D_1$ such that $\phi(P_i)$ is $D_1$--coarsely median-convex. We can therefore repeat the argument with the $P_i$ in place of the $O_i$, pairing subcomplexes that differ only in the second coordinate and obtaining subcomplexes of the form $\R^2\times\prod_{i=3}^n\eps_j[0,\infty)$. After iterating this process $n$ times, we obtain $\mathbb{R}^n$ and and the quasimedian constant.
\end{proof}

We can now prove our analogue of \cref{prop:quasiflats-from-bilip-flats} for higher-dimensional cube complexes. We will need the following definition from \cite{bowditch:quasiflats}.

\begin{definition}[Panel complex] \label{def:panel}
A \emph{panel} is a subset of $\R^n$ that is a product of nontrivial, closed, connected, proper subsets of $\R$. A \emph{face} of a panel is obtained by restricting some of the factors to a point in their boundary in $\R$. A \emph{panel complex} is a complete geodesic space that is a finite union of subsets, called \emph{cells}, such that: with the induced metric, each cell is isometric to a panel; and if two cells intersect, then they intersect in a common face.
\end{definition}

\begin{theorem} \label{prop:quasiflat_from_biLipschitz:asymptotic_rank}
Let $X$ be a finite-dimensional CAT(0) cube complex of asymptotic rank $n$. For each $q$ there exists $K=K(n,q,X)$ such that the following holds. Suppose that $f\colon \R^n\to X$ is a $q$--quasiisometric embedding such that in some asymptotic cone $X_\omega$ the ultralimit $(f(\R^n))_\omega$ is an $n$--flat. There exists an $n$--flat $F\subseteq X$ at Hausdorff distance at most $K$ from $f(\R^n)$.
\end{theorem}

\begin{proof}
\cref{lem:rank_cone_CCC} and \cite[Rem.~3.3]{munropetyt:coarse} tell us that $X$ is a \emph{coarse median space of rank $n$}. Hence \cite[Thm~1.1]{bowditch:quasiflats} shows that there exist: constants $q'$, $m$, and $C$, each depending only on $X$, $q$, and $n$; a finite panel complex $\Omega=\bigcup_{i=1}^pP_i$ of dimension $n$; and a $q'$--quasiisometric embedding $\phi\colon \Omega\to X$ such that $\dist_{\mathrm{Haus}}(f(\R^n),\phi(\Omega))\leq C$ and the restrictions $\phi|_{P_i}$ are $m$--quasimedian. Moreover, the value of $p$ depends only on $X$, $n$, and $q$, and not on the specific map $f$. 

We shall use additional information given in the construction of $\Omega$, see \cite[p.48]{bowditch:quasiflats}, to find a panel complex $F$ that is an $n$--flat and satisfies the same list of properties as $\Omega$. 

The panel complex $\Omega$ is constructed in two steps. First, using the map $f\colon \R^n\to X$, Bowditch constructs an $n$--dimensional panel complex $\Psi$ and a quasiisometric embedding $\theta\colon \R^n\to\Psi$ with uniform constant. The complex $\Psi$ has the additional property that it is CAT(0) when each panel is equipped with the $\ell^2$ metric. From this, Bowditch applies \cite[Lem.~6.1]{bowditch:quasiflats}, which states that any $n$--quasiflat in a finite, $n$--dimensional, CAT(0) panel complex can be uniformly perturbed so that its image is a uniformly coarsely dense subset of a union of $n$--panels. This perturbed subcomplex is $\Omega$, and it can potentially fail to be CAT(0). In other words, by composing $\theta$ with a closest-point map, we obtain a $q''$--quasiisometry $\theta':\R^n\to\Omega$, where $q''=q''(n,q,X)$. Let $\bar\theta'$ be a quasiinverse of $\theta'$. The map $\phi$ is defined by setting $\phi=f\bar\theta'$. Bowditch uses the properties of his construction to show that $\dist_{\mathrm{Haus}}(f(\mathbb{R}^n),\phi(\Omega))\leq C$. 

Now, the map $\phi$ induces a biLipschitz map $\phi_\omega\colon \Omega_\omega\to X_\omega$ with image $(f(\R^n))_\omega$, and the restriction of $\phi_\omega$ to each panel of $\Omega_\omega$ is median-preserving. Since $(f(\R^n))_\omega\subseteq X_\omega$ is an $n$--flat, this implies that $\Omega_\omega$ has exactly $2^n$ panels of dimension $n$, each of which is a standard orthant in $\R^n$. Consequently, $\Omega$ has exactly $2^n$ panels isometric to $[0,\infty)^n$, and each of the finitely many other panels has a finite factor. After relabelling, we can assume that the panels $P_1,\dots,P_{2^n}$ are all isometric to $[0,\infty)^n$. 

By the ``moreover'' statement of \cite[Thm~1.1]{bowditch:quasiflats}, we therefore have that
\[
\dist_{\mathrm{Haus}}\big(\theta'(\R^n),\,\bigcup_{i=1}^{2^n}P_i\big) \,<\, \infty.
\]
Because of this, the same argument as in \cref{prop:quasiflats-from-bilip-flats} shows that $\theta'(\R^n)$ is at uniform Hausdorff distance from its support set $F\subset\Psi$, which is an $n$--flat. Since $\dim F=\dim\Psi$, it follows that $F$ is a union of $n$--dimensional panels of $\Psi$. 
The median structure on $\Psi$ restricts to the standard median on the flat $F$, so we can view $F$ as being equal to $\mathbb{R}^n$, written as the union of the $2^n$ orthants $O_1,\dots,O_{2^n}$ with cone point $0$, and $\phi|_{O_i}$ is $m$--quasimedian. 

The fact that $\theta'(\R^n)$ is at uniform Hausdorff distance from $F$ implies that there exists $q'=q'(n,q,X)$ such that $\phi|_F$ is a $q'$--quasiisometric embedding. Moreover, $f(\R^n)$ is within uniform Hausdorff distance of $\phi(F)$. It therefore suffices to show that $\phi(F)$ lies within uniform Hausdorff distance of an $n$--flat of $X$. 


By \cref{lem:gluing_quasimedian_along_orthant}, there exists $m'=m'(m,q',n,X)$ such that $\phi$ is $m'$--quasimedian.
We will show that there is a constant $K'=K'(m',q',n,X)$ such that $\phi$ sends each singular geodesic in $F$ within Hausdorff distance $K'$ of a geodesic in $X$. This, together with \cref{prop:singular_quasi_flats}, will show that there is a constant $K=K(q',n,K')$ such that $\phi(F)$ lies within Hausdorff distance $K$ of an $n$--flat, which will conclude the proof.

Let $\gamma\subseteq F$ be a singular geodesic line. By \cref{claim:convex_axes}, $\phi(\gamma)$ is $D$--coarsely median-convex, where $D=D(m',q',n,X)$. Therefore, by \cref{lem:coarsely_convex_coarsely_convex}, $\phi(\gamma)$ lies at Hausdorff distance at most $2^{\dim X}D$ from its hull $L$. Thus $\phi\colon \gamma\to L$ is a $(q',q'+2^{\dim X}D)$--quasiisometry. In particular, the finite-dimensional CAT(0) cube complex $L$ is a quasiline, so it contains a biinfinite CAT(0) geodesic $\alpha$, by \cref{prop:lines}.

If $\bar\phi$ is a $q'$--quasiinverse of $\phi$, then $\bar\phi|_{\alpha}\colon \alpha\to \gamma$ is a $q'$--quasiisometric embedding, and so there exists $M=M(q')$ such that it is $M$--coarsely surjective, by \cite[Lem.~10.84]{drutukapovich:geometric}. Given $y\in L$, there exists $x\in\gamma$ such that $\dist(y,\phi(x))\le 2^{\dim X}D$. There then exists $z\in\alpha$ such that $\dist(\bar\phi(z),x)\le M$, and we have
\[
\dist(y,z) \,\le\, \dist(y,\phi(x))+\dist(\phi(x),z) \,\le\, 2^{\dim X}D+\dist(\phi(x),\phi\bar\phi(z))+q' 
    \,\le\, 2^{\dim X}D+Mq'+2q'.
\]
In particular, $\phi(\gamma)$ lies at Hausdorff distance at most $K'=2^{1+\dim X}D+Mq'+2q'$ from the geodesic $\alpha$.
\end{proof}

\subsection{Codimension 1 assumption}

Here we record two further ``coming back from the cone'' results, for use in \cite{baderbensaidpetyt:quasiisometric:rigidity} and \cite{baderbensaidpetyt:quasiisometric:products}. Both rely on additional information involving a codimension-one subflat. \Cref{thm:top_quasiflats_union_parallel_codim1_flats} shows that if, in an asymptotic cone, all parallels of a fixed codimension-one subflat are mapped to flats, then the original quasiflat is Hausdorff-close to a flat. \Cref{thm:quasiflat_join_boundary} treats a splitting case, and goes beyond the codimension $1$ assumption: if, for a quasiisometric embedding $f\colon H\times K\to Y$, both quasiflats $f(H)$ and $f(K)$ are Hausdorff-close to finite unions of orthants of the corresponding dimensions, then the Tits boundary of the quasiflat $f(E)$ decomposes as the corresponding spherical join.

Following \cite{kleiner:local}, we define the \emph{geometric dimension} of a CAT(0) space $X$ to be 
$$ 
\operatorname{gdim}X = \sup\{\dim_{\mathrm{top}}K, \ K\subseteq X\text{ compact}\}.  
$$
A CAT(0) space has geometric dimension at most $1$ if and only if it is an $\mathbb R$--tree. By \cite[Thm~3.4]{wenger:asymptotic}, if $X$ is a complete CAT(0) space of asymptotic rank $n$, then every asymptotic cone of $X$ has geometric dimension at most $n$.

\begin{lemma}\label{prop:biLip_flat_union_of_parallel_subflats}
    Let $n\geq1$, let $E=\mathbb R^n$, and let $H\subset E$ be a $p$--flat, where $p\le n$. Let $Y$ be a complete CAT(0) space, and let $f\colon E\to Y$ be a biLipschitz embedding. Assume that, for every $p$--flat $H'\subset E$ parallel to $H$, the image $f(H')$ is a $p$--flat. Write $P(f(H))=f(H)\times A$ for the parallel set decomposition of $f(H)$. Then there is a closed subset $C\subseteq A$, biLipschitz to $\mathbb R^{n-p}$, such that
    $$
    f(E)=f(H)\times C.
    $$
    Moreover, if $p=n-1$ and $\operatorname{gdim}Y\leq n$, then $f(E)$ is an $n$--flat.
\end{lemma}

\begin{proof}
    Choose an orthogonal complementary flat $K$ to $H$, so that $E = H \times K$. Set $F_0=f(H)$, and for each $k \in K$, set $F_k = f(H\times\{k\})$. Let $L$ be a biLipschitz constant for $f$. For every $k,k'\in K$, we have
    \begin{equation}\label{eq:Haus_distance_between_flat_images}
        \frac{1}{L} d_K(k,k') \leq d(F_k,F_{k'}) \leq d_{\mathrm{Haus}} (F_k,F_{k'}) \leq Ld_K(k,k').
    \end{equation}
    Therefore, the flats $F_k$ are pairwise parallel and, in the canonical decomposition $P(F_0)=F_0\times A$, for each $k \in K$ there is a unique $a_k \in A$ such that $F_k=F_0\times\{a_k\}$. Moreover, it follows from \eqref{eq:Haus_distance_between_flat_images} that the map $k\mapsto a_k$ is biLipschitz. Hence, 
    $$f(E)=F_0\times C,$$
    where $C=\{a_k:k\in K\}$ is biLipschitz to $K$.

    Suppose now that $p=n-1$ and $\operatorname{gdim}Y\leq n$. Since $\operatorname{gdim} F_0 = n-1$, it follows that $\operatorname{gdim} A \leq 1$, and $A$ is therefore an $\mathbb R$--tree. Since $C$ is biLipschitz to $\mathbb R$, it is therefore a geodesic line, and $f(E)=F_0\times C$ is an $n$--flat.
\end{proof}


\begin{theorem}\label{thm:top_quasiflats_union_parallel_codim1_flats}
Let $Y$ be either a finite-dimensional CAT(0) cube complex or a proper, complete CAT(0) space. Let $n$ be the asymptotic rank of $Y$, and let $E = \mathbb R^n$. Suppose that $f\colon  E \to Y$ is a quasiisometric embedding. Let $E_\omega$ be an asymptotic cone of $E$ with respect to some fixed basepoint, and let $f_\omega\colon  E_\omega \to Y_\omega$ be the induced biLipschitz embedding. Let $H\subset E$ be an $(n-1)$--flat.

If $f_\omega(H_\omega')$ is an $(n-1)$--flat in $Y_\omega$ for every $H_\omega' \subseteq E_\omega$ parallel to $H_\omega$, then $f(E)$ lies at finite Hausdorff distance from some $n$--flat in $Y$.
\end{theorem}

\begin{proof}
    %
%
    Since $Y$ has asymptotic rank $n$, its asymptotic cone satisfies $\operatorname{gdim}Y_\omega\leq n$ by \cite[Thm~3.4]{wenger:asymptotic}. Therefore, by \Cref{prop:biLip_flat_union_of_parallel_subflats}, $f(E)_\omega$ is an $n$--flat. The conclusion follows from \cref{prop:quasiflat_from_biLipschitz:asymptotic_rank} in the cube complex case, and from \cite[Cor.~6.11, Thm~6.12]{huangkleinerstadler:morse:II} in the proper CAT(0) case.
\end{proof}



We now turn to the second result, starting with a proposition about asymptotic cones.

\begin{proposition}\label{prop:bilip_flat_join_boundary}
Let $n$ and $p$ be integers with $1\leq p\leq n-1$. Let $E=H\times K=\mathbb R^{p}\times\mathbb R^{n-p}$, and let $Y$ be an asymptotic cone of a finite-dimensional CAT(0) cube complex. Let $f\colon  E \to Y$ be a biLipschitz embedding. Assume:
\begin{enumerate}
    \item for every $p$--flat $H'\subset E$ parallel to $H$, its image $f(H')$ is a singular $p$--flat;
    \item $f(K)$ is a finite union of singular $(n-p)$--orthants.
\end{enumerate}
Then $f(E)$ is contained in the parallel set $P(f(H))$, and the induced metric on $f(E)$ splits isometrically as
$$f(E)= f(H)\times f(K).$$
\end{proposition}

\begin{proof}
Set $L =f(H)$. By \Cref{prop:biLip_flat_union_of_parallel_subflats}, we have
$$ f(E) = L \times C \subseteq P(L) = L \times A,$$
for some closed subset $C \subseteq A$. Let us show that $C = f(K)$. Let $o=f(H\cap K)$, and choose, without loss of generality, the splitting $P(L) = L \times A$ so that $o\in L$ and $o\in A$.

First note that the projection $\pi_A\colon f(K)\to A$ is injective. Indeed, if two distinct points of $f(K)$ had the same $A$--coordinate, then the two corresponding parallels of $H$ in
$E$ would have the same images, contradicting the injectivity of $f$.

Next we claim that $f(K)\subseteq A$. By assumption, $f(K)$ is a finite union of singular $(n-p)$--orthants. Since $L$ is singular, every axis ray of each such orthant lies either in the $L$--factor or in the $A$--factor. No axis ray can lie in
the $L$--factor, otherwise its $A$--projection would be constant, which would contradict the injectivity of $\pi_A|_{f(K)}$. Hence every axis ray lies in the $A$--factor. Therefore, each orthant of $f(K)$ is contained in $\{x\}\times A$, for some $x \in L$. Since $f(K)$ is connected, all these orthants lie in the same $A$--copy. Since $o\in f(K)$ and $o\in A$, this copy is exactly $A$, and therefore $f(K)\subseteq A$. Now, for each $k\in K$, the flat $H\times\{k\}$ maps to the parallel copy of $L$
passing through $f(k)$. Since $f(k)\in A$, this copy is exactly $L\times\{f(k)\}$. Finally, by taking the union we get
\[f(E)= \bigcup_{k\in K} f(H\times \{k\})= \bigcup_{k\in K} L\times\{f(k)\}=L\times f(K). \qedhere\]
\end{proof}

The following is the second main result of this subsection. Although our main application, motivated by \cite[\S3]{baderbensaidpetyt:quasiisometric:products}, concerns the codimension $1$ case, the argument works in arbitrary codimension, and we record the more general statement. Recall from \Cref{sec:structure_quasi_flats} that if a subset $A$ lies at finite Hausdorff distance from a finite union of orthants, then $\partial_TA$ denotes the union of the Tits boundaries of those orthants.

\begin{theorem}\label{thm:quasiflat_join_boundary}
    Let $p$ and $n$ be integers with $1\leq p\leq n-1$. Let $E=H\times K=\mathbb R^{p}\times\mathbb R^{n-p}$, and let $Y$ be a finite-dimensional CAT(0) cube complex of asymptotic rank $n$. 
    Let $f\colon E\to Y$ be a quasiisometric embedding. Assume:
    \begin{enumerate}
    \item Each of $f(H)$ and $f(K)$ lies at finite Hausdorff distance from a finite union of semisingular orthants;
    \item For some asymptotic cone taken with fixed basepoints, the induced biLipschitz embedding $f_\omega\colon E_\omega\to Y_\omega$ sends every $p$--flat in $E_\omega$ parallel to $H_\omega$ to a singular $p$--flat.
\end{enumerate}
Then 
$$ \partial_T f(E) = \partial_T f(H) * \partial_T f(K).$$
\end{theorem}

\begin{proof}
    By \Cref{thm:huang_for_asymptotic_rank}, $f(E)$ lies at finite Hausdorff distance from a finite union of semisingular $n$--orthants, so $\partial_T f(E)$ is well-defined. The first assumptions similarly ensures that $\partial_T f(K)$ and $\partial_T f(H)$ are well-defined. Moreover, $f_\omega(K_\omega)$ is a finite union of singular $(n-p)$--orthants. Therefore, \Cref{prop:bilip_flat_join_boundary} applies to the biLipschitz embedding  $f_\omega\colon E_\omega=H_\omega\times K_\omega\to Y_\omega$. Hence
    \begin{equation} \label{eq:boundary_join}
    \partial_T f_\omega(E_\omega) = \partial_T f_\omega(H_\omega) * \partial_T f_\omega(K_\omega). 
    \end{equation}
    Let $\varphi_T\colon \partial_TY\to\partial_TY_\omega$ be the isometric embedding from \Cref{lem:embedding-Titsboundary-into_link_and_Tits_boundary}. Let $A\in\{E,H,K\}$, and let $Z=\bigcup_{i=1}^m O_i$ be a union of semisingular orthants at finite Hausdorff distance from $f(A)$. Since each $O_i$ is isometric to the Euclidean cone over its Tits boundary, the ultralimit $(O_i)_\omega$ is isometric to $O_i$, hence is again an orthant. Each $\xi\in \partial_T Z_\omega$ is represented by a ray contained in some orthant $(O_i)_\omega$. Since every ray in $(O_i)_\omega$ is the ultralimit of a ray in $O_i$, we deduce that $\xi\in\varphi_T(\partial_T(O_i))$. Hence $\varphi_T\colon \partial_T Z\to \partial_T Z_\omega$ is surjective, and therefore an isometry. Thus $\varphi_T ( \partial_Tf(A))=\partial_Tf_\omega(A_\omega)$ for each $A\in\{E,H,K\}$. The conclusion now follows from \eqref{eq:boundary_join}.
%
\end{proof}

\section{Local separation of cubulated flats}\label{sec:top_arguments}

In \cref{sec:flats_rigidity} we will give conditions under which a quasiisometric embedding induces a biLipschitz map of asymptotic cones that sends singular geodesics to singular geodesics. 
This short section provides an important topological tool for those arguments, \cref{prop:top_argument}.
More concretely, we give conditions under which a topological embedding of the standard cubulation of a flat locally maps a singular geodesic to a singular geodesic. 
We shall need some terminology from \cite{bowditch:large:mapping}.

\begin{definition}[Cubulated] \label{def:cube_cubulation_in_median}
Let $(M,d)$ be a median metric space. A \emph{$k$--cube} in $M$ is a closed, median-convex subset isometric to a finite $\ell^1$–product of compact intervals $\prod_{i=1}^k [a_i,b_i]$.
A closed subset $C \subseteq M$ is said to be \emph{cubulated} if it is a locally finite union of cubes. 
\end{definition}

\begin{definition}[Singularity set]\label{def:regular_singular_sets_mediancase}
Let $M$ be a median metric space of rank $n$, and let $f\colon  \mathbb R^n \to M$ be a biLipschitz embedding. A point $x \in \mathbb R^n$ is called \emph{flat} for $f$ if there exists $r>0$ such that $f(B(x,r))$ is contained in an $n$--cube in $M$. The \emph{singularity set} of $f$ is the set of non-flat points in $\R^n$.
\end{definition}

Let $\Sigma_0$ be the 0--sphere, and for each $n\in\mathbb{N}$ let $\Sigma_n$ denote the simplicial complex given by the $n$--fold suspension of $\Sigma_0$. Equivalently, $\Sigma_n$ is the link of a vertex in the standard cubical structure on $\mathbb{R}^{n+1}$. See \cref{fig:spherical_A_1} for $n=2$.


\begin{figure}[ht]
\includegraphics[height=2.5cm]{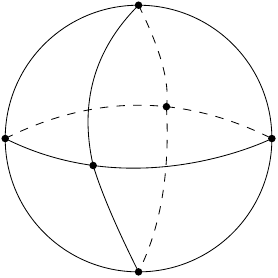}\centering
\caption{The simplicial complex $\Sigma_2$.} \label{fig:spherical_A_1}
\end{figure}

\begin{lemma}\label{lem:separation_sphere_antipodal}
Let $n\in\mathbb{N}$, and let $u$ and $v$ be vertices of $\Sigma_n$. Let $H\subseteq \Sigma_n$ be a subset contained in the $(n-1)$--skeleton such that $H \cap \{u,v\} = \varnothing$. If $u$ and $v$ lie in different connected components of $\Sigma_n \setminus H$, then they are antipodal.
\end{lemma}

\begin{proof}
By the definition of $\Sigma_n$, if $u$ and $v$ are non-antipodal vertices, then they are neighbours and there is an $n$--simplex $\sigma$ containing both of them. Since $H$ is contained in the $(n-1)$--skeleton of $\Sigma_n$, it is disjoint from the interior $\mathring\sigma$ of $\sigma$. Choosing a path $\gamma\colon[0,1]\to \sigma$ from $u$ to $v$ such that $\gamma((0,1))\subseteq \mathring\sigma$, we see that $\gamma\subseteq \Sigma_n\setminus H$. Thus $u$ and $v$ lie in the same connected component of $\Sigma_n\setminus H$.
\end{proof}

\begin{proposition}\label{prop:top_argument}
Let $E=\mathbb{R}^n$, equipped with its median $\ell^1$ metric. Let $x \in E$, let $\gamma \subseteq E$ be a geodesic, and let $H \subseteq E$ be an $(n-1)$--flat such that $\gamma \cap H = \{x\}$. Let $U \subseteq E$ be an open neighbourhood of $x$, and let $f\colon  U \to E$ be a topological embedding. 

If $f(\gamma \cap U)$ and $f(H \cap U)$ are both cubulated, then there exists a neighbourhood $V \subseteq U$ of~$x$ such that $f(\gamma \cap V)$ is contained in a single $1$--cube.
\end{proposition}

\begin{proof}
For $n\le1$ there is nothing to prove. Otherwise, set
$$
A=f(\gamma\cap U), \qquad B=f(H\cap U), \qquad p=f(x).
$$
Since $A$ and $B$ are cubulated, there exists a neighbourhood $W\subset E$ of $p$ such that each of $A \cap W$, and $B \cap W$ is a finite union of cubes containing $p$. Moreover, up to subdividing, we may assume that $p$ is a vertex of each cube. Since $A$ is homeomorphic to an interval near $p$, we may assume that it is the union of exactly two $1$-cubes containing $p$, and write $A\cap W=e_1\cup e_2$. Similarly, $B$ is homeomorphic to $\mathbb R^{n-1}$ near $p$, so we may assume that $B \cap W=\bigcup_{j=1}^m Q_j$, where each $Q_j$ has dimension at most $n-1$.

Since $f$ is a topological embedding, $f(U)$ is open. Hence there exists a ball $D$, with respect to the $\ell^2$ metric, centred on $p$ such that $D \subseteq W \cap f(U)$. Let $S= \partial D$, endowed with the standard spherical complex structure, so that $S$ is naturally identified with $\Sigma_{n-1}$.
    
Since $A\cap W=e_1\cup e_2$, the set $S \cap A$ consists of two vertices $\xi_1,\xi_2$ of $S$. Since each $Q_j$ contains $p$ and has dimension at most $n-1$, the set $S \cap B$ is contained in the $(n-2)$--skeleton of $S$. Moreover, $H$ separates $\gamma$ at $x$, and $f$ is a local homeomorphism at $x$, hence the subset $D \cap B$ separates $e_1$ and $e_2$ at $p$. Therefore, $\xi_1$ and $\xi_2$ lie in different connected components of $S \setminus (S \cap B)$. By \cref{lem:separation_sphere_antipodal}, the vertices $\xi_1,\xi_2$ are antipodal. We conclude that $(e_1\cup e_2)\cap D$ forms a single $1$--cube. Take $V=f^{-1}(D)$.
\end{proof}


\section{Branching flats and biLipschitz embeddings} \label{sec:flats_rigidity}

Any quasiisometric embedding between spaces induces a biLipschitz embedding between their asymptotic cones. In this section we give conditions that allow us to control the images of certain flats in the asymptotic cone under this biLipschitz embedding.

More precisely, in \cref{subsec:def_of_branching} we introduce the notion of a geodesic being \emph{branch-complemented}. In \cref{subsec:ultralimit_of_fully_branching_geod}, we show that this property and related properties for flats are preserved by passing to asymptotic cones. Our key result is then \cref{prop:fully_branching_implies_geod_in_cone}, which shows that the images of such flats under biLipschitz embeddings are again flats. We summarise this in \cref{cor:fully_branching_implies_flat_in_cone} for later use.
In \cref{sec:fully_branching_theorems} we will use these results about asymptotic cones to deduce similar statements about the original quasiisometry.

The argument is roughly as follows. In order to obtain some form of rigidity, we will use the structure of top-dimensional biLipschitz flats in the asymptotic cone, as given by \Cref{prop:Bowditch_bilipflat_cubulated}. From the results in \cref{sec:structure_quasi_flats}, we already know that we can pass some of their structure down to intersections, and in fact this works more smoothly in asymptotic cones because we are now dealing with fine properties rather than coarse ones. This leads to the notion of \emph{branching} flats. With more information about a branching geodesic, we can use the methods of \cref{sec:top_arguments} to show that a biLipschitz embedding must send it to a singular geodesic. This is formulated using the notion of \emph{branch-complemented} geodesics. We can then control the images of flats that are spanned by such geodesics. Again, the fine setting makes this easier than in \cref{sec:singular_quasiflats}.

\subsection{Branching and branch-complementing} \label{subsec:def_of_branching}

Although the definitions in this subsection are given in the general CAT(0) setting, it is worth keeping in mind that our main applications are for CAT(0) cube complexes, and especially right-angled Artin groups. 

\begin{definition}\label{def:branching_flats}
Let $X$ be a CAT(0) space of asymptotic rank $n$. We say that a flat $H \subseteq X$ is \emph{branching} if $H$ is the intersection of finitely many $n$--flats.
\end{definition}

In particular, every $n$--flat is branching. Note that branching geodesics can fail to be singular. For instance, let $X$ be the CAT(0) cube complex from \cref{eg:corner_to_corner}, let $T$ be the 3--regular tree, and consider $X\times T$. Each geodesic in $X\times T$ whose projection to $T$ is a vertex is branching but not singular. 

\begin{lemma} \label{lem:branching_singular}
Let $X$ be a finite-dimensional CAT(0) cube complex of asymptotic rank $n$. Every branching flat in $X$ is semisingular. If $\dim X=n$, then every branching flat is singular.
\end{lemma}

\begin{proof}
In the case $\dim X=n$, \cref{lem:top_dim_singular} shows that every $n$--flat in $X$ is singular, so every flat that is an intersection of $n$--flats is singular.

In the general case, \cref{lem:top_rank_semisingular} shows that every $n$--flat in $X$ is semisingular. If $H$ is a flat that is an intersection of finitely many $n$--flats, then for every asymptotic cone $\hat X$ of $X$ for which the ultralimit $\hat H$ exists, \cref{lem:intersection_flats_cone} shows that $\hat H$ is the intersection of the ultralimits of those flats. Hence $H$ is semisingular.
\end{proof}

\begin{definition}[Branch-complemented]\label{def:fully_branching_CAT(0)}
Let $X$ be a CAT(0) space of asymptotic rank $n\geq 1$. For $r\ge0$, a geodesic $\gamma \subseteq X$ is \emph{$r$--branch-complemented} if there exists a pair $(F,H)$ satisfying the following:
\begin{itemize}
\item   $F\subseteq X$ is an $n$--flat containing $\gamma$;
\item   $H\subseteq F$ is a branching $(n-1)$--flat transverse to $\gamma$;
\item   every $x\in F$ lies at distance at most $r$ from a branching geodesic $\gamma'\subseteq F$ that is parallel to $\gamma$ and a branching $(n-1)$--flat $H'\subseteq F$ that is parallel to $H$.
\end{itemize}
We say that $\gamma$ is \emph{branch-complemented} if it is $r$--branch-complemented for some $r$.
\end{definition}

Note that if $\gamma$ is $r$--branch-complemented, then it lies at Hausdorff distance at most $r$ from a branching geodesic. The uniformity of $r$ ensures that ultralimits of branch-complemented geodesics are $0$--branch-complemented; see \Cref{lem:ultralimit_fully_branching_geod}.

\begin{remark} \label{rem:fully_branching_geod_in_raag}
When $X$ is the universal cover of a Salvetti complex, if a flat is branching, then all of its parallels at integer distance are branching too. Hence if a geodesic is $r$--branch-complemented, then it is 1--branch-complemented. More generally, if $X$ is a cobounded CAT(0) space and $F$ is a periodic flat, then to show that $\gamma$ is branch-complemented it suffices to show that it is branching and has a single transverse $(n-1)$--subflat of $F$ that is branching.
\end{remark}

We will be interested in flats that are comprised of branch-complemented geodesics.

\begin{definition}[Directionally branch-complemented]\label{def:fully_branching_flats}
Let $X$ be a CAT(0) space of asymptotic rank $n\geq 1$. For $r\ge0$, a $k$--flat $F\subseteq X$ is \emph{directionally $r$--branch-complemented} if it is spanned by geodesics $\gamma_1,\dots,\gamma_k$ such that $F$ is $r$--coarsely covered by $r$--branch-complemented parallels in $F$ of $\gamma_i$, for each $i$. 

We say that $F$ is \emph{directionally branch-complemented} if it is directionally $r$--branch-complemented for some $r$.
\end{definition}

\begin{lemma} \label{lem:dbc_semisingular}
If $X$ is a finite-dimensional CAT(0) cube complex, then every directionally branch-complemented flat in $X$ is semisingular. If the dimension and asymptotic rank of $X$ are equal, then every directionally branch-complemented flat in $X$ is singular.
\end{lemma}

\begin{proof}
Each of the geodesics $\gamma_1,\dots,\gamma_k$ from \cref{def:fully_branching_flats} is branching, and hence is (semi)singular by \cref{lem:branching_singular}, so the flat they span is also (semi)singular.
\end{proof}

Every branch-complemented geodesic is a directionally branch-complemented 1--flat, but it should be noted that directionally branch-complemented $k$--flats need not be branching when $k\ge2$. Even if $F$ is an $n$--dimensional directionally branch-complemented flat, it can happen that the $n$--flats $F_i$ witnessing that the $\gamma_i$ are branch-complemented are necessarily distinct from $F$. See \cref{fig:dbc_not_branching}.

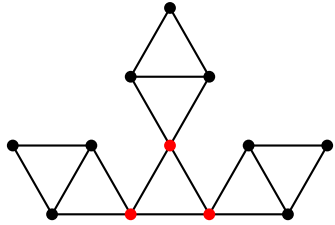
\begin{figure}[ht]
\centering\begin{tikzpicture}[baseline=-.1cm, scale=1.3]
\draw[thick]
(-1.6,1.05) -- (-0.8,1.05)
(0.8,1.05) -- (1.6,1.05)
(-1.2,.35) -- (1.2,.35)
(-0.4,1.75) -- (0.4,1.75);
\draw[thick]
(-1.6,1.05) to (-1.2,.35) to (-.8,1.05) to (-.4,.35) to (0,1.05) to
(.4,.35) to (.8,1.05) to (1.2,.35) to (1.6,1.05);
\draw[thick]
(0,1.05) to (-0.4,1.75) to (0,2.45) to (0.4,1.75) to (0,1.05);
\fill (-1.6,1.05) circle(.06);
\fill (-1.2,.35) circle(.06);
\fill (-.8,1.05) circle(.06);
\fill[red] (-.4,.35) circle(.06);
\fill[red] (0,1.05) circle(.06);
\fill[red] (.4,.35) circle(.06);
\fill (.8,1.05) circle(.06);
\fill (1.2,.35) circle(.06);
\fill (1.6,1.05) circle(.06);
\fill (-0.4,1.75) circle(.06);
\fill (0,2.45) circle(.06);
\fill (0.4,1.75) circle(.06);
\end{tikzpicture}
\caption{In the right-angled Artin group given by this graph, the standard 2--flats from pairs of red vertices are directionally branch-complemented but not branching. The red 3--flats are branching and directionally branch-complemented, but do not witness the fact that their standard geodesics are branch-complemented.}
\label{fig:dbc_not_branching}
\end{figure}

As with branch-complemented geodesics, we will show in \Cref{lem:ultralimit_fully_branching_flat} that ultralimits of directionally branch-complemented $k$--flats are directionally $0$--branch-complemented.

\begin{example} \label{ex:fully_branching}
Here are some examples of branch-complemented geodesics and directionally branch-complemented flats. A simplicial tree is said to be \emph{bushy} if every vertex has three unbounded complementary components.
\begin{itemize}
\item   In a product of bushy trees, every singular $k$--flat in the $k$--skeleton is branching. Hence every singular geodesic is branch-complemented and every singular flat is directionally branch-complemented. 
\item   In right-angled Artin groups, a singular geodesic can be branching without being a coset of a cyclic subgroup (e.g.\ in $F_2 \times F_2$) and vice versa (e.g.\ in $\mathbb{Z}^2$).
\item   In a $2$--dimensional right-angled Artin group $A_\Gamma$, a coset of a cyclic subgroup corresponding to a standard generator is a branching geodesic if and only if the corresponding vertex in $\Gamma$ has degree at least two.
\item   If $F$ is directionally branch-complemented in $X$ and $T$ is a bushy tree, then every flat of the form $F\times \gamma$, with $\gamma\subseteq T$ a singular geodesic, is directionally branch-complemented in $X\times T$.
\end{itemize}
\end{example}

Note that the conclusion of the following can fail for more general CAT(0) spaces, because $H$ may fail to ``line up'' with any of the subflats of $F$ spanned by geodesics witnessing that $F$ is directionally branch-complemented.

\begin{lemma} \label{lem:subflats_of_fully_branching_ccc}
Let $X$ be a finite-dimensional CAT(0) cube complex of asymptotic rank $n$, and let $F\subseteq X$ be a directionally branch-complemented flat. If $H\subseteq F$ is a branching flat, then $H$ is directionally branch-complemented.
\end{lemma}

\begin{proof}
By \cref{lem:branching_singular}, every branching geodesic in $X$ is semisingular, and $H$ is also semisingular. Moreover, \cref{lem:dbc_semisingular} shows that $F$ is semisingular. Let $m=\dim F$ and $k=\dim H$.

Let $\gamma_1,\dots,\gamma_m\subset F$ be geodesics as in \cref{def:fully_branching_flats}, demonstrating that $F$ is directionally branch-complemented. Let $\hat X$ be an asymptotic cone of $X$ with respect to a fixed basepoint. Let $\hat H$ and $\hat F$ be the ultralimits of $H$ and $F$, respectively. Since $H$ and $F$ are semisingular, after relabelling the $\gamma_i$ we have that $\hat H$ is equal to the span of the ultralimits $\hat\gamma_1,\dots,\hat\gamma_k$. Since $H$ and $F$ are flats and the $\gamma_i$ are geodesics in $F$, this implies that $\gamma_1,\dots,\gamma_k$ have parallels $\gamma'_1,\dots,\gamma'_k$ that lie in $H$.

Since $\gamma_1,\dots,\gamma_m$ satisfy \cref{def:fully_branching_flats} for $F$, the geodesics $\gamma'_1,\dots,\gamma'_k$ satisfy \cref{def:fully_branching_flats} for $H$.
\end{proof}

We will show in \cref{thm:fully_branching_general} that quasiisometric embeddings map directionally branch-complemented top-dimensional flats uniformly Hausdorff-close to singular flats. From this, we can deduce a similar result for intersections of such flats. We will therefore be interested in geodesics that are intersections of branch-complemented top-dimensional flats, and in the flats that such geodesics span.

\begin{definition}[Strong branching]\label{def:strong_fully_branching}
Let $X$ be a CAT(0) space of asymptotic rank $n\geq 1$. A geodesic is \emph{strongly $r$--branch-complemented} if it is at Hausdorff distance at most $r$ from an intersection of directionally $r$--branch-complemented $n$--flats. It is \emph{strongly branch-complemented} if it is strongly $r$--branch-complemented for some $r$.

A flat is \emph{directionally strongly $r$--branch-complemented} if it is directionally $r$--branch-complemented and the $r$--branch-complemented geodesics witnessing that fact are themselves strongly $r$--branch-complemented. It is \emph{directionally strongly branch-complemented} if it is directionally strongly $r$--branch-complemented for some $r$.
\end{definition}

\begin{figure}[ht]
\centering\begin{tikzpicture}[baseline=-.1cm, scale=1.3]
\draw[thick]
(-1.6,1.05) -- (-0.8,1.05)
(0.8,1.05) -- (1.6,1.05)
(-1.2,.35) -- (1.2,.35);
\draw[thick]
(-1.6,1.05) to (-1.2,.35) to (-.8,1.05) to (-.4,.35) to (0,1.05) to
(.4,.35) to (.8,1.05) to (1.2,.35) to (1.6,1.05);
\fill (-1.6,1.05) circle(.06);
\fill (-1.2,.35) circle(.06);
\fill (-.8,1.05) circle(.06);
\fill[red] (-.4,.35) circle(.06);
\fill (0,1.05) circle(.06);
\fill[red] (.4,.35) circle(.06);
\fill (.8,1.05) circle(.06);
\fill (1.2,.35) circle(.06);
\fill (1.6,1.05) circle(.06);
\end{tikzpicture}
\caption{In the right-angled Artin group given by this graph, the standard geodesics corresponding to red vertices are branch-complemented but not strongly branch-complemented.}
\label{fig:dbc_not_strongly_branching}
\end{figure}
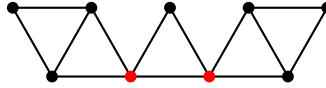

\begin{remark} \label{rem:not_nec_branching}
It is important to note that a directionally (strongly) branch-complemented flat of dimension greater than $1$ need not be at finite Hausdorff distance from any branching flat, nor even be contained in a top-dimensional flat; see for example the standard flats associated to the edges that are not contained in any 3--clique in Figure~\ref{fig:fully_branching_subgraphs_3D}. 
\end{remark}

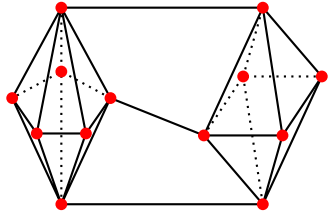
\begin{figure}[ht]
\centering
\begin{tikzpicture}[baseline=-.1cm, scale=1.3]

\coordinate (tL) at (-.85,1);
\coordinate (bL) at (-.85,-1);

\coordinate (v1) at (-1.35,.08);
\coordinate (v2) at (-1.10,-.28);
\coordinate (v3) at (-.60,-.28);
\coordinate (v4) at (-.35,.08);
\coordinate (v5) at (-.85,.35);

\coordinate (a1) at (.6,-.3);
\coordinate (a2) at (1.4,-.3);
\coordinate (a3) at (1.8,.3);
\coordinate (a4) at (1,.3);
\coordinate (tR) at (1.2,1);
\coordinate (bR) at (1.2,-1);

\draw[thick]
(v1) -- (v2) -- (v3) -- (v4);
\draw[thick,dotted]
(v4) -- (v5) -- (v1);

\draw[thick]
(tL) -- (v1) (tL) -- (v2) (tL) -- (v3) (tL) -- (v4)
(bL) -- (v1) (bL) -- (v2) (bL) -- (v3) (bL) -- (v4);
\draw[thick,dotted]
(tL) -- (v5)
(bL) -- (v5);

\draw[thick]
(a1) -- (a2) -- (a3);
\draw[thick,dotted]
(a1) -- (a4) -- (a3);

\draw[thick]
(a1) -- (tR)
(a2) -- (tR)
(a3) -- (tR)
(a1) -- (bR)
(a2) -- (bR)
(a3) -- (bR);
\draw[thick,dotted]
(a4) -- (tR)
(a4) -- (bR);

\draw[thick]
(tL) -- (tR)
(bL) -- (bR)

(v4) -- (a1);

\fill[red] (v1) circle(.06);
\fill[red] (v2) circle(.06);
\fill[red] (v3) circle(.06);
\fill[red] (v4) circle(.06);
\fill[red] (v5) circle(.06);
\fill[red] (tL) circle(.06);
\fill[red] (bL) circle(.06);

\fill[red] (a1) circle(.06);
\fill[red] (a2) circle(.06);
\fill[red] (a3) circle(.06);
\fill[red] (a4) circle(.06);
\fill[red] (tR) circle(.06);
\fill[red] (bR) circle(.06);
\end{tikzpicture}
\caption{In the right-angled Artin group associated with this graph, all standard geodesics are strongly branch-complemented, hence all standard flats are directionally strongly branch-complemented. Nonetheless, the standard $2$-flats associated to the three middle edges are not contained in any 3--flat.}
    \label{fig:fully_branching_subgraphs_3D}
\end{figure}

\begin{example}
As in \cref{ex:fully_branching}, if $X$ is a symmetric space or thick Euclidean building, then every singular flat in $X$ is directionally strongly branch-complemented. 
\end{example}


\subsection{Branching properties and ultralimits} \label{subsec:ultralimit_of_fully_branching_geod}

In this subsection we show that directionally $r$--branch-complemented flats give directionally $0$--branch-complemented flats in the asymptotic cone. 

We start with a general result showing that intersections of flats are close to intersections of finitely many flats, which lets us show in \cref{lem:uniform_branching} that branching is preserved under ultralimits. We then prove the 1--dimensional case, in \cref{lem:ultralimit_fully_branching_geod}, and from that we derive the general case in \cref{lem:ultralimit_fully_branching_flat}.

\begin{lemma} \label{lem:almost_intersection_finitely_many}
For every $m$ there exists an integer $M=M(m)$ such that the following holds for all $\eps>0$ and all $k\in\mathbb N$. 

Let $X$ be a complete CAT(0) space, and let $H$ be a $k$--flat that is the intersection of a family $\cal F=\{F_i\}$ of $(k+m)$--flats. There is a subset $\cal F'\subset\cal F$ of cardinality at most $M$ such that the intersection of the elements of $\cal F'$ lies at Hausdorff distance at most $\eps$ from $H$.
\end{lemma}

\begin{proof}
Fix $F_0\in\cal F$. For each $F_i$ we have that $C_i=F_i\cap F_0$ is a closed convex subspace of the $(m+k)$--flat $F_0$ that contains $H$. If $x\in C_i$, then since $C_i$ is convex, it contains the geodesic $[x,y]$ for every $y\in H$. Hence $C_i$ contains the parallel of $H$ through $x$, because it is closed. Thus $C_i$ splits metrically as a product $C_i=H\times B_i$, where $B_i$ is a closed convex subset of $\R^{m}$ containing the origin. 

From the assumptions of the lemma, the intersection of all the $B_i$ is exactly the origin of $\R^{m}$. The lemma therefore reduces to showing that a uniformly finite subset of the $B_i$ have intersection contained in the $\eps$--ball of $\R^{m}$. 

Let $I$ be the set of all $i$ such that $B_i$ does not contain the sphere $S=S_{\R^{m}}(0,\frac\eps2)$.
Note that for every $x\in S$ there is some $B_i$ that does not contain it, with $i\in I$.

If $x\in S$ lies outside $B_i$, then because $B_i$ is convex it is contained in an affine halfspace whose boundary contains $x$. Thus $S'\ssm B_i$ contains a spherical cap of polar angle at least $\frac\pi3$, where $S'=S_{\R^{m}}(0,\eps)$ is the $\eps$--sphere in $\R^{m}$. By letting $x$ vary over $S$, we obtain a covering of $S'$ by spherical caps of polar angle at least $\frac\pi3$. We can therefore pass to a subcover of cardinality depending only on $m$. The corresponding elements of $I$, together with $F_0$, give us the desired collection $\cal F'$.
\end{proof}

\begin{corollary} \label{lem:uniform_branching}
Let $X$ be a complete CAT(0) space of asymptotic rank $n$, and let $r\ge0$. Suppose that $(H_i)$ is a sequence of $k$--flats such that each $H_i$ lies at Hausdorff distance at most $r$ from a branching $k$--flat $H'_i$. For every asymptotic cone $X_\omega$ of $X$, the ultralimit $(H_i)_\omega$ is branching.
\end{corollary}

\begin{proof}
We have $(H_i)_\omega=(H'_i)_\omega$. By \cref{lem:almost_intersection_finitely_many}, for each $i$ there exist at most $M=M(n-k)$ flats $F_{i,j}$ of dimension $n$ such that $H'_i$ lies at Hausdorff distance at most one from $I_i=\bigcap_jF_{i,j}$. It follows that $(H_i)_\omega=(I_i)_\omega$ is equal to the intersection of the $n$--flats $\big((F_{i,j})_i\big)_\omega$, of which there are at most $M$.
\end{proof}

\begin{lemma}\label{lem:ultralimit_fully_branching_geod}
Let $X$ be a complete CAT(0) space of asymptotic rank $n$. Let $(\gamma_m)$ be a sequence of geodesics in $X$, and let $\hat X$ be an asymptotic cone of $X$ such that the ultralimit $\hat\gamma=\lim_\omega(\gamma_m)$ exists.
If there exists $r\ge0$ such that every $\gamma_m$ is $r$--branch-complemented, then $\hat\gamma\subseteq \hat X$ is a $0$--branch-complemented geodesic.
\end{lemma}

\begin{proof}
For each $m$, consider a pair $(F_m,H_m)$ as in \Cref{def:fully_branching_CAT(0)}. The ultralimit $\hat F=\lim_\omega F_m$ is an $n$--flat containing $\hat\gamma$, and $\hat H=\lim_\omega H_m$ is an $(n-1)$--flat in $\hat F$ transverse to $\hat\gamma$. Given $\hat x\in \hat F$, let $(x_m)\subset F_m$ be a sequence that represents $\hat x$. By the choice of $(F_m,H_m)$, there is a branching $(n-1)$--flat $K_m\subseteq F_m$ parallel to $H_m$ and a branching geodesic $\ell_m\subseteq F_m$ parallel to $\gamma_m$ with $d(x_m,K_m)\le r$ and $d(x_m,\ell_m)\le r$. Let $\hat K=\lim_\omega(K_m)$ and $\hat\ell=\lim_\omega(\ell_m)$. The former is a subflat of $\hat F$ parallel to $\hat H$, and the latter is a geodesic in $\hat F$ parallel to $\hat\gamma$. Moreover, $\hat x\in\hat K\cap\hat\ell$. Finally, by \cref{lem:uniform_branching}, both $\hat K$ and $\hat\ell$ are branching.
\end{proof}

We now turn to the $k$--dimensional case.

\begin{proposition}\label{lem:ultralimit_fully_branching_flat}
Let $X$ be a complete CAT(0) space of asymptotic rank $n$. Let $(F_m)$ be a sequence of $k$--flats in $X$, and let $\hat X$ be an asymptotic cone of $X$ such that the ultralimit $\hat F=\lim_\omega(F_m)$ exists.
If there exists $r\ge0$ such that every $F_m$ is directionally $r$--branch-complemented, then $\hat F \subseteq \hat X$ is a directionally $0$--branch-complemented $k$--flat.
\end{proposition}

\begin{proof}
Since $F_m$ is directionally $r$--branch-complemented, it is spanned by geodesics $\gamma_{m,1},\dots,\gamma_{m,k}\subseteq F_m$ as in \Cref{def:fully_branching_flats}. The properties of the $\gamma_{m,i}$ imply that, after replacing each by a parallel, we can assume that the ultralimit $\hat\gamma_i=\lim_\omega(\gamma_{m,i})$ exists for all $i$.

Fix $i$ and let $\hat\beta\subseteq \hat F$ be any geodesic parallel to $\hat\gamma_i$. Represent $\hat\beta$ by a sequence of geodesics $\beta_m\subseteq F_m$ parallel to $\gamma_{m,i}$. Because $F_m$ is directionally $r$--branch-complemented, each $\beta_m$ lies at Hausdorff distance at most $r$ from some $r$--branch-complemented geodesic $\delta_m\subseteq F_m$. Hence $\hat\delta=\lim_\omega \delta_m$ also exists, and by Lemma~\ref{lem:ultralimit_fully_branching_geod} it is a $0$--branch-complemented geodesic. The bound $d_{\mathrm{Haus}}(\beta_m,\delta_m)\le r$ implies that $\hat\beta=\hat\delta$, and $\hat \beta$ is therefore $0$--branch-complemented. This holds for every $\hat\beta$ parallel to $\hat\gamma_i$, for every $i$, and hence $\hat F$ is directionally $0$--branch-complemented.
\end{proof}

\subsection{BiLipschitz embeddings} \label{subsec:applications_to_CCC}

In this section we show that every biLipschitz embedding between asymptotic cones of CAT(0) cube complexes sends every directionally $0$--branching-complemented $k$--flat to a $k$--flat. 

We will require some tools from geometric measure theory.
We denote the $n$--dimensional Hausdorff measure by $\cal H^n$. The following definition appears as \cite[Def.~3.2.14]{federer:geometric}.

\begin{definition}[Countably rectifiable] \label{def:rectifiable}
Let $X$ be a metric space. A subset $E\subseteq X$ is called \emph{countably $\mathcal{H}^{n}$--rectifiable} if there is a countable family of Lipschitz maps $f_i\colon  A_i \to X$, where $A_i \subseteq \mathbb{R}^n$ is $\mathcal{L}^n$--measurable, such that $\mathcal{H}^n(E \ssm \bigcup_i f_i(A_i)) =0$. 
\end{definition}

Note that a countably $\mathcal{H}^{0}$--rectifiable set is just a countable set.

\begin{theorem}[{\cite[Thm~3.2.22]{federer:geometric}}] \label{thm:rectifiable-level-sets}
Suppose that $W \subseteq \mathbb{R}^d$ is $\mathcal{H}^{n}$--measurable and countably $\mathcal{H}^{n}$--rectifiable. If $f\colon  W \to \mathbb{R}^m$ is Lipschitz, with $m\leq n$, then for $\mathcal{H}^{m}$--almost-every $y \in \mathbb{R}^m$, the fibre $f^{-1}(y)$ is $\mathcal{H}^{n-m}$--measurable and countably $\mathcal{H}^{n-m}$--rectifiable.
\end{theorem}


We shall only make use of this theorem in the following form.

\begin{corollary}\label{lem:rectifiable_parallel_flats}
Let $n\ge2$, let $k\in\{1,\dots,n\}$, and let $S\subseteq \mathbb R^n$ be an $\mathcal H^{k}$--measurable and countably $\mathcal H^{k}$--rectifiable subset. Let $H$ and $F$ be $p$-- and $q$--dimensional affine subspaces of $\mathbb R^n$, respectively, with $H\subsetneq F$.

If $q-p\leq k$, then for almost every $p$--dimensional affine subspace $H'\subseteq F$ parallel to $H$, the intersection $H'\cap S$ is $\mathcal H^{k-(q-p)}$--measurable and countably $\mathcal H^{k-(q-p)}$--rectifiable. If $q-p>k$, then for almost every $p$-dimensional affine subspace $H'\subseteq F$ parallel to $H$, the intersection $H'\cap S$ is empty.
\end{corollary}

\begin{proof}
First assume that $q-p\leq k$. Let $\varphi\colon  F \to \mathbb R^{q-p}$ be the projection onto a $(q-p)$--dimensional affine subspace of $F$ that is orthogonal to $H$.
Then $\varphi$ is Lipschitz, and its fibres are precisely the $p$--dimensional affine subspaces of $F$ that are parallel to $H$.
Since $S$ is $\mathcal H^{k}$--measurable and countably $\mathcal H^{k}$--rectifiable, it follows from \Cref{thm:rectifiable-level-sets} applied to $\varphi|_S$, that the intersection $H'\cap S$ is $\mathcal H^{k-(q-p)}$--measurable and countably $\mathcal H^{k-(q-p)}$--rectifiable, for almost every $p$--dimensional affine subspace $H'\subseteq F$ parallel to $H$.

If instead $(q-p)>k$, then, in particular, $q>k$. Let $P$ be a $(q-k)$-dimensional affine subspace such that  $H\subseteq P \subseteq F$. Since $q-(q-k)\leq k$, we get from the previous case that $P'\cap S$ is countably $\mathcal{H}^{0}$-rectifiable, hence countable, for almost every $P'$ which is parallel to $P$ in $F$.
Therefore, all but countably many parallels of $H$ inside such a parallel do not intersect $S$. As almost every parallel of $H$ in $F$ is inside such a $P'$ and almost all of its parallel inside $P'$ do not intersect $S$, we get that almost every parallel of $H$ inside $F$ does not intersect $S$.
\end{proof}

We shall also use the corollary together with the following result of Bowditch to prove \cref{prop:fully_branching_implies_geod_in_cone}. Specifically, to show that in the setting of the proposition, almost every parallel of a geodesic inside a flat does not intersect the singular set of the map $f$.

\begin{proposition}[{\cite[Prop.~4.3, Lem.~3.5]{bowditch:large:mapping}}] \label{prop:Bowditch_bilipflat_cubulated}
Let $M$ be a complete median metric space of rank $n$. If $f\colon  \mathbb R^n \to M$ is a biLipschitz embedding, then $f(\mathbb R^n)$ is cubulated. Moreover, there is a cubulated subset $L \subseteq f(\mathbb R^n)$ of dimension at most $n-2$ such that the singularity set of $f$ is $f^{-1}(L)$.
\end{proposition}

We now return to the setting of complete CAT(0) spaces.

\begin{proposition}\label{prop:fully_branching_implies_geod_in_cone}
Let $X$ be a complete CAT(0) space, and let $Y$ be a finite-dimensional CAT(0) cube complex. Assume that $X$ and $Y$ both have asymptotic rank $n$. Let $\hat X$ and $\hat Y$ be asymptotic cones of $X$ and $Y$, respectively, and let $f\colon  \hat X\to \hat Y$ be a biLipschitz embedding.  

If $\hat F\subseteq \hat X$ is a directionally $0$--branch-complemented $k$--flat, then $f(\hat F)$ is a singular $k$--flat.
\end{proposition}


\begin{proof}
The proposition will follow from the case $k=1$ using \cref{prop:singular_quasi_flats} for biLipschitz maps where $D=0$, by the definition of directionally $0$-branch-complemented flats. Thus, it is enough to show that if $\hat \gamma_0\subseteq \hat X$ is a 0--branch-complemented geodesic, then $f(\hat\gamma_0)$ is a singular geodesic.

By \cref{lem:rank_cone_CCC}, the asymptotic cone $\hat Y$ is a median algebra of rank at most $n$. Since $\hat X$ admits a biLipschitz embedding in $\hat Y$, the rank of $\hat Y$ must be equal to $n$. Therefore, by \Cref{prop:Bowditch_bilipflat_cubulated}, every biLipschitz $n$--flat in $\hat Y$ is cubulated. Hence, the image under $f$ of every branching flat is cubulated, as an intersection of cubulated subsets.

Let $\hat \gamma_0 \subseteq \hat X$ be a $0$--branch-complemented geodesic. By definition, there exists an $n$--flat $\hat F$ and a branching $(n-1)$--flat $\hat H_0 \subseteq \hat F$ that is transverse to $\hat \gamma_0$ and such that all parallels of $\hat \gamma_0$ and $\hat H_0$ inside $\hat F$ are branching. 

Let $S \subseteq \hat F$ be the singularity set of $f|_{\hat F}$: if $x\in \hat F \setminus S$, then $x$ is flat for $f|_{\hat F}$ (see \Cref{def:regular_singular_sets_mediancase}). By \Cref{prop:Bowditch_bilipflat_cubulated}, $S$ is $\mathcal{H}^{n-2}$--measurable and countably $\mathcal{H}^{n-2}$--rectifiable. Consequently, by \cref{lem:rectifiable_parallel_flats}, almost every parallel of $\hat\gamma_0$ in $\hat F$ does not intersect $S$. By continuity of $f$, it therefore suffices to show that if $\hat \gamma \subseteq \hat F$ is a parallel of $ \hat \gamma_0$ that does not intersect $S$, then $f (\hat \gamma)$ is a geodesic. 
Let $\hat \gamma \subseteq \hat F$ be a parallel of $ \hat \gamma_0$ that does not intersect $S$,
and let $x \in \hat \gamma$. It is flat for $f|_{\hat F}$, because $\hat \gamma \cap S = \varnothing$. Thus, there exists $r>0$ such that $f (B_{\hat F}(x,r))$ is a contained in a single real $n$--cube of $\hat Y$.
Let $\hat H$ be the parallel of $\hat H_0$ that contains $x$.
Since both $\hat \gamma$ and $\hat H$ are branching, their images $f(\hat \gamma)$ and $f(\hat H)$ are cubulated. In particular, there exists $U$, an open neighbourhood of $x$ in $\hat F$, such that $U \subseteq B_{\hat F}(x,r)$ and both $f(\hat \gamma \cap U)$ and $f(\hat H \cap U)$ are cubulated. 

Since $f(U)$ is contained in a single $n$--cube, one can identify $f$ with a local topological embedding of $\mathbb R^n$ equipped with the median $\ell^1$ metric into itself. It follows from \Cref{prop:top_argument} that there exists a neighbourhood $V \subseteq U$ of $x$ in $\hat F$ such that $f (V \cap \hat\gamma)$ is contained in a single $1$--cube of $\hat Y$. 

Since $x \in \hat \gamma$ was arbitrary, this shows that $f(\hat \gamma)$ is a local singular geodesic in $\hat Y$. Since $\hat Y$ is CAT(0), we conclude that $f(\hat \gamma)$ is a singular geodesic.
%
%
\end{proof}

\begin{corollary} \label{cor:fully_branching_implies_flat_in_cone}
Let $X$ be a complete CAT(0) space, and let $Y$ be a finite-dimensional CAT(0) cube complex. Assume that $X$ and $Y$ both have asymptotic rank $n$. Let $\hat X$ and $\hat Y$ be asymptotic cones of $X$ and $Y$, respectively, and let $f\colon  \hat X\to \hat Y$ be a biLipschitz embedding.

If $\hat F\subseteq \hat X$ is the ultralimit of a sequence of directionally $r$--branch-complemented $k$--flats in~$X$, for some $r \geq 0$, then $f(\hat F)$ is a singular $k$--flat in $\hat Y$.
\end{corollary}

\begin{proof}
By \Cref{lem:ultralimit_fully_branching_flat}, $\hat F$ is directionally 0--branch-complemented, so the statement follows from \cref{prop:fully_branching_implies_geod_in_cone}.
\end{proof}

\section{The branching theorems} \label{sec:fully_branching_theorems}

Here we prove our main general results stated in the introduction about quasiisometric embeddings of directionally (strongly) branch-complemented flats into CAT(0) cube complexes.


Namely, in \cref{thm:fully_branching_general} we show that quasiisometric embeddings map each directionally branch-complemented flat $F$ of top rank Hausdorff-close to a flat. We then split the discussion according to whether the cube complexes have the same dimension as $F$. \cref{subsec:main_dim} handles the $n$--dimensional case, where the arguments need less machinery than in general. The results in \cref{subsec:main_general} are more general, but the conclusions are slightly weaker. In either case, we show that certain subflats of $F$ are mapped at finite, but uncontrolled, distance from (semi)singular flats (Theorems~\ref{thm:subflat_branching_in_fully_branching} and~\ref{thm:subflat_branching_asymptotic_rank}). We can also handle lower-rank flats that do not lie in any $n$--flat provided that they are directionally \emph{strongly} branch-complemented (Theorems~\ref{thm:strong_fully_branching} and~\ref{thm:strong_branching_asymptotic_rank}), provided that the domain is $n$--dimensional. When the codomain is also $n$--dimensional, we even obtain controlled distance bounds. We conclude by interpreting these results on the level of boundaries in \cref{subsec:boundaries}.

\begin{theorem}\label{thm:fully_branching_general}
Let $X$ be a complete CAT(0) space and let $Y$ be finite-dimensional CAT(0) cube complex. Assume that $X$ and $Y$ both have asymptotic rank~$n$. For each $q\geq1$ there exists $D=D(q,Y)$ such that the following holds for every $q$--quasiisometric embedding $f\colon X \to Y$. 

For each directionally branch-complemented $n$--flat $F \subseteq X$, the image $f(F)$ lies within Hausdorff distance at most $D$ of some $n$--flat $F' \subseteq Y$. 

Moreover, if $\dim Y=n$, then $D=D(q)$.
\end{theorem}

\begin{proof}
Let $F\subseteq X$ be a directionally branch-complemented $n$--flat. Let $\hat f\colon \hat X\to\hat Y$ be a biLipschitz map between asymptotic cones that is induced by $f$, where the asymptotic cones are taken with fixed basepoints. Let $\hat F\subset\hat X$ be the ultralimit of $F$. By \cref{cor:fully_branching_implies_flat_in_cone}, its image $\hat f(\hat F)\subset\hat Y$ is an $n$--flat. The existence of $F'$ follows from \cref{prop:quasiflats-from-bilip-flats} in the case $Y$ is $n$--dimensional, or from \cref{prop:quasiflat_from_biLipschitz:asymptotic_rank} in the general case.
\end{proof}

\subsection{Dimension and asymptotic rank agree} \label{subsec:main_dim}

Here we consider quasiisometric embeddings between $n$--dimensional CAT(0) cube complexes. The following adds to \cref{thm:fully_branching_general} by giving more information about subflats. 

\begin{theorem} \label{thm:subflat_branching_in_fully_branching}
Let $X$ and $Y$ be $n$--dimensional CAT(0) cube complexes, and let $f\colon X\to Y$ be a quasiisometric embedding. If $F\subseteq X$ is a directionally branch-complemented $n$--flat and $H \subseteq F$ is a branching flat, then $f(H)$ lies within finite Hausdorff distance of a singular flat of $Y$. 

In particular, if $\gamma\subseteq F$ is a singular geodesic, then $f(\gamma)$ lies within finite Hausdorff-distance of a singular geodesic of $Y$.
\end{theorem}

\begin{proof}
Let $H \subseteq F$ be a directionally branch-complemented subflat of dimension $k$. By definition, $H$ is the intersection of finitely many $n$--flats $F_1,\dots,F_m$. According to \cite[Thm~1.1]{huang:top}, for each $i$ the image $f(F_i)$ is at finite Hausdorff distance from a finite union of orthants. By \cref{lem:top_dim_singular}, each $F_i$ is singular, hence convex. It therefore follows from \Cref{prop:intersection_orthants} that $f(H)$ is at finite Hausdorff distance from a finite union of singular orthants  $\bigcup_{j=1}^t O_j$. Recall from \cref{prop:boundary_quasi_flat_sphere} that we denote by $\partial_Tf(H)$ the union of $\partial_T O_j$. \cref{prop:boundary_quasi_flat_sphere} tells us that $\partial_T f(H)$ is homeomorphic to $\mathbb S^{k-1}$. 
    
Let $\hat f\colon \hat X\to \hat Y$ be an induced biLipschitz map of asymptotic cones, taken with respect to fixed basepoints $x$ and $y=f(x)$, respectively. Let $\hat H\subseteq \hat X$ be the ultralimit of $H$. Note that $\hat f(\hat H)\subset\hat Y$ is equal to the ultralimit of $f(H)$.

On the one hand, $H$ is directionally branch-complemented, by \cref{lem:subflats_of_fully_branching_ccc}, so \Cref{cor:fully_branching_implies_flat_in_cone} tells us that $\hat f(\hat H)\subseteq \hat Y$ is a singular $k$--flat. In particular, $\Sigma_{(y)}\hat f(\hat H)$ is a round $(k-1)$--sphere. On the other hand, $\partial_T f(H)$ is isometric to the subset $\Sigma_{(y)}\hat f(\hat H)$ of the link $\Sigma_{(y)} \hat Y$, by \Cref{lem:embedding-Titsboundary-into_link_and_Tits_boundary}. Hence $\partial_Tf(H)$ is a round $(k-1)$--sphere.
    
Let $F'\subseteq Y$ be an $n$--flat at finite Hausdorff distance from $f(F)$, as given by \cref{thm:fully_branching_general}. It is singular by \cref{lem:top_dim_singular}. We have that
\[
\partial_T f(H) \,\subset\, \partial_T f(F) \,=\, \partial_T F' \,=\, \mathbb{S}^{n-1}
\]
is a subcomplex that is a round $(k-1)$--sphere. Therefore, there exists a singular $k$--flat $H' \subseteq F'$ such that $\partial_T H'=\partial_T f(H)$. Since $H'$ and $\bigcup_{j=1}^t O_j$ have the same compact Tits boundary, convexity of the distance function implies that they are at finite Hausdorff distance from one another. We conclude that $f(H)$ and $H'$ are at finite Hausdorff distance.

The second statement holds because every singular geodesic in $F$ lies at finite Hausdorff-distance from a branch-complemented geodesic, by the assumption that $F$ is directionally branch-complemented.
\end{proof}


\cref{thm:subflat_branching_in_fully_branching} does not give a uniform bound on Hausdorff distances. If we wish to obtain a version with uniform bounds, then by \cref{prop:singular_quasi_flats} it is enough to do so for spanning geodesics. This suggests using \cref{thm:fully_branching_general} and \cref{prop:intersection_semisingular}, which would give us the result for intersections of directionally branch-complemented flats. Having a spanning set of geodesics this way amounts to assuming that the flat is directionally strongly branch-complemented in the sense of \cref{def:strong_fully_branching}. 

\begin{theorem}\label{thm:strong_fully_branching}
Let $X$ and $Y$ be $n$--dimensional CAT(0) cube complexes. For every $q\geq 1$ and $r\geq 0$, there exists a constant $D=D(n,q,r)$ such that the following holds for every $q$--quasiisometric embedding $f\colon X\to Y$. 

If $H\subseteq X$ is a directionally strongly $r$--branch-complemented $k$--flat, for some $k\le n$, then $f(H)$ lies at Hausdorff distance at most $D$ from a singular $k$--flat in $Y$.
\end{theorem}

\begin{proof}
We first consider the case $k=1$. Let $H$ be a strongly $r$--branch-complemented geodesic.
By definition, $H$ lies at Hausdorff distance at most $r$ from the intersection $\gamma$ of finitely many directionally $r$--branch-complemented $n$--flats $F_1,\dots, F_s$. By \cref{lem:top_dim_singular}, each $F_i$ is singular.

\cref{thm:fully_branching_general} states that each $f(F_i)$ lies at Hausdorff distance at most $D'$ from an $n$--flat of $Y$, where $D'=D'(q)$, and that $n$--flat is singular by \cref{lem:top_dim_singular}. \cref{prop:intersection_semisingular} thus provides a constant $D''=D''(D',q)$ such that $f(\gamma)$ lies at Hausdorff distance at most $D''$ from a singular geodesic of $Y$. Hence $f(H)$ lies at Hausdorff-distance at most $D''+qr+q$ from a singular geodesic of~$Y$. 

Now let $H$ be a directionally strongly $r$--branch-complemented $k$--flat, for some $k>1$. It is singular by \cref{lem:dbc_semisingular}. By definition, it is spanned by geodesics $\gamma_1,\dots,\gamma_k$ such that, for each $i$, there are strongly $r$--branch-complemented parallels in $H$ of $\gamma_i$ that together $r$--coarsely cover $H$. 

Let $\gamma'_i$ be one such parallel. From the above arguments, we know that $f(\gamma'_i)$ lies at Hausdorff distance at most $D''+qr+q$ from a singular geodesic of $Y$. Since every parallel in $H$ of $\gamma_i$ lies at Hausdorff distance at most $r$ from some such $\gamma'_i$, this shows that the $f$--image of each parallel in $H$ of $\gamma_i$ lies at Hausdorff distance at most $D''+2qr+q$ from a singular geodesic of $Y$. Thus \cref{prop:singular_quasi_flats} applies, and it gives a constant $D=D(q,k,r,D'')$ such that $f(H)$ lies at Hausdorff distance at most $D$ from a singular $k$--flat in $Y$.
%
\end{proof}


    

    





We now describe conditions under which we can additionally control the images of orthants under quasiisometric embeddings. This will enable us to induce maps of certain boundaries.

\begin{corollary}\label{cor:embedding_orthant_strong_fullybranching}
Let $X$ and $Y$ be $n$--dimensional CAT(0) cube complexes. If $f\colon X\to Y$ is a quasiisometric embedding, then for each $r\ge0$ there exists $D\ge0$ such that $f$ maps every singular orthant in every directionally strongly $r$--branch-complemented flat within Hausdorff distance at most $D$ of a singular orthant of $Y$.
\end{corollary}

\begin{proof}
Let $H\subseteq X$ be a directionally strongly $r$--branch-complemented flat. Its singular geodesics are strongly $r$--branch-complemented. Hence, by \cref{thm:strong_fully_branching}, there exists $D'\ge0$ such that every singular geodesic in $H$ is mapped within Hausdorff distance $D'$ of some singular geodesic of $Y$. We can thus apply \cref{lem:orthants_to_semisingular_orthants}, which provides a constant $D$ such that every singular orthant of $H$ is mapped within Hausdorff distance $D$ of a singular orthant of~$Y$.
\end{proof}



If one does not assume that $H$ is directionally strongly branch-complemented, then one can no longer conclude that its orthants are mapped within finite Hausdorff distance of orthants. 
Nevertheless, one can still show that the axis rays of certain singular $2$--orthants are mapped Hausdorff-close to the axis rays of singular $2$--orthants. 
Note that the assumption on $H$ in the following is weaker than requiring it to be directionally strongly branch-complemented. Indeed, a branch-complemented geodesic that belongs to a directionally branch-complemented $n$--flat need not be strongly branch-complemented.

\begin{theorem}\label{thm:fully_branching_2_flat} 
Let $X$ and $Y$ be $n$--dimensional CAT(0) cube complexes, and let $f\colon X \to Y$ be a quasiisometric embedding. Let $H \subseteq X$ be a directionally branch-complemented $2$--flat such that each branch-complemented geodesic of $H$ belongs to some directionally branch-complemented $n$--flat. 

If $Q\subseteq H$ is a singular 2--orthant, then there exists a singular 2--orthant $Q'\subseteq Y$ such that the two axis rays of $Q$ are mapped by $f$ within finite Hausdorff distance of the two axis rays of $Q'$.
\end{theorem}

\begin{proof}
The second statement of \cref{thm:subflat_branching_in_fully_branching} tells us that every singular geodesic of $H$ is mapped by $f$ within finite Hausdorff distance of a singular geodesic of $Y$. Moreover, \Cref{cor:fully_branching_implies_flat_in_cone} implies that the ultralimit $f(H)_\omega$ of $f(H)$ inside an asymptotic cone of $Y$, taken with respect to a fixed basepoint, is a $2$--flat, because $H$ is directionally branch-complemented. The result therefore follows from \Cref{prop:2_orthant_in_target}.
\end{proof}
%
%

\subsection{The general case} \label{subsec:main_general}

Here we prove analogues of \cref{thm:subflat_branching_in_fully_branching}, \cref{thm:strong_fully_branching}, and \cref{cor:embedding_orthant_strong_fullybranching} without requiring the dimension to match the asymptotic rank. The cost of this greater generality is that semisingularity replaces singularity in the conclusions, and there is also less uniformity in the constants. These are both necessary, as can be seen from \cref{eg:corner_to_corner}.

\begin{theorem} \label{thm:subflat_branching_asymptotic_rank}
Let $X$ and $Y$ be finite-dimensional CAT(0) cube complexes of asymptotic rank $n$, and let $f\colon X\to Y$ be a quasiisometric embedding. If $F\subset X$ is a directionally branch-complemented $n$--flat and $H\subset F$ is a branching subflat, then $f(H)$ lies within finite Hausdorff distance of a semisingular flat of $Y$.
\end{theorem}

\begin{proof}
Let $H\subset F$ be a branching $k$--flat. By definition, there are $n$--flats $F_1,\dots,F_m$ such that $H=\bigcap_{i=1}^mF_i$. According to \cref{prop:subflat_image_union_orthants}, there is a finite union of orthants at finite Hausdorff distance from $f(H)$. By \cref{prop:boundary_quasi_flat_sphere}, this implies that $\partial_Tf(H)\cong\mathbb S^{k-1}$.

We know argue as in \cref{thm:subflat_branching_in_fully_branching}, where $F'$ is semisingular by \cref{lem:top_rank_semisingular} and the resulting $H'$ is semisingular by \cref{lem:semisingular_subflat_boundary}.
\end{proof}


\begin{theorem} \label{thm:strong_branching_asymptotic_rank}
Let $X$ be an $n$--dimensional CAT(0) cube complex, and let $Y$ be finite-dimensional CAT(0) cube complex of asymptotic rank $n$. For every $q\ge1$ and $r\ge0$ there exists $D=D(n,q,r,Y)$ such that the following holds for every $q$--quasiisometric embedding $f\colon X\to Y$.

If $H\subset X$ is a directionally strongly $r$--branch-complemented $k$--flat for some $k\le n$, then $f(H)$ lies at Hausdorff distance at most $D$ from a semisingular $k$--flat of $Y$.
\end{theorem}

\begin{proof}
The argument is the same as in \cref{thm:strong_fully_branching}, with the exception that we use \cref{lem:top_rank_semisingular} in place of \cref{lem:top_dim_singular} and the constant given by \cref{thm:fully_branching_general} now depends on $Y$.
\end{proof}

\begin{remark} \label{rem:no_asrk}
An attempt at an analogous proof of \cref{thm:strong_fully_branching} fails for the case where both domain and codomain are finite-dimensional and of asymptotic rank $n$. This is because the analogue of \cref{prop:intersection_semisingular} in this setting does not have uniformity. 

Indeed, we cannot get uniform constants in \cref{lem:image_of_intersection_of_flats_top_rank}, even if we assume that the flats are semisingular in a CAT(0) cube complex. Essentially, this is because there is no way to lower-bound the \emph{angle} at which semisingular flats begin to diverge, even though asymptotically they diverge orthogonally. Consider pairs of geodesics in (a CAT(0) square complex quasiisometric to) the hyperbolic plane, for instance: they can fellow-travel for an arbitrarily long time.

For the same reason, one could build a non-homogeneous $X$ and a quasiisometric embedding $X\to X$ showing that an analogue of \cref{prop:intersection_semisingular} for the case where the rank and the dimension do not agree does not hold with uniform bounds.

However, it is possible that one could leverage the additional branching properties in the assumptions of \cref{thm:strong_fully_branching,thm:strong_branching_asymptotic_rank} to obtain a uniform analogue of \cref{prop:intersection_semisingular} and hence of \cref{thm:strong_fully_branching} in this setting.
\end{remark}

\begin{corollary} \label{cor:orthants_asymptotic_rank}
Let $X$ be an $n$--dimensional CAT(0) cube complex, and let $Y$ be finite-dimensional CAT(0) cube complex with asymptotic rank $n$. If $f\colon X\to Y$ is a quasiisometric embedding, then there exists $D\ge0$ such that $f$ maps every semisingular orthant in every directionally strongly $r$--branch-complemented flat within Hausdorff distance at most $D$ of a semisingular orthant of $Y$.
\end{corollary}

\begin{proof}
Given a directionally strongly $r$--branch-complemented $k$--flat $H$, let $\gamma_1,\dots,\gamma_k\subset H$ be geodesics as in \cref{def:strong_fully_branching}. According to \cref{lem:dbc_semisingular}, every $\gamma_i$ is singular, and so is $H$. Consequently, the singular geodesics in $H$ are precisely the parallels of the $\gamma_i$ inside $H$.

By \cref{thm:strong_branching_asymptotic_rank}, there exists $D\ge0$ such that every parallel of every $\gamma_i$ is mapped within Hausdorff distance at most $D$ of a semisingular geodesic of $Y$. \cref{lem:orthants_to_semisingular_orthants} now provides a constant $D'$ such that $f$ maps every semisingular orthant in $H$ within Hausdorff distance at most $D'$ of some semisingular orthant of $Y$.
\end{proof}

In the following, we describe a generalisation of the above results that includes products of certain hyperbolic spaces. This should be compared to the main results of \cite{bowditch:quasiisometric}.

\begin{remark} \label{rem:cor_for_hyperbolic_spaces}
It was shown in \cite[Thm~1.8]{haglundwise:combination}, using \cite[Thm~1.1]{bonkschramm:embeddings}, that every hyperbolic group is quasiisometric to a proper CAT(0) cube complex. More generally, a hyperbolic space is quasiisometric to a finite-dimensional CAT(0) cube complex if and only if it has finite \emph{asymptotic dimension} \cite[Thm~4.8]{petyt:mapping}. Such CAT(0) cube complexes have asymptotic rank one.

\cref{thm:fully_branching_general} and the results of \cref{subsec:main_general} can therefore be stated in greater generality: when $X$ or $Y$ can have dimension greater than their asymptotic rank, we can replace them with $X=X_1\times\cdots\times X_k\times X'$ and $Y=Y_1\times \cdots\times Y_m\times Y'$ respectively, where the $X_i$ and $Y_i$ are hyperbolic spaces of finite asymptotic dimension, and $X'$ and $Y'$ are finite-dimensional CAT(0) cube complexes of asymptotic rank $n-k$ and $n-m$, respectively.

To see this, let $X''=\prod_{i=1}^kX_i$ and $Y''=\prod_{i=1}^kY_i$. Let $f_X:X''\to Q_X$ and $f_Y:Y''\to Q_Y$ be factor-wise quasiisometries to products of CAT(0) cube complexes. By the Morse lemma there exists $D$ such that $f_X$ and $f_Y$ both send geodesics in the factors of $X''$ and $Y''$, respectively, within Hausdorff-distance at most $D$ of geodesics in the corresponding factors of $Q_X$ and $Q_Y$, respectively. Hence $f_X$ and $f_Y$ send $k$--flats uniformly Hausdorff-close to $k$--flats. Moreover, note that for every $r\geq 0$, there exists $r'\geq 0$ such that $f_X$ maps directionally $r$--branch-complemented $p$--flats to directionally $r'$--branch-complemented $p$--flats for every $p\leq k$.


Lastly, we note that when $X'$ is trivial and the $X_i$ are non-elementary hyperbolic groups (or more generally, \emph{bushy} hyperbolic spaces, see \cite{bowditch:quasiisometric}), there exists an $r\geq 0$ such that every top-rank flat in $X$ is directionally $r$--branch-complemented and every flat is branching inside a top-rank flat. Hence, \cref{thm:subflat_branching_asymptotic_rank} shows that every quasiisometry $\phi\colon X\to Y$ maps every flat Hausdorff-close to a flat. Compare \cite[Thm~1.1]{bowditch:quasiisometric}. 
\end{remark}

\section{Induced maps of boundaries} \label{subsec:boundaries}

We conclude by describing how the results of \cref{sec:flats_rigidity,sec:fully_branching_theorems} can be used to induce maps between certain subsets of Tits boundaries of CAT(0) cube complexes. Recall from \cref{lem:embedding-Titsboundary-into_link_and_Tits_boundary} that if $X$ is a CAT(0) space and $\hat X$ is an asymptotic cone of $X$ with respect to a fixed basepoint, then there is a natural isometric embedding $\varphi_T\colon \partial_TX\to\partial_T\hat X$.

\begin{definition}[Singular boundary] \label{def:singular_boundary_graph}
Let $X$ be a CAT(0) cube complex or the asymptotic cone of one. The \emph{singular boundary graph} of $X$, denoted $\partial_{\mathrm{sing}}X$, is the following graph. The vertices of $\partial_{\mathrm{sing}}X$ are the endpoints in $\partial_T X$ of singular geodesic rays in $X$. Two vertices of $\partial_{\mathrm{sing}}X$ are joined by an edge whenever they admit singular geodesic representatives that span a singular 2--orthant of $X$.

The \emph{semisingular boundary graph} $\partial_{\mathrm{ss}}X$ has a vertex for each endpoint of a semisingular geodesic ray. Two vertices are joined by an edge whenever their $\varphi_T$-images in $\partial_T\hat X$ admit representatives that span a singular 2--orthant in $\hat X$.
\end{definition}

Note that, in view of \cref{lem:orthants_are_weyl_cones} and \cref{lem:asymptotic_Weyl_cone}, the edge sets of $\partial_{\mathrm{sing}}X$ and $\partial_{\mathrm{ss}}X$ are exactly the sets of pairs of vertices at angle $\frac\pi2$. In other words, if we give the edges of $\partial_{\mathrm{sing}}X$ or $\partial_{\mathrm{ss}}X$ length $\frac\pi2$, then it is isometrically embedded in $\partial_TX$.

If $X$ is the universal cover of a Salvetti complex, then $\partial_TX$ has a natural cell structure with respect to which this embedding maps into the 1--skeleton. However, the Tits boundary of a general CAT(0) cube complex has no such structure.

\begin{example} \label{eg:singular_boundary}
We illustrate the differences between different notions of boundary.
\begin{itemize}
\item   If $X$ is the CAT(0) square complex from \cref{eg:corner_to_corner}, then $\partial_TX=\partial_{\mathrm{ss}}X$ is a pair of points, but $\partial_{\mathrm{sing}}X=\varnothing$.
\item   If $X$ is the ``staircase'' square complex, obtained from the orthant $[0,\infty)^2$ by deleting all squares containing a point $(x,y)$ with $y>x$, then $\partial_TX$ is an arc of length $\frac\pi4$, whereas $\partial_{\mathrm{sing}}X=\partial_{\mathrm{ss}}X$ is a single point. This also distinguishes $\partial_{\mathrm{sing}}$ from other combinatorial notions of boundary for CAT(0) cube complexes, such as the Roller boundary \cite{roller:poc}, simplicial boundary \cite{hagen:simplicial}, and simplicial Roller boundary \cite{guralnik:coarse,genevois:contracting}, which are all 1--simplices for this particular square complex $X$.
\item   Let $\alpha_2>\alpha_1>0$. If $X$ is the subcomplex of $[0,\infty)^2$ consisting of all squares whose points $(x,y)$ all satisfy $\alpha_1 x\le y\le\alpha_2x$, then $\partial_TX$ is an arc but $\partial_{\mathrm{sing}}X=\partial_{\mathrm{ss}}X=\varnothing$.
\item   Let $X$ be the subcomplex of $[1,\infty)^2$ obtained by deleting, for each $n$, the subcomplex $[(2^{2^n},0),(2^{2^{n+1}},2^{n+1}-1)]$ for all $n$. This is modelled on removing the region below the $\log x$ graph. The Tits boundary of $X$ is an arc: one endpoint corresponds to the vertical geodesic, and the other corresponds to the piecewise-linear geodesic $\gamma$ through the points $(2^{2^n},2^n)$. Note that the ultralimit of $\gamma$ is singular in every asymptotic cone where it exists. The semisingular boundary graph is thus an edge. On the other hand, the singular boundary graph of $X$ is a single point, corresponding to the vertical geodesic.
\end{itemize}
\end{example}

A geodesic ray is said to be \emph{branch-complemented} if it is contained in a branch-complemented geodesic. Note that this implies that it is semisingular, by \cref{lem:branching_singular}.

\begin{definition}[Branch-complemented boundary]\label{def:fully_branching_boundary}
Let $X$ be a CAT(0) cube complex. The \emph{(strongly) branch-complemented boundary graph} of $X$, denoted $\partial_{\mathrm{sbc}}X$ or $\partial_{\mathrm{bc}}X$ accordingly, is the following subgraph of $\partial_{\mathrm{ss}}X$. It has a vertex for each endpoint in $\partial_T X$ of (strongly) branch-complemented geodesic rays that lie in directionally branch-complemented $n$--flats. Vertices $\xi$ and $\eta$ are joined by an edge whenever there exists a directionally (strongly) branch-complemented $2$--flat $H\subseteq X$ such that $\xi,\eta\in\partial_T H$.
\end{definition}

Note that the condition that the rays representing points of $\partial_{\mathrm{sbc}}X$ lie in directionally branch-complemented $n$--flats is an empty condition: see \cref{def:strong_fully_branching}. It is only a restriction in the case of $\partial_{\mathrm{bc}}X$.

For universal covers of Salvetti complexes, the branch-complemented boundary is an induced subgraph of the singular boundary graph, since the existence of a directionally branch-complemented $2$--flat is automatic. Indeed, if $Q$ is a $2$--orthant whose axis rays are asymptotic to branch-complemented geodesic rays, then, up to replacing $Q$ by a suborthant, we may assume that its axis rays have the same labels as the branch-complemented rays by \cite[Cor.~2.6]{bestvinakleinersageev:asymptotic}. Therefore, the axes can be extended, using the same labels, to branch-complemented geodesics, and hence span a directionally branch-complemented $2$--flat by \cref{rem:fully_branching_geod_in_raag}.

Already \cref{thm:subflat_branching_in_fully_branching} (respectively \cref{thm:subflat_branching_asymptotic_rank}) shows that quasiisometries between CAT(0) cube complexes of dimension $n$ (respectively asymptotic rank $n$) map the vertex set of $\partial_{\mathrm{bc}}X$ into the vertex set of $\partial_{\mathrm{sing}}Y$ (respectively $\partial_{\mathrm{ss}}Y$). Understanding what happens to edges is more subtle. 
The following two corollaries summarise the implications of our results on the level of boundaries.

\begin{corollary}\label{cor:embedding_fully_branching_boundary}
Let $X$ and $Y$ be $n$--dimensional CAT(0) cube complexes, and let $f\colon X\to Y$ be a quasiisometric embedding. Then the following hold:
\begin{itemize}
\item   $f$ induces a graph embedding $\partial_{\mathrm{sbc}}X \to \partial_{\mathrm{sing}}Y$.
\item   If every branch-complemented geodesic of $X$ is contained in some directionally branch-complemented $n$--flat, then $f$ induces a graph embedding $\partial_{\mathrm{bc}}X \to \partial_{\mathrm{sing}}Y$.
\end{itemize}
\end{corollary}

\begin{proof}
As noted above, \cref{thm:subflat_branching_in_fully_branching} shows that $f$ induces a map from the vertex set of $\partial_{\mathrm{bc}}X$ to the vertex set of $\partial_{\mathrm{sing}}Y$.

Let $\xi_1$ and $\xi_2$ be two adjacent vertices of the (strongly) branch-complemented boundary graph of $X$. There is a directionally (strongly) branch-complemented 2--flat $H$ such that $\xi_1,\xi_2\in\partial_TH$, and there is a singular orthant $O\subset H$ whose two axis rays represent $\xi_1$ and $\xi_2$. By applying either \cref{cor:embedding_orthant_strong_fullybranching} or \cref{thm:fully_branching_2_flat}, according to which case we are in, we see that there is a singular orthant of $Y$ at finite Hausdorff distance from $f(O)$, so the images of $\xi_1$ and $\xi_2$ are adjacent in $\partial_{\mathrm{sing}}Y$.
\end{proof}


\begin{corollary} \label{cor:boundary_map_asymptotic_rank}
Let $X$ be an $n$--dimensional CAT(0) cube complex, and let $Y$ be finite-dimensional CAT(0) cube complex with asymptotic rank $n$. Every quasiisometric embedding $f\colon X\to Y$ induces a graph embedding $\partial_{\mathrm{sbc}}X\to\partial_{\mathrm{ss}}Y$.
\end{corollary}

\begin{proof}
\cref{thm:subflat_branching_asymptotic_rank} shows that $f$ induces a map from the vertex set of $\partial_{\mathrm{sbc}}X$ to the vertex set of $\partial_{\mathrm{ss}}Y$. Let $\xi_1,\xi_2\in\partial_{\mathrm{sbc}}X$ be adjacent vertices. There is a directionally strongly branch-complemented 2--flat $H$ with $\xi_1,\xi_2\in\partial_TH$, which is semisingular by \cref{lem:dbc_semisingular}. Thus there is a semisingular orthant $O\subset H$ whose axis rays represent $\xi_1$ and $\xi_2$. \cref{cor:orthants_asymptotic_rank} now shows that $f(O)$ lies at finite Hausdorff distance from a semisingular orthant of $Y$, so the images of $\xi_1$ and $\xi_2$ are adjacent in $\partial_{\mathrm{ss}}Y$.
\end{proof}

The following is a consequence of \cref{cor:fully_branching_implies_flat_in_cone} by the same argument as in \cref{cor:embedding_fully_branching_boundary}.

\begin{corollary} \label{prop:embedding_boundary_asymptotic_cone}
    Let $X$ be an $n$--dimensional CAT(0) cube complex, and let $Y_1,\dots, Y_n$ be hyperbolic spaces. Let $f\colon X\to Y=Y_1\times \cdots \times Y_n$ be a quasiisometric embedding. Then the following hold:
    \begin{itemize}
        \item $\partial_{\mathrm{sbc}}X$ is $n$--colourable.
        \item If every branch-complemented geodesic of $X$ is contained in some directionally branch-complemented $n$--flat, then $\partial_{\mathrm{bc}}X$ is $n$--colourable.
    \end{itemize}
\end{corollary}

\begin{proof}
    Let $f_\omega\colon  X_\omega \to Y_\omega$ be the induced biLipschitz embedding of asymptotic cones. By \Cref{lem:embedding-Titsboundary-into_link_and_Tits_boundary}, we have an isometric embedding $\partial_T X\hookrightarrow \partial_T X_\omega$. Let $\xi$ be a vertex of $\partial_{\mathrm{sbc}}X$. Its representative ray $\gamma$ lies in a directionally branch-complemented geodesic. Taking the ultralimit, we get a directionally $0$--branch-complemented geodesic $\gamma_\omega\subset X_\omega$.

    For each $i\in\{1,\dots,n\}$, the asymptotic cone $(Y_i)_\omega$ is a real tree, so there exists a simplicial tree $T_i$ such that there is an isometric embedding $(Y_i)_\omega\to(T_i)_\omega$. Letting $Z$ be the $n$--dimensional CAT(0) cube complex $Z=\prod_{i=1}^nT_i$, we get an isometric embedding $\iota:Y_\omega\to Z_\omega$.
    
    According to  \Cref{prop:fully_branching_implies_geod_in_cone}, we have that $(\iota f)_\omega(\gamma_\omega)$ is a singular geodesic in $Z_\omega$. We therefore obtain a map $\partial_{\mathrm{sbc}}X \longrightarrow \partial_{\mathrm{sing}} Z_\omega$, and it is injective because distinct elements of $\partial_{\mathrm{sbc}}X$ come from geodesic rays at infinite Hausdorff distance.

    Now let $\xi,\eta$ be adjacent in $\partial_{\mathrm{sbc}}X$. By definition, they are represented by the two axis rays of a singular orthant in a directionally branch-complemented $2$--flat $H\subset X$. Passing to the cone and applying \Cref{prop:fully_branching_implies_geod_in_cone} to $H_\omega$ and $(\iota f)_\omega$ as before, we find that $(\iota f)_\omega(H_\omega)$ is a singular $2$--flat in $Z_\omega$. Hence the images of $\xi,\eta$ span a singular $2$--orthant in $Z_\omega$, so they are adjacent in $\partial_{\mathrm{sing}} Z_\omega$. Under the additional assumption, the same argument applies to $\partial_{\mathrm{bc}}X$. The conclusion follows from the fact that $Z_\omega$ is the product of $n$ real trees, and hence its boundary is $n$--colourable.
\end{proof}

In \cite{baderbensaidpetyt:quasiisometric:rigidity}, \cref{cor:embedding_fully_branching_boundary} is used to study rigidity properties of quasiisometric embeddings between right-angled Artin groups. The following is a small sample of this. Although it uses only \cref{prop:embedding_boundary_asymptotic_cone} and thus does not need the results of \cref{sec:fully_branching_theorems}, those results will be important when considering more general right-angled Artin groups in \cite{baderbensaidpetyt:quasiisometric:rigidity}.
Recall that $C_n$ denotes the cycle graph with $n$ vertices, and that $A_{C_n}$ denotes the corresponding right-angled Artin group.

\begin{corollary}\label{eg:c5}
Let $Y_1$ and $Y_2$ be hyperbolic spaces. If $n>1$ is odd, then there is no quasiisometric embedding $A_{C_n}\to Y_1\times Y_2$.
\end{corollary} 

\begin{proof}
If $n=3$ then $A_{C_3}=\Z^3$ and the result is clear. Otherwise, identify $A_{C_n}$ with the universal cover of its Salvetti complex. The geodesic rays based at the identity (or at any other vertex) corresponding to positive powers of the generators are strongly 1--branch-complemented, so they induce a $C_n$-subgraph of $\partial_{\mathrm{sbc}}(A_{C_n})$. In particular,  $\partial_{\mathrm{sbc}}(A_{C_n})$ is not $2$--colourable. 
The result follows from \cref{prop:embedding_boundary_asymptotic_cone}.
\end{proof}

\appendix
\section{Appendix: Symmetric spaces and buildings of the same type}\label{sec:appendix}

\subsection{Type \texorpdfstring{$A_1^n$}{A1n}}\label{appendix:type_A_1}

In this short discussion we will explain that the results of \cref{sec:fully_branching_theorems} hold for symmetric spaces and Euclidean buildings of type $A_1^n$ without any thickness assumption.

Euclidean buildings of spherical type $A_1^n$ are products of $n$ metric trees. Each is quasiisometric to a simplicial tree. The quasiisometry sends geodesics to geodesics and the product of these quasiisometries sends branching $k$--flats to branching $k$--flats. Hence, all the theorems in \cref{sec:fully_branching_theorems} follow by replacing ``$n$-dimensional CAT(0) cube complex'' with ``Euclidean building of spherical type $A_1^n$'' by composing with this quasiisometry or its inverse.

\subsection{General spherical types} \label{sec:general_spherical_type}

The results of \cref{sec:fully_branching_theorems} are also true for general Euclidean buildings and symmetric spaces. These results were already known, see \cite{fisherwhyte:quasiisometric} and \cite{nguyen:quasiisometric}. Nonetheless, we will explain here how they can be derived from our work.

Recall that in a symmetric space of non-compact type or a thick Euclidean building, a flat is \emph{singular} if it is the intersection of finitely many maximal flats.

We aim to get analogues of \cref{thm:fully_branching_general,thm:strong_fully_branching,cor:embedding_orthant_strong_fullybranching}, namely \cref{thm:fully_branching_most_general} and \cref{thm:strong_fully_branching_general_case}, for symmetric spaces of non-compact type and thick Euclidean buildings. We will amend the proof of \cref{thm:fully_branching_general}, and derive the analogues of \cref{thm:strong_fully_branching} and \cref{cor:embedding_fully_branching_boundary} directly from it, using work of Kleiner and Leeb, namely \cite[Thm~3.1]{kleinerleeb:rigidity:invariant}.

Figure~\ref{fig:logic_diagram_main_theorems} shows the reliance between the statements in the paper leading to \cref{thm:fully_branching_general}. In black, statements for general complete CAT(0) spaces or proofs that hold verbatim in our case. In blue, we replace \cref{prop:quasiflats-from-bilip-flats} with \cite[Prop.~7.1.1]{kleinerleeb:rigidity}, in pink, a statement for which the proof holds for symmetric spaces of non-compact type and thick Euclidean buildings with replacing ``orthants'' with ``Weyl cones''. In red we have statements whose proofs need amendments that would be described in this appendix.

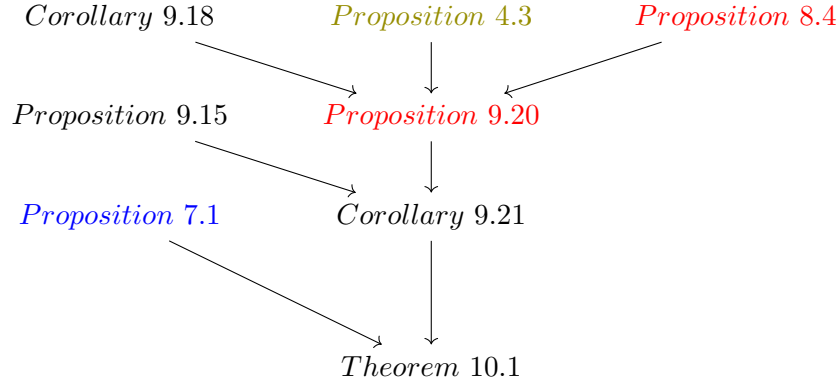
\begin{figure}[ht]
\centering
\begin{tikzcd} 
	{\cref{lem:rectifiable_parallel_flats}} & \color{olive}{{\cref{prop:singular_quasi_flats}}} & \color{red}{{\cref{prop:top_argument}}} \\
	{\cref{lem:ultralimit_fully_branching_flat}} & \color{red}{{\cref{prop:fully_branching_implies_geod_in_cone}}} \\
	\color{blue}{{\cref{prop:quasiflats-from-bilip-flats}}} & {\cref{cor:fully_branching_implies_flat_in_cone}} \\
	\\
	& {\cref{thm:fully_branching_general}} 
	\arrow[from=1-1, to=2-2]
	\arrow[from=1-2, to=2-2]
	\arrow[from=1-3, to=2-2]
	\arrow[from=2-1, to=3-2]
	\arrow[from=2-2, to=3-2]
	\arrow[from=3-1, to=5-2]
	\arrow[from=3-2, to=5-2]
\end{tikzcd}
\caption{Logic diagram of \cref{thm:fully_branching_general}} \label{fig:logic_diagram_main_theorems}

\end{figure}

In more detail:
in
\cref{prop:singular_quasi_flats}, similarly to the CAT(0) cube complex case, if we assume $Y$ is a symmetric space or building and we assume the geodesics are mapped within Hausdorff distance $D$ from singular geodesics we can take the resulting $k$--flat to be singular as well. 
\cref{prop:2_orthant_in_target} can be adapted to symmetric spaces and buildings as well, but is not needed for the proof of the main theorems, since all singular flats in thick Euclidean buildings and symmetric spaces are directionally strongly branch-complemented.

\cref{prop:quasiflats-from-bilip-flats} can be replaced with \cite[Prop~7.1.1]{kleinerleeb:rigidity} and that is the only relevant result from \cref{sec:back_from_cone}.

\cref{sec:top_arguments} is used to get \cref{prop:fully_branching_implies_geod_in_cone}, we will explain in detail the small necessary changes in both. \cref{thm:fully_branching_most_general} will then follow.



\subsubsection{Local separation of singular flats}

The following is an analogue of \Cref{prop:top_argument} that applies to Euclidean Coxeter complexes of arbitrary spherical type. The assumptions are stronger than in the $A_1^n$ type, but they will be satisfied in the cases of symmetric spaces and thick Euclidean buildings.

\begin{proposition}[Analogue of \cref{prop:top_argument}]\label{prop:top_argument_walls}
Let $E$ be a Euclidean Coxeter complex of dimension $n$. Let $x \in E$, let $r>0$, and let $f\colon  B(x,r) \to E$ be a topological embedding. 

If that $f(H \cap B(x,r))$ is contained in a finite union of singular $(n-1)$--flats for every singular $(n-1)$--flat $H$ containing~$x$, then there exists $r' \leq r$ such that $f(H \cap B(x,r'))$ is contained in a single singular $(n-1)$--flat for every singular $(n-1)$--flat $H$ containing $x$. Moreover, $f(\gamma \cap B(x,r'))$ is contained in a single singular geodesic for every singular geodesic $\gamma$ containing $x$.
\end{proposition}

\begin{proof}
Since every singular geodesic is an intersection of singular $(n-1)$--flats, and only finitely many singular geodesics contain $x$, there exists $r'\le r$ such that, for every singular geodesic $\gamma$ containing $x$, the path $f(\gamma\cap B(x,r'))$ is contained in at most two singular geodesics. 

Choose $\delta>0$ such that $B(f(x),\delta)\subseteq f\bigl(B(x,r')\bigr)$. Consider the $(n-1)$--spheres $S:=S(x,r')$ and $S':=S(f(x),\delta)$. Endow $S$ and $S'$ with the spherical Coxeter complex structures inherited from $\partial_TE$. The choice of $r'$ ensures that $f$ induces a well-defined bijection $\tilde f\colon V(S)\to V(S')$ between their vertex sets.

Let show that if $u,v\in V(S)$ are non-adjacent, then $\tilde f(u)$ and $\tilde f(v)$ are also non-adjacent.  Indeed, since $S$ is a spherical Coxeter complex, if $u$ and $v$ are not adjacent then there exists a wall $s\subseteq S$ separating $u$ from $v$. Let $H\subseteq E$ be the corresponding singular $(n-1)$--flat containing $x$. By assumption, $f\bigl(H\cap B(x,r')\bigr)$ is contained in a finite union of singular $(n-1)$--flats, hence $f\bigl(H\cap B(x,r')\bigr)\cap S'$ is contained in the $(n-2)$--skeleton of $S'$. Furthermore, this set separates $\tilde f(u)$ and $\tilde f(v)$, so $\tilde f(u)$ and $\tilde f(v)$ cannot be adjacent.

As the 1-skeletons of $S$ and $S'$ are finite isomorphic graphs, they have the same number of non-edges. Hence if $v$ is a neighbour of $u$, then $f(v)$ must be a neighbour of $f(u)$. This shows that $\tilde f$ is a graph isomorphism. Consequently, $\tilde f$ sends walls to walls. Therefore, for every singular $(n-1)$--flat $H$ containing $x$, the image $f(H \cap B(x,r'))$ is contained in a single singular $(n-1)$--flat. The conclusion for singular geodesics containing $x$ follows as they are intersections of singular $(n-1)$--flats
\end{proof}

\subsubsection{The analogue of \cref{prop:fully_branching_implies_geod_in_cone}}





\cref{thm:fully_branching_general} will follow from an analogue of \cref{cor:fully_branching_implies_flat_in_cone}. This will follow immediately from an analogue of \cref{prop:fully_branching_implies_geod_in_cone}.

We note that asymptotic cones of symmetric spaces of non-compact type or thick Euclidean buildings are $\mathbb{R}$--buildings (by that we mean, branching everywhere) of the same spherical type. Every singular flat in such spaces is directionally $0$--branch-complemented.

\begin{proposition}[Analogue of \cref{prop:fully_branching_implies_geod_in_cone}]\label{prop:fully_branching_implies_geod_in_cone_general_case}
Let $\hat X$ and $\hat Y$ be $\mathbb{R}$-buildings of the same spherical type and let $f\colon  \hat X\to \hat Y$ be a biLipschitz embedding. If $\hat F \subseteq \hat X$ is a singular flat in $\hat X$, then $f(\hat F)$ is a singular flat in $\hat Y$.
\end{proposition}
The proof is very similar to the proof of \cref{prop:fully_branching_implies_geod_in_cone}, where we use \cite[Lem.~3.1]{fisherwhyte:quasiisometric} instead of \cref{prop:Bowditch_bilipflat_cubulated} and we show that the conditions of \cref{prop:top_argument_walls} hold, instead of \cref{prop:top_argument}. These conditions are stronger than the conditions in the cube complex analogue, but the assumption on our spaces (branching everywhere) are stronger as well.

We note that the usage of \cref{prop:singular_quasi_flats}, is done similarly to the CAT(0) cube complex case, where the resulting flat is singular, as explained in the beginning of \cref{sec:general_spherical_type}

\begin{proof}
The proposition will follow from the case $k=1$ using \cref{prop:singular_quasi_flats} for biLipschitz maps where $D=0$, as singular flats are spanned by singular geodesics. Thus, it is enough to show that if $\hat \gamma_0\subseteq \hat X$ is a singular geodesic, then $f(\hat\gamma_0)$ is a singular geodesic.

    Let $\hat \gamma_0 \subseteq \hat X$ be a singular geodesic. Let $\hat F$ be an $n$--flat containing $\hat\gamma_0$.
    
     Let $S \subseteq \hat F$ be the singular set of $f|_{\hat F}$, that is, for every $x\in \hat F \setminus S$, there exists a neighbourhood $U\subseteq \hat F$ such that $f(U)$ is contained in a single $n$--flat. By \cite[Lem.~3.1]{fisherwhyte:quasiisometric}, $S$ is $\mathcal{H}^{n-2}$-measurable and countably $\mathcal{H}^{n-2}$-rectifiable. By \cref{lem:rectifiable_parallel_flats}, almost every parallel of $\hat\gamma_0$ in $\hat F$ does not intersect $S$. By continuity of $f$, it therefore suffices to show that if $\hat \gamma \subseteq \hat F$ is a parallel of $ \hat \gamma_0$ that does not intersect $S$, then $f (\hat \gamma)$ is a geodesic. 

Let $\hat \gamma \subseteq \hat F$ be a parallel of $ \hat \gamma_0$ that does not intersect $S$, and let $x \in \hat \gamma$. 
     
It is flat for $f|_{\hat F}$, because $\hat \gamma \cap S = \varnothing$. Thus, there exists $r>0$ such that $f (B_{\hat F}(x,r))$ is a contained in a unique $n$--flat $\hat F' \subseteq \hat Y$.
Let $\hat H$ be a singular $(n-1)$--flat transverse to $\hat \gamma$ and containing $x$. Both $\hat\gamma$ and $\hat H$ are branching, as $\hat X$ is an $\mathbb{R}$-building.

By \cite[Lem.~3.1]{fisherwhyte:quasiisometric}, every bi-Lipschitz $n$--flat is contained in a finite union of $n$--flats. Therefore, $f(\hat H)$ is a finite intersection of such and so there is an $r>0$ such that $f(\hat H \cap B(x,r))$ is in a finite union of singular $(n-1)$--flats. Similarly, $f(\hat\gamma\cap B(x,r))$ is in a finite union of geodesics. As $f(\hat\gamma)$ is homeomorphic to a real interval, up to decreasing $r$, $f(\hat\gamma\cap B(x,r))$ is in the union of two singular geodesics meeting at $f(x)$.

As there are only finitely many singular $(n-1)$--flats through $x$, we can assume $r$ satisfies the property above for all of them. It follows from \Cref{prop:top_argument_walls} that there exists $0<r' \leq r$ such that $f (\hat \gamma \cap B_{\hat F}(x,r'))$ is contained in a singular geodesic of $\hat F'$.

Since $x \in \hat \gamma$ was arbitrary, this shows that $f(\hat \gamma)$ is a local geodesic in $\hat Y$. Since $\hat Y$ is CAT(0), we conclude that $f(\hat \gamma)$ is a geodesic.
    \end{proof}

\begin{remark}
\cref{prop:top_argument} used the specific structure of $\Sigma_n$, the simplicial $n$--fold suspension of the $0$--sphere, and so does not hold for general type, while \cref{prop:top_argument_walls} does hold, but the assumptions are too strong.    
    In the CAT(0) cube complex case, the assumptions of \cref{prop:top_argument_walls} are not necessarily satisfied, as not all singular $(n-1)$--flats have to be branching.
\end{remark}
Now the analogue of \cref{cor:fully_branching_implies_flat_in_cone} follows verbatim. Using the fact that for symmetric spaces of non-compact type and thick Euclidean buildings there exists an $r$ such that every singular flat is directionally $r$--branch-complemented (see \cref{ex:fully_branching}), we obtain the following.

\begin{corollary}[Analogue of \cref{cor:fully_branching_implies_flat_in_cone}]\label{cor:fully_branching_implies_flat_in_cone_general_case}
    Let $X$ and $Y$ each be a symmetric space of non-compact type, or a thick Euclidean building and assume that they have the same spherical type. Let $\hat X$ and $\hat Y$ be asymptotic cones of $X$ and $Y$, respectively, and let $f\colon  \hat X\to \hat Y$ be a biLipschitz embedding. If $\hat F\subseteq \hat X$ is the ultralimit of a sequence of singular $k$--flats in $X$, then $f(\hat F)$ is a singular $k$--flat in $\hat Y$.
\end{corollary}

\subsubsection{The main theorems}

The proof of \cref{thm:fully_branching_general} now follows word for word with replacing \cref{prop:quasiflats-from-bilip-flats} with \cite[Prop~7.1.1]{kleinerleeb:rigidity} and \cref{cor:fully_branching_implies_flat_in_cone} with \cref{cor:fully_branching_implies_flat_in_cone_general_case}. We get the following, as every $n$--flat is directionally $r$--branch-complemented. In what follows, the symmetric spaces and Euclidean buildings under consideration are not assumed to be irreducible.

\begin{theorem}\label{thm:fully_branching_most_general}
Let $X$ and $Y$ each be a symmetric space of non-compact type, or a thick Euclidean building and assume that they  both have rank $n$ and the same spherical type. Let $f\colon X \to Y$ be a $q$-quasiisometric embedding. There exists a constant $D = D(q,n, \text{spherical type})$ such that $f(F)$ lies within Hausdorff distance at most $D$ of some $n$--flat $F' \subseteq Y$ for every $n$--flat $F \subseteq X$. 

\end{theorem}

As a corollary, under the conditions of \cref{thm:fully_branching_most_general}, we get a combinatorial embedding $\partial_T X\to \partial_T Y$. Indeed, if $F\subseteq X$ is an $n$--flat, then $f(F)$ lies at finite Hausdorff distance from an $n$--flat $F'\subseteq Y$, and hence determines an apartment $\partial_TF'\subseteq\partial_TY$. Since vertices in a thick spherical building are intersections of finitely many apartments, and since $f$ coarsely preserves intersections of flats, this determines a combinatorial embedding $\partial_TF\to\partial_TF'$. As $X$ and $Y$ have the same spherical type, the combinatorial embedding $\partial_TF\to\partial_TF'$ is an isomorphism. These maps are compatible on overlaps of apartments, and therefore induce a combinatorial embedding $\partial_TX\to\partial_TY$.


\begin{corollary}\label{cor:embedding_fully_branching_boundary_general}
Let $X$ and $Y$ each be either a symmetric space of non-compact type, or a thick Euclidean building, and assume that they have the same spherical type and rank. If $f\colon X \to Y$ is a quasiisometric embedding, then $f$ induces an embedding
$$
\partial_{\mathrm{sing}}X \to \partial_{\mathrm{sing}}Y
$$
as a subgraph.
\end{corollary}

By \cite[Thm~3.1]{kleinerleeb:rigidity:invariant}, the image of the induced map is a subbuilding, which is the boundary of a subsymmetric space or building $Y'\subseteq Y$ and $f\colon X\to Y'$ is a quasiisometry. This yields the following consequence of \cite[Theorem~1.1.3]{kleinerleeb:rigidity}.

\begin{theorem}[Analogue of \cref{thm:strong_fully_branching} and \cref{cor:embedding_orthant_strong_fullybranching}]\label{thm:strong_fully_branching_general_case}
Let $X$ and $Y$ each be a symmetric space of non-compact type or a thick Euclidean building. Assume that $X$ and $Y$ have rank $n$, the same spherical type, and no rank-one factors. For every $q\geq 1$ and $r\geq 0$, there exists a constant $D$ such that the following holds: if $f\colon  X \to Y$ is a $q$-quasiisometric embedding, and $H \subseteq X$ is a singular $p$--flat, then $f(H)$ lies at Hausdorff distance at most $D$ from a singular $p$--flat in $Y$.

Furthermore, there exists $D'(q,D)$, such that $f$ maps every singular Weyl cone within Hausdorff distance $D'$ from a singular Weyl cone.
\end{theorem}

\begin{theorem}
Let $X = X_1 \times \cdots \times X_n$ and $Y = Y_1 \times \cdots \times Y_n$ be products of irreducible symmetric spaces of non-compact type and thick Euclidean buildings, with no compact, Euclidean, or rank-one factors. Assume moreover that $X$ and $Y$ have the same spherical type and rank. Every quasiisometric embedding $f\colon  X \to Y$ is at finite distance from a product map $(f_1,\dots,f_n):X_1\times\cdots\times X_n\to Y_1\times\cdots\times Y_n$, after possibly permuting the factors of $Y$, where each $f_i\colon X_i\to Y_i$ is a homothetic embedding.
\end{theorem}
\begin{proof}
    By \Cref{thm:fully_branching_most_general}, $f$ sends top-dimensional flats of $X$ within uniformly finite Hausdorff distance from top-dimensional flats of $Y$. Hence $f$ induces a combinatorial embedding $\partial_T X \to \partial_T Y$. Its image is a top-dimensional subbuilding. By \cite[Theorem~3.1]{kleinerleeb:rigidity:invariant}, since $Y$ has no rank-one factor, this subbuilding is the Tits boundary of a totally geodesic product $Y'\subseteq Y$ of subsymmetric spaces or subbuildings. Thus $f$ may be viewed as a quasiisometry $X\to Y'$. The conclusion then follows from \cite[Theorem~1.1.2]{kleinerleeb:rigidity}.
\end{proof}

Compare with \cite[Theorem~1.2]{nguyen:quasiisometric}, which gives a product decomposition for quasiisometric embeddings between reducible symmetric spaces.

\bibliographystyle{alpha}
\footnotesize{\bibliography{biblio}}
\Addresses
\end{document}